\documentclass[12pt]{amsart}
\usepackage{hyperref}
\usepackage{amsmath}
\usepackage{amsthm}
\usepackage{amsfonts}
\usepackage{amssymb}
\usepackage{graphicx}
\DeclareGraphicsExtensions{.eps}

\newtheorem{theorem}{Theorem}[section]
\newtheorem{utv*}{Proposition}
\newtheorem{hyp*}{Conjecture}
\newtheorem{lemma}[theorem]{Lemma}

\newtheorem*{th*}{Theorem}

\newcommand{\sha}[0]{\ensuremath{\mathbb{S}
}}
\def\la{\lambda}

\def\a{\alpha}
\def\R{\mathbb{R}}

\def\s{\sigma}

\def\Om{\Omega}

\def\C{\mathbb{C}}
\def\D{\mathbb{D}}
\def\cz{Calder\'on--Zygmund}
\def\eps{\varepsilon}
\def\f{\varphi}
\def\pd{\partial}
\def\al{\alpha}
\def\la{\langle}
\def\ra{\rangle}
\def\bP{\mathbb{P}}
\def\bE{\mathbb{E}}
\def\cF{\mathcal{F}}
\def\cL{\mathcal{L}}
\def\diff{\begin{bmatrix}dx\\dy\end{bmatrix}}

\input epsf.sty

\begin{document}

\title{Bellman function technique in Harmonic Analysis}
\author{Alexander Volberg}
\address{Department of Mathematics, Michigan State University, East Lansing, MI 48824, USA}
\email{sashavolberg@yahoo.com}
\date{}
\maketitle
\begin{abstract}
It is strange but  fruitful to think about the functions as random processes. Any function can be viewed as a martingale (in many different ways) with discrete time. But
it can be useful to have continuous time too. Processes can emulate functions, expectation of profit functional on the solution of stochastic differential equation can emulate
the functional on usual familiar functions. The advantage is that now we have ``all" admissible functions ``enumerated" as solutions of one stochastic differential equation, and choosing the best function optimizing a given functional becomes a problem of choosing the right control process. But such problem has been long since solved in the part of mathematics called Stochastic Optimal Control.
So-called Bellman equation reduces an infinite dimensional problem of choosing the best control to a finite dimensional (but non-linear as a rule) PDE called Bellman equation. Its solution, called Bellman function of a given optimization problem, gives us a lot of information about optimum and optimizers. This method gave some interesting results in the classical Harmonic Analysis, having on the surface nothing to do with probability. Sometimes the results obtained by this method did not find ``classical" proofs so far. It is especially well-suited to estimates of singular integrals, probably because of the underlying
probabilistic structure of classical singular integrals.
\end{abstract}

\section{Quasiconformal maps: sharp distortion estimates and sharp regularity}

We deal first with Beltrami equation 
\begin{equation}
\label{Beltr1}
f_{\bar z} -\mu f_z = 0\,,
\end{equation}
with bounded function $\mu$ called {\it Beltrami coefficient}, for simplicity $\mu$ is compactly supported on $\C$, $f$ being analytic near $\infty$ (see \eqref{Beltr1}) supposed to have the following Laurent decomposition at infinity
\begin{equation}
\label{Laur}
f(z)= z+ c_0 + \frac{c_{-1}}{z}+\dots\,.
\end{equation}

If $\mu$ is smooth it is not difficult to see that the solution is smooth on the whole $\C$. But we are interested in just measurable bounded $\mu$:

\begin{equation}
\label{Beltr2}
\|\mu\|_{L^{\infty}(\C)} =k<1\,.
\end{equation}

Several natural questions appear:

1. What is the smoothness of $f$ depending on $\mu, k$?

2. What are distortion properties of $f$? How it distorts the area and other measures?

3. In what classes (Sobolev, say) we can solve \eqref{Beltr1} in such a way that it will be continuous $\hat{\C}\rightarrow \hat{\C}$, where $\hat{\C}= \C\cup \infty$?
As at infinity it is a perfect holomorphic map, this question concerns only finite part of $\C$, so it is local.

\bigskip

Denote $g:=f_{\bar z}$. It is a function with compact support. We can restore  $f$:
\begin{equation}
\label{Cauchy}
f(z)= \frac1{\pi}\int_{\C} \frac{1}{\zeta-z} g(\zeta)\, dm_2(\zeta)  +c_0 +z\,.
\end{equation}
We used \eqref{Laur} and we naturally assume integrability of $g$.

Then obviously
$$
f_z=\frac1{\pi} \int_{\C} \frac{1}{(\zeta-z)^2} g(\zeta)\, dm_2(\zeta) +1 \,,
$$
where the integral is singular and should be understood , e.g. in the sense of principal values.
This is an important operator called the Ahlfors--Beurling transform (AB transform):
$$
Tg := \frac1{\pi} \int_{\C} \frac{1}{(\zeta-z)^2} g(\zeta)\, dm_2(\zeta) \,.
$$
Then \eqref{Beltr1} automatically becomes

\begin{equation}
\label{Beltr3}
g -\mu Tg = (I-\mu\,T)g = h\,,
\end{equation}
where $h=\mu$ is bounded with compact support. So in particular $h\in \cap_{p\ge 1} L^p(\C)$.

It is easy to make Fourier analysis of convolution kernel $\frac{\pi}{z^2}$ of AB operator, and to see that it is the Fourier multiplier with symbol
$
\zeta/\bar{\zeta}\,.
$
Therefore, $\|T\|_{L^2(\C)}=1$ and having then  $\|\mu\,T\|_{L^2(\C)}\le k<1$, we can conclude that \eqref{Beltr3} has a solution in $L^2(\C)$ given by the usual Neumann  series:
\begin{equation}
\label{Neu}
g= h + \mu T h + \mu T\mu T h +\dots\,,.
\end{equation}

Notice that $g$ is compactly supported (as $\mu$ is). Restore $f$ by \eqref{Cauchy}. The boundedness of $T$ in $L^2$ implies now that
\begin{equation}
\label{W2}
f\in W^{2}_{1, loc}(\C)\,.
\end{equation}
Let us see now that $g$ given in \eqref{Neu} is actually better than in $L^2$. Operator $T$ has norm $1$ in $L^2$ and it has norm {\it close} to $1$ in $L^p$, $p>2, p\approx 2$.
In fact, it is an operator with \cz \, kernel, and as such it is bounded in all $L^p$. Interpolating between, say $L^2$ and $L^4$, we get
\begin{equation}
\label{normAB}
\|T\|_{L^p(\C)} = : n(p) \rightarrow 1, \, p\rightarrow 2\,.
\end{equation}

So we can find such a $p=p(k)=2+\eps(k), \eps(k)>0,$ that the series in \eqref{Neu} converges in this $L^{2+\eps(k)}$. So $g\in L^{2+\eps(k)}$. Again restore $f$ by formula \eqref{Cauchy} (it is the same $f$ of course), again use that $f_z = Tf_{\bar z} +1= Tg +1$, and that $T$ is bounded in all $L^p$ being a \cz\,  operator. We got that $f$ self-improves from \eqref{W2} to

\begin{equation}
\label{Wp}
f\in W^{p}_{1, loc}(\C)\,, \,\, p=2+\eps(k), \eps(k)>0\,.
\end{equation}
We formulate this small fact as a fundamental Ahlfors--Bers--Bojarski's theorem:
\begin{theorem}
\label{ABB}
Any solution of \eqref{Beltr1}  in $W^2_{1, loc}$ self-improves to being in $W^{2+\eps(k)}_{1,loc}$, $\eps(k)>0$. In particular, any such solution is continuous on $\C$ and even H\"older continuous.
There exists a solution, which is a homeomorphism of $\hat{\C}$ into itself.
\end{theorem}

New questions appear:

4. What is the largest $2+\eps(k)$?

5. What is $n(p)$ in \eqref{normAB}?

To this we want to add some more questions. Introduce a constant
$$
K= \frac{1+k}{1-k}\in [1, \infty);\,\, k=\frac{K-1}{K+1}\,.
$$
It has a geometric meaning: it gives the maximal ratio of the axis of infinitesimal ellipses obtained as the images of infinitesimal circles by all possible solutions of \eqref{Beltr1}.

\noindent{\bf Definition.} Any solution of \eqref{Beltr1} from Theorem \ref{ABB} is called K-quasiregular map. Any homeomorphic solution is called K-quasiconformal map or  K-quasiconformal homeomorphism.  It is basically unique because by normalization at infinity it can be only shifted.

Again questions:

6. What is the sharp distortion of K-quasiconformal maps? Namely, if $f(\D)=\D$, where $\D$ denotes the unit disc and $f(0)=0$, then what is the {\it sharp} (largest) exponent in 
\begin{equation}
\label{dist1}
\forall E\subset \D\,,\,\,|f(E)| \le C_K |E|^{e(K)}\,?
\end{equation}

Without normalizations, allowing $f$ to be any K-quasiconformal map this becomes the question what is the best (largest) exponent in
\begin{equation}
\label{dist2}
\forall E\subset B\,,\,\,\frac{|f(E)|}{|f(B)|} \le C_K\bigg(\frac{ |E|}{|B|}\bigg)^{e(K)}\,?
\end{equation}

Function
$$
f_0(z) := z|z|^{\frac1K-1} \,, |z|\le 1,\,\,\text{and}\,\, = z\,,\text{for}\,\, |z|>1
$$
shows that $e(K) \le \frac1K$. 

\bigskip

\noindent{\bf Gehring's problem:} $e(K)=\frac1K$. It is equivalent to saying (we will se that) in Question 4 the sharp exponent of Sobolev integrability  is $2+\eps(k)=1+\frac1k -$.
This is very tough, but it was done by Astala \cite{A94}.

\bigskip

Glance now at \eqref{Neu}:  it gives that if we want to show that the exponent $p$ of Sobolev integrability  goes up to $1+\frac1k -$, it is enough to prove that
$$
\|T\|_{L^{1+1/k}} = 1/k\,,
$$
 in other words that
\begin{equation}
\label{pminus1}
\|T\|_{L^p} = p-1, p>2\,.
\end{equation}
This is very open, we will show how Bellman function gives partial results.

\bigskip

\noindent{\bf Big Iwaniec's problem or $p-1$-problem:} $n(p) = \max (p-1, \frac{p}{p-1}-1)$. 

\bigskip

Yet another question naturally arises: in Theorem \ref{ABB} we started with a priori solution in $W^2_{1, loc}$. How much {\it below} this we can start to have the same self-improvement?

7. Let $f\in W^q_{1, loc}$ solves \eqref{Beltr1}, and $q\in (1,2)$. What is the smallest $q=q(k)$ such that we still have for each such $f$ self-improvement to $W^2_{1, loc}$? (And then automatically to $W^{2+\eps(k)}_{1, loc}$  by Theorem \ref{ABB}, and then up to $W^{1+1/k -}_{1, loc}$ by Astala's \cite{A94}?)

\bigskip

\noindent{\bf Iwaniec's problem:} $q(K)=1+k$. 

\bigskip

We will prove it here using Bellman function technique and Astala's sharp distortion result \cite{A94}. We follow the exposition of \cite{AIS01}.

\subsection{Invertibility of Beltrami operator}
\label{inv}

If Big Iwaniec's problem were solved than we would immediately get

\begin{equation}
\label{inv1}
\text{If}\,\, p\in [2, 1+1/k), \, \, \text{then} \,\, \|(I-\mu\,T)^{-1}\|_{L^p} \le \frac{C(k)}{1+\frac1k-p}\,.
\end{equation}
(Actually even with $C(k)=1/k$, but this we do not care about as $k$ is fixed and we vary $p$.)

By duality and small talk one would get

\begin{equation}
\label{inv2}
\text{If}\,\, p\in I_k:=(1+k, 1+1/k), \, \, \text{then} \,\, \|(I-\mu\,T)^{-1}\|_{L^p} \le \frac{C(k)}{dist (p, \R\setminus I_k)}\,.
\end{equation}

Big Iwaniec's conjecture is still a conjecture, but this is  a Theorem of Petermichl--Volberg \cite{PV}, which we start to prove now. It will use Bellman function technique. Notice that now there exists an even more precise version of this result, namely, see \cite{AIPS}.

\begin{theorem}
\label{PV}
$\text{If}\,\, p\in I_k:=(1+k, 1+1/k), \, \, \text{then} \,\, \|(I-\mu\,T)^{-1}\|_{L^p} \le \frac{C(k)}{dist (p, \R\setminus I_k)}\,.$
\end{theorem}

Let $f$ be  a K-quasiconformal homeomorphism, let $p\in [2, 1+1/k)$, where $k=\frac{K-1}{K+1}$. Denote $J_f = |f_z|^2 -|f_{\bar z}|^2$, the Jacobian of $f$.
We need lemma:

\begin{lemma}
\label{a2norm}
Let $f,\,p$ be as above, denote $w:= J_f^{1-p/2}$.
Then $w\in A_2$ with
$$
[w]_{A_2} \le \frac{p^2C(K)}{1+\frac1k -p}\,.
$$
\end{lemma}

\noindent{\bf Remark.}
We will give the proof following \cite{AIS01}. There is another very interesting proof in \cite{AIPS}.

\begin{proof}
Notice first that  
\begin{equation}
\label{sravn}
(1-k^2) |f_z|^2\le J_f =|f_z|^2-|f_{\bar z}|^2 \le |f_z|^2 \le |f_z|^2 + |f_{\bar z}|^2\,.
\end{equation}
That is all this quantities are comparable with $C=C(K)$. The next step is to show that there is $C(K)$ such that if $B\subset \C$ 
is a disc and if $f$ is a K-quasiconformal homeomorphism of $\C$, then
\begin{equation}
\label{above1}
\frac1{|B|}\int_B (|f_z|+|f_{\bar z}|) ^p \le \frac{pC(K)}{1+\frac1k - p} \bigg(\frac{|f(B)|}{|B|}\bigg)^{\frac{p}2}\,.
\end{equation}

Using linear maps to pre-compose and to post-compose with $f$ we reduce it to normalized case $|f(B)|=|B|=1$. Apply
\eqref{dist2} proved by Astala in \cite{A94} to the set
$$
E_t=\{z\in B: |f_z|^2 +|f_{\bar z}|^2 \ge t\}\,,\, t>0
$$ 
we get
$$
|E_t| \le \frac1t \int_{E_t} ( |f_z|^2 +|f_{\bar z}|^2) dm_2 \le K \frac1t \int_{E_t} ( |f_z|^2 -|f_{\bar z}|^2) dm_2 =
$$
$$
K\frac1t |f(E_t)| \le C_1(K) \frac1t |E_t|^{\frac1K}\,.
$$
Therefore,
$$
|E_t|\le \min ( 1, C_2(K) \frac1{t^{\frac{K}{K-1}}} )\,.
$$
This is the same as
$$
|\{z\in B:  |f_z| +|f_{\bar z}| \ge t\}\le \min ( 1, C_3(K) \frac1{t^{\frac{2K}{K-1}}} )\,.
$$
Distribution function calculation now shows

\begin{equation}
\label{above2}
\int_B (|f_z|+|f_{\bar z}|) ^p \le C' + C''p \int_1^{\infty} \frac{t^{p-1}}{t^{\frac{2K}{K-1}} }= C' +C''p\int_1^{\infty} \frac{1}{t^{2+\frac1k-p } }\,dt \le C' +C''p\frac{1}{1+\frac1k-p}\,,
\end{equation}
as $\frac{2K}{K-1}= 1+\frac1k$. This proves \eqref{above1}.

\bigskip

Now we are ready to prove Lemma \ref{a2norm}. Notice that $w= J_f^{1-p/2} = (J_{f^{-1}}\circ f )^{p/2-1}$. Then
$$
\frac1{|B|}\int w\,dm_2 = \frac1{|B|}\int_B (J_{f^{-1}}\circ f )^{p/2-1}(z)\, dm_2(z) = \frac1{|B|}\int_{f(B)} J_{f^{-1}}^{p/2}(\zeta)\frac{J_{f^{-1}}(\zeta)}{J_{f^{-1}}(\zeta)} \, dm_2(\zeta)\,,
$$
where we made the change of variable $z= f^{-1}(\zeta)$. We continue
$$
\frac1{|B|}\int_B w\,dm_2 =\frac{|f(B)|}{|B|}\frac1{|f(B)|}\int_{f(B)} J_{f^{-1}}^{p/2}(\zeta) \, dm_2(\zeta) \le \frac{pC(K)}{1+\frac1k -p} \bigg(\frac{|B|}{|f(B)|}\bigg)^{p/2} \frac{|f(B)|}{|B|}\,.
$$
So
\begin{equation}
\label{w}
\frac1{|B|}\int_B w\,dm_2 \le \frac{pC(K)}{1+\frac1k -p} \bigg(\frac{|B|}{|f(B)|}\bigg)^{p/2-1} \,.
\end{equation}
We used here \eqref{above1} with K-{\bf quasidisc} $f(B)$ instead of a disc $B$. But this does not matter as any $K-quasidisc$ (:= the image of a disc by K-quasiconformal map) is an almost disc with constants depending only on $K$.

Now notice that we assumed $p\ge 2$, so if $p_n:=p-2$ we can write
$$
\frac1{|B|}\int_B w^{-1}\,dm_2 = \frac1{|B|}\int_B (J_{f} )^{p/2-1}(z)\, dm_2(z) = \frac1{|B|}\int_{B} J_{f}^{p_n/2}(z)\, dm_2(z)\,,
$$
and we use again \eqref{above1} with $p_n$ replacing $p$, gives us:
\begin{equation}
\label{wminus1}
\frac1{|B|}\int_B w^{-1}\,dm_2 \le \frac{p_nC(K)}{1+\frac1k -p_n} \bigg(\frac{|f(B)|}{|B|}\bigg)^{p_n/2} =\frac{\max (C', p-2)C(K)}{3+\frac1k -p} \bigg(\frac{|f(B)|}{|B|}\bigg)^{p/2-1} \,.
\end{equation}
\end{proof}
Multiplying \eqref{w} and \eqref{wminus1} we get Lemma \ref{a2norm}:
\begin{equation}
\label{a2}
w= J_f ^{1-\frac{p}2}, \, p\in [2, 1+\frac1k) \Rightarrow [w]_{A_2} \le \frac{p^2 C(K)}{1+\frac1k -p}\,.
\end{equation}

Now we can reap a first consequence:

\begin{theorem}
\label{invth}
Suppose  $\|\mu\|_{\infty}=k <1$, let $p\in I_k:=(1+k, 1+\frac1k)$. Then operators $I-\mu T, I-T\mu$ are boundedly invertible in $L^p$.
\end{theorem}

\begin{proof}
We can work with $I-\mu T$ as $I-T\mu = T(I-\mu T) T^{-1}$ and $T$ is boundedly invertible in each $L^q, 1< q< \infty$ ($T^{-1}$ is again a Fourier multiplier of \cz\, \,type).

Suppose we know how to prove the estimate from below
\begin{equation}
\label{frbel}
\|(I-\mu T)g\|_p \ge c(p,k) \|g\|_p, \forall g\in L^p(\C), \, p\in I_k\,.
\end{equation}

Then we would now exactly the same for $I-\mu T$, and, so, for the adjoint operator $(I-\mu\,T)^*$. Then $I-\mu T$ would have dense images in all $L^p$ we consider. Joining this with the estimate from below \eqref{frbel} we would conclude that $I-\mu T$ are invertible in in all $L^p$, $p\in I_k$.

\bigskip

So it is enough to have \eqref{frbel}. And we would like a good estimate for $c(K,p)$ in it.

\medskip

It is enough to prove \eqref{frbel} for the dense set of functions
$$
g\in C_0^{\infty}(\C), \, \, \int_{\C} g\, dm_2 =0\,.
$$
Let $\phi$ be the Cauchy transform of $g$: $\phi=\frac1{\pi}\int\frac{g(\zeta)}{\zeta-z}\,dm_2(\zeta)$.

Denoting $h:= g-\mu Tg $ we come to equation
$$
\phi_{\bar z}-\mu \phi_z =h\,,\text{in which we want to estimate}\,\, \|\phi_{\bar z}\|_p \le  C(K,p)\|h\|_p\,\, \text{if}\,\,p\in I_k\,.
$$

By Theorem \ref{ABB} there is a $K$-qc homeomorphism $f$
satisfying $f_{\bar z}-\mu f_z=0$.  Set
$$
u=\phi\circ f^{-1}\,,
$$
and let us see how equation $\phi_{\bar z}-\mu \phi_z =h$ will be transformed by this change of variable.

We calculate
$$
\phi_{\bar z }= (u_z\circ f) f_{\bar z} + (u_{\bar z}\circ f) \bar f_z\,,
$$
$$
\phi_{ z }= (u_z\circ f) f_{ z} + (u_{\bar z}\circ f) \bar f_{\bar z}\,,
$$
$$
\phi_{\bar z} -\mu \phi_{ z }-h= (u_z\circ f) f_{\bar z} + (u_{\bar z}\circ f) \bar f_z -\mu ((u_z\circ f) f_{ z} + (u_{\bar z}\circ f) \bar f_{\bar z})-h=
$$
$$
 (u_{\bar z}\circ f) \bar f_z -\mu  (u_{\bar z}\circ f) \bar \mu \bar f_z -h = (1-|\mu|^2) (u_{\bar z}\circ f)\,  \bar f_z-h =0\,.
 $$
 Hence obviously
 $$
 \int |u_{\bar z}\circ f)|^p| f_z|^p \le C(K)\int |h|^p\Rightarrow  \int |u_{\bar z}\circ f)|^p\, J_f^{p/2-1}J_f \le C(K)\int |h|^p
 $$
 And changing variable we get
 \begin{equation}
 \label{ubarz}
 \int |u_{\bar z}|^p| (J_{f^{-1}})^{1-p/2}= \int |u_{\bar z}|^p| (J_f\circ f^{-1})^{p/2-1}\le C(K)\int |h|^p
 \end{equation}
On the other hand, 
  $$
 \int |u_{ z}\circ f)|^p| f_{\bar z}|^p \le \frac{k^2}{1-k^2}  \int |u_{ z}\circ f)|^p| J_f^{p/2-1}J_f = \frac{k^2}{1-k^2}\int |u_z|^p (J_{f^{-1}})^{1-p/2}
 $$
Denote by $W:= (J_{f^{-1}})^{1-p/2}$. It is the one in Lemma \ref{a2norm}, only $f$ replaced by $f^{-1}$, which is a $K$-qc homeomorphism as well.

But the last expression above can be written as
$$
\int |u_z|^p (J_{f^{-1}})^{1-p/2}=\int |T(u_{\bar z})|^p (J_{f^{-1}})^{1-p/2} = \int |T(u_{\bar z})|^p  W=: TU
$$
(in fact, $u$ can be restored from $u_{\bar z}$ by Cauchy integral with no addition because  $u$ vanishes at infinity; then $u_z=T(u_{\bar z}$).
Notice that \eqref{ubarz} says that
\begin{equation}
\label{ubarzW}
 \int |u_{\bar z}|^p W\le C(K) |h|^p\,.
 \end{equation}
 
Combine \eqref{ubarzW} with (here $F$ is some ``unknown" function on $[1,\infty)$, but finite for all finite arguments)
$$
TU =   \int |T(u_{\bar z})|^p  W= F([w]_{A_2})  \int |u_{\bar z}|^p  W
$$
to get
\begin{equation}
\label{phibarz}
\|\phi_{\bar z}\|_p \le C(  \int |u_{\bar z}|^p W +  \int |T(u_{\bar z})|^p  W) \le  C(K)( 1+F([w]_{A_2}) \|h\|_p^p\,.
\end{equation}
 
\bigskip

Noticing that Lemma \ref{a2norm} gives the estimate $[w]_{A_2} \le \frac{p^2C(K)}{1+\frac1k -p}$, we conclude finally that
$$
\|g\|_p \le C(K) F( \frac{p^2C(K)}{1+\frac1k -p}) \|h\|_p\,,\,\text{if}\,\, p\in [2, 1+\frac1k)\,.
$$
We need the same estimate now for $1+k <p \le 2$. We need $W\in A_p$ (now $p\le 2$. But it is the same as to say that $W^{-1/(p-1)}\in A_{p'}, p'=p/(p-1)$. In our case  $W:= (J_{f^{-1}})^{1-p/2}$, so $W^{-1/(p-1)}$ will be
$(J_{f^{-1}}) ^{\frac{p'}2-1}$, which is inverse to the one in Lemma \ref{a2norm}, so also in $A_2\subset A_{p'}$. We get for  the whole interval of $p$'s: $p\in I_k= (1+k, 1+\frac1k)$ also
\begin{equation}
\label{F}
\|g\|_p \le F( \max(\frac{p^2C(K)}{1+\frac1k -p}, \frac{p^2C(K)}{1+\frac1k -p'})) \|h\|_p=F( \frac{p^2C(K)}{dist (p, \R\setminus I_k)}) \|h\|_p\,.
\end{equation}

Theorem \ref{inv} is proved.

\end{proof}

In \cite{AIS01} the following conjecture was formulated that claims that function $F$ in \eqref{F} is just linear. Notice that this would in fact easily follow from Big Iwaniec' conjecture.

\vspace{.2in} 

\noindent{\bf Conjecture}
\begin{equation}
\label{conj}
\|g\|_p \le \frac{p^2C(K)}{dist (p, \R\setminus I_k)}) \|h\|_p\,,\text{equivalently} \,\,\|(I-\mu T)^{-1}\|_p \le \frac{p^2C(K)}{dist (p, \R\setminus I_k)}) \,.
\end{equation}
 
 We will prove now this conjecture using the Bellman function technique. But first let us derive the corollary of the conjecture. As always $\|\mu\|_{\infty}=k<1$.
 
\begin{theorem}[Corollary of the conjecture]
\label{wqr}
Any solution of
$$
F_{\bar z} -\mu  F_z=0\,,
$$
which is in $W^{1+k}_{1, loc}$ is automatically in $W^{2}_{1, loc}$, and so satisfies Theorem \ref{ABB}. It automatically self-improves then (by Astala's \cite{A94}) to be in $W^{1+\frac1k -}_{1, loc}$.
\end{theorem}

First use Conjecture to prove

\begin{lemma} [Behavior at the end points of interval $I_k$]
\label{inj}
Operators $I-\mu T, I-T\mu$ have dense range in $L^{1+\frac1k}$ and, correspondingly, trivial kernels on $L^{1+k}$.
\end{lemma}

\begin{proof}
By $T(I-\mu T)T^{-1}=I-T\mu$ and invertibility of $T$ in all spaces $L^p, 1<p<\infty$, it is enough to prove just the dense range of $I-\mu T$ in $L^{1+\frac1k}$.
Consider $\eps>0$ and equation
\begin{equation}
\label{eq1}
\phi_{\eps} -(1-\eps)\mu T\phi_{\eps} =h
\end{equation} 
for nice $h\in C_0^{\infty}$. We want to consider the solution for $p_0=1+\frac1k$. We consider this $p_0$ in $I_{(1-\eps)k}$ because $\|(1-\eps)\mu\|_{\infty}=(1-\eps)k$. Point $p_0$ is obviously $C(K)\eps$ close to the right end point of $I_{(1-\eps)k}$. 

Hence, applying conjecture we conclude that
$$
\|\phi_{\eps}\|_{p_0} \le \frac{C(K)}{\eps}\|h\|_{p_0}\,.
$$

Notice two things: 1) In $L^2$ the norma of $\phi_{\eps}$ are uniformly bounded by $C(K)$ just by using Neumann series in $L^2$ in  \eqref{eq1}; 2) in $L^{p_0}$ the norms of $ T(\eps\phi_{\eps})$ are uniformly bounded. It is immediate to conclude from 1) and 2) that
$$
\eps\mu\,T\phi_{\eps}\,\,\text{converges weakly to zero in}\,\, L^{p_0}\,.
$$

Rewrite our equation \eqref{eq1} as follows:
$$
\phi_{\eps} -\mu T\phi_{\eps} =h-\eps \mu T\phi_{\eps}\,.
$$
The right hand side weakly in $L^{p_0}$ converges to any function $h$, whose family is strongly dense in $L^{p_0}$. So the right hand side is weakly dense in $L^{p_0}=L^{1+\frac1k}$. But it is in $Range (I-\mu T)$, so this range is weakly dense in $L^{1+\frac1k}$. Being a linear set this range is then strongly dense in $L^{1+\frac1k}$. Lemma \ref{inj} is proved

\end{proof}

\begin{proof}[The proof of Theorem  \ref{wqr}]

Consider $R_{\bar z}-\mu R_{ z} =0, R\in W^{1+k}_{1, loc}$. Choose $\phi\in C_0^{\infty}$. Set $G=\phi R$. Then
$$
G_{\bar z} - G_z = (\phi_{\bar z} -\mu\phi_z) R\,.
$$
Looking at this formula we can start to think that the support of $\mu$ is compact (is contained in the support of $\phi$).

As $G$ vanishes at infinity it is the Cauchy transform of its $\bar\partial G= G_{\bar z}$, and therefore, $G_z= T G_{\bar z}$.
We can rewrite the equation
$$
(I-\mu T) \psi = h\,;\,\, \psi:= G_{\bar z}\,,\,\, h=(\phi_{\bar z} -\mu\phi_z) R\in L^{\frac{2(1+k)}{1-k}} \subset L^2(\C)\cap L^{2+\eps}(\C)\,.
$$
The inclusion above for function $R$ is by Sobolev imbedding, in fact, we assumed that $R\in W^{1+k}_{1, loc}$, and by the compactness of the support of $R$.
It has been already remarked, that in the last two equation we have the right to think that $\mu=0$ outside of the support of $\phi$.
Let us consider the convergent in $L^2(\C)$ of the series of compactly supported functions:
$$
\psi_0= h +\mu Th +\mu T\mu T h+\dots\,.
$$
It solves our equation, it is in $L^2(\C)$ and it is compactly supported, hence it is in $L^{1+k}(\C)$. Therefore we got a solution $\psi_0$ of $(I-\mu T)\psi_0 =h$, which is in $L^{1+k}\cap L^2$.
But $\psi= G_{\bar z}$ is also in $L^{1+k}$. By Lemma \ref{inj} we have $\psi=\psi_0\in L^2$. It means that $R\in W^2_{1, loc}$. Theorem \ref{wqr} is proved.

\end{proof}

\noindent{\bf Example.} $f=\frac{|z|^{1-\frac1K}}{z}$ for $z\in \D$, $f=\frac1z$ outside $\D$. It is a solution of Beltrami equation with $\mu, \|\mu\|_{\infty}=k=\frac{K-1}{K+1}$, and it is in $W^{q}_{1, loc}$ for every $q<1+k$.  But it is NOT K- quasiregular mapping, it has a singularity at $0$.

\bigskip

\noindent This example shows how sharp is Theorem \ref{wqr}. Its proof hinges on Conjecture \ref{conj}.
We prove this Conjecture now using Bellman technique. First we analyze function $F$ from \eqref{F}. 
Recall that  $W:= (J_{f^{-1}})^{1-p/2}$, and if $p\in (1+k, 1+\frac1k)$ the estimate in \eqref{F} is $F_p([W]_{A_2})$, where $F_p$ is the best function one can have in the estimate
$$
\|T\|_{L^p(W)} \le F_p([W]_{A_p})\,.
$$

Suppose we can prove
\begin{theorem}
\label{sharp}
$F(x) \le C\, x^{\max (1, 1/(p-1))}\,.$
\end{theorem}

Then we recall that  for $1+k<p\le 2, \,p':=\frac{p}{p-1}$  we already estimated $[W]_{A_p}^{1/(p-1)}= [W^{-1/(p-1)}]_{A_{p'}} \le [W^{-1/(p-1)}]_{A_2}\le \frac{p^2C(K)}{p-1-k}$. And for $1+\frac1k>p\ge 2$, $[W]_{A_2}\le \frac{p^2C(K)}{1+\frac1k-p}$. These estimates were based on sharp distortion theorem of Astala \cite{A94}. We made these estimates in Lemma \ref{a2norm}. These estimates and Theorem \ref{sharp} then imply trivially conjecture \eqref{conj}. Hence this Theorem \ref{sharp} is the only ingredient left to be proved to have Theorem \ref{wqr}.

\section{Linear estimates of weighted Ahlfors--Beurling transform by Bellman function technique}
\label{BellmanAB}

Let $\omega$ be any weight on $\R^2$, denote its heat extension into $\R^3_+$ by $\omega(x,t)=\omega(x_1,x_2,t)$:
\[
\omega(x,t)= \frac1{\pi t}\int\int_{\R^2} \omega(y)\exp(-\frac{\|x-y\|^2}{t})dy_1dy_2 \,.
\] 

We define 
\[
[\omega]_{A_p}^{heat}:= \sup_{(x,t)\in \R^3_+} \omega(x,t)\left(\omega^{-\frac1{p-1}}(x,t)\right)^{p-1}\,.
\]

The weights $w$  with finite $[w]_{A_p}^{heat}$ are called $A_p$ weights. There is an extensive theory of $A_p$ weights, see for example \cite{St},\cite{GCRF}. The usual definition differs from the one above, but it describes the same class of weights. Actually, we will say more about the relationship between the classical definition and ours. But first we state two more theorems, whose combined use gives Theorem \ref{sharp} at least for $p\ge 2$. 

\bigskip

\noindent{\bf Remark.} The method called Rubio de Francia extrapolation--one can  see its exposition in \cite{CUMPbook}--actually shows that to have a full range of $p$'s  in Theorem \ref{sharp} it is enough to prove it only for $p=2$.

\medskip

\begin{theorem}
\label{t2}
For any $A_p$ weight $w$ and any $p\geq 2$
we have
\[
\|T\|_{L^p(wdA)\rightarrow L^p(wdA)}\leq C(p)( [w]_{A_p}^{heat})^{\frac1{p-1}}\,.
\]
\end{theorem}

We want to discuss the connection between $[w]_{A_p}^{heat}$ and 
$[w]_{A_p}^{class}$. Here $[w]_{A_p}^{class}$ denotes the following supremum over all discs in the plane:
\[
[w]_{A_p}^{class}:=\sup_{B(x,R)}\,\left(\frac1{|B(x,R)|}\int_{B(x,R)}\omega dA\right)\,\cdot
\left(\frac1{|B(x,R)|}\int_{B(x,R)}\omega^{-\frac1{p-1}} dA\right)^{p-1}\,.
\]

Obviously, there exists a positive absolute constant $a$ such that for any function $w$
\[
a\,[w]_{A_p}^{class} \leq [w]_{A_p}^{heat}\,.
\]

{\bf Remark}. The opposite inequality is easy to prove too. In fact, we have

\begin{theorem}
\label{t4}
There exists a finite absolute constant $b$ such that
\[
[w]_{A_p}^{heat} \leq b\,[w]_{A_p}^{class}\,.
\]
\end{theorem}

\begin{proof} Constants will be denoted by the letters $c,C$; they may vary from line to line and even within the same line.
We introduce the following notations. $B_k$ denotes $B(0,2^k),\, k=0,1,2,...$, $\langle f\rangle_B$ stands for the average $\frac1{|B|}\int_B f dA$, $f(B)$ stands for $\int_B f dA$. If $B=B(0,r)$, then $\langle f\rangle^h_B$ 
stands for $\frac1{\pi\,r^2}\iint_{\R^2} f(x)\exp(-\frac{\|x\|^2}{r^2}) dx_1 dx_2$.

\medskip

\begin{lemma}
\label{l1t4}
Suppose  $f$ and $g$, positive functions on the plane, are such that $\sup_B\,\langle f\rangle_B\langle g\rangle_B
=A$, then there exists a finite absolute constant $c$ such that 
\[
\langle f\rangle_B \langle g\rangle^h_B \leq cA
\] 
for any disc $B$.
\end{lemma}

\begin{proof}

Scale invariance allows us to prove this only for one disc $B=B(0,1)$.
We start the estimate:
\[
\langle f\rangle_B \langle g\rangle^h_B \leq c\,\langle f\rangle_B\Sigma_k 2^{2k}\exp(-2^{2k-2})\frac{A}{\langle f\rangle_{B_k}}\,.
\]
On the other hand $\langle f\rangle_{B_k}>c \langle f\rangle_{B_{k-1}}>\dots c^k\langle f\rangle_{B}$ (recall that $B$ is the unit disc). Plugging this in the inequality above, we get

\[
\langle f\rangle_B \langle g\rangle^h_B \leq c \langle f\rangle_B\Sigma_k C^k \exp(-2^{2k-2})
\frac{A}{\langle f\rangle_{B}}\,.
\]
In other words,
\[
\langle f\rangle_B \langle g\rangle^h_B \leq cA\Sigma_k C^k \exp(-2^{2k-2}) =cA 
\]
and the lemma is proved.
\end{proof}

Now we want to prove Theorem \ref{t4}. Fix $B$. Again by scale invariance it is enough to consider $B=B(0,1)$.
By the previous lemma, we know that
\begin{equation}
\label{fg}
\langle f\rangle_{B_k} \langle g\rangle^h_{B_k} \leq cA
\end{equation}
for any $k$.

Now

\[
\langle f\rangle^h_B \langle g\rangle^h_B \leq c\langle g\rangle^h_B\Sigma 2^{2k}\exp(-2^{2k-2})\langle f\rangle_{B_k} \leq c\langle g\rangle^h_B\Sigma 2^{2k}\exp(-2^{2k-2})\frac{cA}{\langle g\rangle^h_{B_k}}\,.
\]
The last inequality used \eqref{fg}.

On the other hand, $\langle g\rangle^h_{B_k}>c \langle g\rangle^h_{B_{k-1}}>\dots c^k \langle g\rangle^h_{B}$ (recall that $B$  is the unit disc). Plugging this in the inequality above, we get

\[
\langle f\rangle^h_B \langle g\rangle^h_B \leq c\langle g\rangle^h_B\Sigma_k C^k \exp(-2^{2k-2})
\frac{cA}{\langle g\rangle^h_{B}}\,.
\]

In other words,

\[
\langle f\rangle^h_B \langle g\rangle^h_B \leq c^2 A\Sigma_k C^k \exp(-2^{2k-2})=c^2 A\,.
\]

Theorem \ref{t4} is completely proved.

\end{proof}

\bigskip

The next  result proves Theorem \ref{sharp} for $p= 2$. We will show later how to extrapolate just from the result at $p=2$ to all possible $p$'s. 

\begin{theorem}
\label{t5}
For any $A_2$ weight $w$ 
we have
\[
\|T\|_{L^2(wdA)\rightarrow L^2(wdA)}\leq C\, [w]_{A_2}^{class}\,.
\]
\end{theorem}

\begin{proof}
There will be many steps.  But we are going to prove only Theorem \ref{t2} and only for $p=2$. By Theorem \ref{t4} and Rubio de Francia extrapolation this is enough.

\medskip

The operator $T$ is given in the Fourier domain $(\xi_1,\xi_2)$ by the multiplier $\frac{{\zeta}}{\bar\zeta}=
\frac{{\zeta}^2}{|\zeta|^2}= \frac{(\xi_1+i\xi_2)^2}{\xi_1^2 +\xi_2^2} = \frac{\xi_1^2}{\xi_1^2 +\xi_2^2}
-\frac{\xi_2^2}{\xi_1^2 +\xi_2^2}+2i\frac{\xi_1 \xi_2}{\xi_1^2 +\xi_2^2}$. Thus, $T$ can be written as $T=R_1^2 -R_2^2 +2i R_1 R_2$, where $R_1,R_2$ are Riesz transforms on the plane (see \cite{St} for their definition and properties). Another way of writing $T$ is
\[
T=m_1+im_2\,,
\]
where $m_1,m_2$ are Fourier multiplier operators.
Notice that the multipliers themselves (as functions, not as multiplier operators) are connected by
\[
m_2 =m_1\circ \rho\,,
\] 
where $\rho$ is $\pi/4$ rotation of the plane. So the multiplier {\it operators} are related by
\[
m_2 = U_{\rho}\,m_1\,U_{\rho}^{-1}\,,
\]
where $U_{\rho}$ is an operator of $\rho$-rotation in $(x_1,x_2)$ plane. But for any operator $K$ we have
\[
\|U_{\rho}\,K\,U_{\rho}^{-1}\|_{L^2(wdA)\rightarrow L^2(wdA)} = \|K\|_{L^2(w\circ\rho^{-1}dA)\rightarrow L^2(w\circ\rho^{-1}dA)}\,.
\]
Combining this with the fact that $Q_{w,2}^{heat}=Q_{w\circ\rho^{-1},2}^{heat}$ for any rotation, we conclude that we only need the desired estimate of Theorem \ref{t5} for $m_1=R_1^2-R_2^2$.
Actually, we will show that
\begin{equation}
\label{1.2}
\|R_i^2\|_{L^2(wdA)\rightarrow L^2(wdA)} \leq C\, Q_{w,2}^{heat},\,\,\,i=1,2\,.
\end{equation}

\bigskip

To prove \eqref{1.2} we fix, say, $R_1^2$ and two test functions $\f,\psi\in C_0^{\infty}$.
We will be using heat extensions. For $f$ on the plane, its heat extension is given by the formula
\[
f(y,t):=\frac1{\pi\,t}\iint_{\R^2} f(x)\exp(-\frac{|x-y|^2}{t}) \,dx_1 dx_2,\,\,\, (y,t)\in \R_+^3\,.
\]
We usually use the same letter to denote a function and its heat extension.

\begin{lemma}
\label{l1.1}
Let $\f,\psi\in C_0^{\infty}$ . Then the integral $\iiint \frac{\pd\f}{\pd x_1}\cdot \frac{\pd\psi}{\pd x_1}\, dx_1dx_2dt$ converges absolutely and
\begin{equation}
\label{1.13}
\iint R_1^2\f\cdot\psi\, dx_1dx_2 = - 2\iiint \frac{\pd\f}{\pd x_1}\cdot \frac{\pd\psi}{\pd x_1}\, dx_1dx_2dt \,.
\end{equation}
\end{lemma}

\begin{proof}
The proof of this lemma is actually trivial. It is based on the  well-known fact that a function is an integral of its derivative, and also involves Parseval's formula.
Consider $\f,\psi\in C_0^{\infty}$ and  now
\[
\iint \psi R^2_1 \f dx_1 dx_2= \iint\frac{\xi_1^2}{\xi_1^2 +\xi_2^2}\hat{\f}(\xi_1,\xi_2)\hat{\psi}(-\xi_1,-\xi_2) d\xi_1 d\xi_2 =
\]
\[
2 \iint\int_0^{\infty}e^{-2t(\xi_1^2 +\xi_2^2)}\xi_1^2\hat{\f}(\xi_1,\xi_2)\hat{\psi}(\xi_1,\xi_2) d\xi_1 d\xi_2 dt = 
\]
\[
-2\int_0^{\infty} \iint i\xi_1\hat{\f}(\xi_1,\xi_2)e^{-t(\xi_1^2 +\xi_2^2)}\cdot i\xi_1\hat{\psi}(-\xi_1,-\xi_2)e^{-t(\xi_1^2 +\xi_2^2)} d\xi_1 d\xi_2 dt =
\]
\[
-2\int_0^{\infty} \iint \frac{\pd \f}{\pd x_1}(x_1, x_2,t)\frac{\pd \psi}{\pd x_1}(x_1,x_2,t) dx_1dx_2 dt =
\]
\[
-2\iiint_{\R^3_+} \frac{\pd \f}{\pd x_1}(x_1, x_2,t)\frac{\pd \psi}{\pd x_1}(x_1,x_2,t) dx_1dx_2 dt\,.
\]

Above we used Parseval's formula twice, and also we  used the absolute convergence of the integrals 
$$
\iiint_{\R^3_+} e^{-2t(\xi_1^2 +\xi_2^2)}\xi_1^2\hat{\f}(\xi_1,\xi_2)\hat{\psi}(\xi_1,\xi_2) d\xi_1 d\xi_2 dt\,,
$$ 
$$
\iiint_{\R^3_+} \frac{\pd \f}{\pd x_1}(x_1, x_2,t)\frac{\pd \psi}{\pd x_1}(x_1,x_2,t) dx_1dx_2 dt\,.
$$
For the first integral this is obvious. The absolute convergence of the second integral can be easily proved. We leave this as an exercise for the reader .

\end{proof}

\bigskip

Our next goal is to estimate the right side of \eqref{1.13} from above.

\begin{theorem}
\label{1.2}
For any $\f,\psi\in C_0^{\infty}$, and any positive function $w$ on the plane we have
\[
\iiint_{\R^3_+}\left|\frac{\pd\f}{\pd x_1}\right|
\left|\frac{\pd\psi}{\pd x_1}\right|\,dx_1\, dx_2 \,dt\leq A\,Q_{w,2}^{heat}\left(\iint |\f|^2 w \,dx_1\,dx_2 + \iint|\psi|^2\frac{1}{w} \,dx_1 \,dx_2\right)
\]
where $A$ is an absolute constant.
\end{theorem}

\vspace{.1in}

\begin{center}{\bf Bellman function}\end{center}
In the proof we at last use {\bf a Bellman function} tailored for this problem. It is $B$ from the following theorem.  The meaning of $Q$ in the next theorem is $Q:= [w]_{A_2}^{heat}$. 

We use the notation $H_f$ for the Hessian matrix of function $f:\R^k\rightarrow \R$ (the matrix of second derivatives of $f$), and $d^2f$ for the second differential form, which  is the quadratic form $(H_f(x) dx, dx)$, where $(\cdot, \cdot)$  is the usual scalar product in $\R^k$,  $x$ is a point in the domain of definition of $f$, $dx$ is an arbitrary vector in $\R^k$.

\begin{theorem}
\label{1.3}
For any $Q>1$  define the domain  $D_Q:=\{0 < (X,Y,x,y,r,s): \,x^2 < Xs,\, y^2 < Yr,\, 1< rs < Q\}$. Let $K$ be any compact subset of $D_Q$. Then there exists a function $B=B_{Q,K}(X,Y,x,y,r,s)$ infinitely differentiable in a small neighborhood of $K$ such that
\[
1)\,0\leq B\leq 5Q(X+Y)\,,
\]
\[
2)\,-d^2B\geq|dx||dy|\,.
\]
\end{theorem}

We prove Theorem \ref{1.3}  later.
Now we use it to obtain the proof of Theorem \ref{1.2}.

\begin{proof}

Given a non-constant smooth $w$ that is constant outside some large ball, we consider $Q=Q_{w,2}^{heat}$. We treat only the case $w\in A_2$, that is $Q<\infty$, for otherwise there is nothing to prove. Consider two nonnegative functions $\f,\psi\in C_0^{\infty}$. Now  take $B=B_{Q,K}$, where a compact $K$ remains to be chosen.

We are interested in

\[
b(x,t) :=  B((\f^2w)(x,t), (\psi^2w^{-1})(x,t), \f(x,t),\psi(x,t),w(x,t), w^{-1}(x,t))\,.
\]

This is a well defined function, because the choice of $Q$ ensures that the $6$-vector $v$,  consisting of heat extensions of corresponding functions on $\R^2$,
\[
v:= ((\f^2w)(x,t), (\psi^2w^{-1})(x,t), \f(x,t),\psi(x,t),w(x,t), w^{-1}(x,t))
\]
lies in $D_Q$ for any $(x,t) \in \R^3_+$. Also we can fix any compact subset $M$ of the open set $\R^3_+$
and guarantee that for $(x,t)\in M$, the vector $v$ lies in a  compact $K$. In fact, notice that for  our $w$ and  for compactly supported $\f,\psi$ the mapping $(x,t)\rightarrow v(x,t)$ maps compacts in $\R^3_+$ to compacts in $D_Q$. Now just take $K$ large enough.  

The main object we want to study is 

\begin{equation}
\label{svsn}
\left(\frac{\pd}{\pd t} -\Delta\right) b(x,t)\,.
\end{equation}
For simplicity we assume that $B$ is already $C^2$ up to the boundary of $D_Q$. The technical details what to do without this assumption are left to the audience, see \cite{PV}.
We want to estimate the expression in \eqref{svsn} 1) from above {\it in average} and 2) from below in a pointwise way.

\medskip

1) Take a ``slab" $S_{\eps, H}:= \{(x,t)\in \R^3_+: \eps\le t\le H\}$. Notice that for any fixed positive $t$
$$
\int_{\R^2}\Delta b (x,t) dx =0\,.
$$
This is because we assumed $B$ to be smooth and because $v(x,t)\rightarrow 0$ for a fixed $t$ when $x\rightarrow \infty$ rather fast, and the same is true for $\nabla v(x,t)$.
Hence,
$$
\int_{S_{\eps, H}} \left(\frac{\pd}{\pd t} -\Delta\right) b(x,t)\,dxdt = \int_{S_{\eps, H}} \frac{\pd}{\pd t}  b(x,t)\,dxdt = \int_{\R^2} b(x, H)\, dx -\int_{\R^2} b(x, \eps)\, dx\,.
$$
Now we recall that $b=B\circ v$, that $B\ge 0$ (so we can throw away a ``minus" term above), and that $B(X,Y,\dots) \le 5Q (X+Y)$. Then we get (functions below are heat extensions of the corresponding symbols on $\R^2$):
$$
\int_{S_{\eps, H}} \left(\frac{\pd}{\pd t} -\Delta\right) b(x,t)\,dxdt \le
$$
\begin{equation}
\label{sv1}
 5 Q\int_{\R^2}( \f^2w(x, H) +\psi^2 w^{-1}(x, H))\, dx = 5Q \int_{\R^2}[ (\f^2w)(x) +(\psi^2 w^{-1})(x)]\, dx\,.
\end{equation}

2) Now we make a pointwise estimate of \eqref{svsn} from below. The next calculation is simple but it is  key to the proof. In it  as everywhere 
$$
v=((\f^2w)(x,t), (\psi^2w^{-1})(x,t), \f(x,t),\psi(x,t),w(x,t), w^{-1}(x,t))\,.
$$
\begin{lemma}
\label{1.10}
\[
\left(\frac{\pd}{\pd t} -\Delta\right) b(x,t)=\left(\left(-d^2B\right)\frac{\pd v}{\pd x_1},\frac{\pd v}{\pd x_1}\right)_{\R^6} + \left(\left(-d^2B\right)\frac{\pd v}{\pd x_2},\frac{\pd v}{\pd x_2}\right)_{\R^6}\,.
\]
\end{lemma}

\begin{proof}
\[
\frac{\pd}{\pd t} b = (\nabla B, \frac{\pd v}{\pd t})_{\R^6}\,,
\]
\[
\Delta b =\left((d^2B) \frac{\pd v}{\pd x_1}, \frac{\pd v}{\pd x_1}\right)_{\R^6} + \left((d^2B) \frac{\pd v}{\pd x_2}, \frac{\pd v}{\pd x_2}\right)_{\R^6} + (\nabla B, \Delta v)_{\R^6}\,.
\]
We just used the chain rule. Now

\[
\left(\frac{\pd}{\pd t} -\Delta\right) b= \left(\nabla B, (\frac{\pd v}{\pd t}-\Delta v)\right)_{\R^6}-\left((d^2B) \frac{\pd v}{\pd x_1}, \frac{\pd v}{\pd x_1}\right)_{\R^6} - \left((d^2B) \frac{\pd v}{\pd x_2}, \frac{\pd v}{\pd x_2}\right)_{\R^6}\,.
\]
However, the first term is zero because all entries of the vector $v$ are solutions of the heat equation.
\end{proof}

\bigskip

By Theorem \ref{1.3} 
\[
-d^2B\geq|dx||dy|\,.
\]
Therefore, for $(x,t)$ Lemma \ref{1.10} gives:

\begin{equation}
\label{1.35}
\left(\frac{\pd}{\pd t} -\Delta\right) b(x,t)\geq \left|\frac{\pd \f}{\pd x_1}\right|\left|\frac{\pd \psi}{\pd x_1}\right| + \left|\frac{\pd \f}{\pd x_2}\right|\left|\frac{\pd \psi}{\pd x_2}\right|\,.
\end{equation}

\vspace{.1in}

Combining \eqref{sv1} \eqref{1.35} we get
\begin{equation}
\label{1.40}
\iiint_{S_{\eps, H}} \left(\left|\frac{\pd \f}{\pd x_1}\right|\left|\frac{\pd \psi}{\pd x_1}\right| + \left|\frac{\pd \f}{\pd x_2}\right|\left|\frac{\pd \psi}{\pd x_2}\right|
\right) \leq 5Q\,\left( \iint \f^2w + \iint\psi^2 w^{-1}\right)\,.
\end{equation}

Theorem \ref{1.2} is completely proved by using {\bf a Bellman function} of our problem whose existence is claimed in Theorem \ref{1.3}

\end{proof}

Theorem \ref{t5} is proved.

\end{proof}

\noindent{\bf Remark.} The proof of Theorem \ref{1.2} is actually ``equivalent" to  solution of an obstacle problem for a  certain fully non-linear PDE. Consider 
$$
\sigma=\begin{bmatrix} 0,& 1\\1,& 0\end{bmatrix}\,.
$$
Then we need $H_B \pm \sigma\ge 0$ in each point in $D_Q$. As we a looking for the best possible $B$ satisfying these relationships, it is natural that
we should require
$$
\det (H_B +\sigma)=0\,\,\text{or}\,\, \det(H_B-\sigma) =0\,,
$$
these are Monge--Amp\`ere equations. {\bf This approach  is used in \cite{VaVo}, \cite{SSV}. Now we use another method to prove the existence of $B_Q$ required in Theorem \ref{1.3}.}

\subsection{More Bellman functions to prove the existence of Bellman function $B_Q$ from Theorem \ref{1.3}. Dyadic shifts.}
\label{proofexist}

We start with a much simpler ``model" operator---$T_{\sigma}$. The logic will be the following.
We want to get a sharp weighted estimate of $\|T_{\s}\|_{L^2(w)\rightarrow L^2(w)}$ via the $A_2$ characteristic of $w$. 
In the paper
of Nazarov, Treil, Volberg, see \cite{NTV99}, one can find that the norm $\|T_{\s}\|_{L^2(u)\rightarrow L^2(v)}$ is attained on some ``simple" test functions---and that this
holds for every pair $u,v$. Thus also for $u=v=w$. However, on the family $\mathcal{T}$ of test functions one can compute the
$N_{w,2}(T_{\s}):=\sup\{\|T_{\sigma}t\|_{L^2(w)}:\, t\in {\mathcal T},\, \|t\|_{L^2(w)}=1\}$. It turns out that 
\begin{theorem}
\label{Wi}
$N_{w,2}(T_{\s}) \approx
Q_{w,2}^{class}$.
\end{theorem}

 J. Wittwer does that in \cite{Wi} basing her approach on \cite{NTV99}; see also
\cite{WiP}. Thus, we get $\|T_{\s}\|_{L^2(w)\rightarrow L^2(w)}= N_{w,2}(T_{\s}) \approx Q_{w,2}^{class}$.

\bigskip

So let us show what the model operator is, what  its sharp weighted estimate is, and how one obtains a special function (Bellman function) from this estimate.

\bigskip

Consider the family of dyadic singular operators $T_{\s}$:
\[
T_{\s}f = \Sigma_{I\in {\mathcal D}} \s(I) \,(f, h_I)\, h_I\,.
\]
Here ${\mathcal D}$ is a dyadic lattice on $\R$, $h_I$ is a Haar function associated with the dyadic interval 
$I$ ($h_I$ is normalized in $L^2(\R,dx)$), and $\s(I) = \pm 1$. We call the family $T_{\s}$ {\it the martingale transform}. It is a
dyadic analog of a Calder\'on-Zygmund operator. Here are important questions about $T_{\s}$, the first one about two-weight estimates
and the second one about one weight estimates:

\medskip

1) What are necessary and sufficient conditions for $\sup_{\s} \|T_{\s}\|_{L^2(u)\rightarrow L^2(v)} <\infty$? 

\medskip
2) What is the sharp bound on $\sup_{\s} \|T_{\s}\|_{L^2(w)\rightarrow L^2(w)}$ in terms of $w$? 
How can one compute $\sup_{\s} \|T_{\s}\|_{L^2(w)\rightarrow L^2(w)}$? 

\bigskip

These questions are dyadic analogs of notoriously difficult questions about ``classical" Calder\'on-Zygmund operators like the Hilbert transform, the Riesz transforms and the Ahlfors-Beurling transform. The dyadic model is supposed to be easier than the continuous one. This turned out to be true. The answers to the questions above appeared in \cite{NTV99}, \cite{Wi}. Moreover these answers are key to answering questions about ``classical" Calder\'on-Zygmund operators.

Strangely enough, the answer to the second question (which seems to be easier, because it is about ``one weight") seems to require the ideas from the  ``two-weight" case.  Here is our explanation of this phenomena. The necessary and sufficient conditions on $(u,v)$ to answer  the first question were given in \cite{NTV99}. They amount to the fact that $\sup_{\s} \|T_{\s}\|_{L^2(u)\rightarrow L^2(v)}$ is almost attained on the family of simple test functions. This fact has beautiful consequences in the one weight case. For then $\sup_{\s} \|T_{\s}\|_{L^2(w)\rightarrow L^2(w)}$ is attainable (almost) on the family of simple test functions. One may try to compute $\sup_{\s} \|T_{\s} \,t\|_{L^2(w)}$ for every element of this test family, thus getting a good estimate for the norm $\sup_{\s} \|T_{\s}\|_{L^2(w)\rightarrow L^2(w)}$. Test functions are rather simple, so this program can be carried out. This has been done in Wittwer's paper \cite{Wi}. We will give another proof below.
Here is the result. Recall that 
\[
Q_{w,2}^{dyadic} := \sup_{I\in {\mathcal D}}\langle w\rangle_I \langle w^{-1}\rangle_I\,.
\]

\begin{theorem}
\label{tE}
\[
\sup_{\s} \|T_{\s}\|_{L^2(w)\rightarrow L^2(w)}\leq A \, Q_{w,2}^{dyadic}\,.
\]
\end{theorem}

\vspace{.1in}

\noindent{\bf Remark.} We will postpone the proof of Theorem \ref{tE} (another use of {\bf Bellman function technique}), here we will use it first to finish the proof of the 
existence of $B_Q$ claimed in Theorem \ref{1.2}.

\bigskip

So we assume now that Theorem \ref{tE} is already proved. Let us rewrite Theorem \ref{tE} as follows

\[
\sup_{\s(I)=\pm 1} \left|\Sigma_{I\in {\mathcal D}} \s(I) \,(f,h_I)\, (g,h_I)\right| \leq A
Q_{w,2}^{dyadic}\|f\|_{L^2(w)}\|g\|_{L^2(w^{-1})}\,,
\]

or

\[
\Sigma_{I\in {\mathcal D}} |(f,h_I)|\, |(g,h_I)| \leq A Q_{w,2}^{dyadic}\|f\|_{L^2(w)}\|g\|_{L^2(w^{-1})}\,.
\]

This inequality is scaleless, so we write it as 

\begin{equation}
\label{1.60}
J\in {\mathcal D},\,\, \frac1{4|J|}\Sigma_{I\in {\mathcal D},\,I\subset J} |\langle f\rangle_{I_-}-\langle f\rangle_{I_+}|\, |\langle
g\rangle_{I_-}-\langle g\rangle_{I_+}||I| \leq A Q_{w,2}^{dyadic}\langle f^2w\rangle_{J}^{1/2}\langle g^2/w\rangle_{J}^{1/2}\,.
\end{equation} 

Here $I_-,I_+$ are the left and the right halves of $I$, and $\langle\cdot\rangle_l$ means averaging over $l$ as usual. 
Given a fixed $J\in {\mathcal D}$ and a number $Q>1$, we wish to introduce the Bellman function of \eqref{1.60}:

\[
\aligned
B(X,Y,x,y,r,s) &= \sup\{\frac1{4|J|}\Sigma_{I\in {\mathcal D},\,I\subset J} |\langle f\rangle_{I_-}-\langle f\rangle_{I_+}|\, 
|\langle g\rangle_{I_-}-\langle g\rangle_{I_+}||I| :\\
\langle f\rangle_J&=x,\,\langle g\rangle_J = y,\,\langle w\rangle_J =r,\, \langle w^{-1}\rangle_J=s,\\
\langle f^2w\rangle_J&=X,\, \langle g^2/w\rangle_J=Y,\, w\in A_2^{dyadic},\, Q_{w,2}^{dyadic}\leq Q \}\,.
\endaligned
\]

Obviously, the function $B$ does not depend on $J$, but it does depend on $Q$. Its domain of definition is the following:
\[
R_Q := \{ 0\leq (X,Y,x,y,r,s), \, x^2 \leq Xs,\, y^2\leq Yr, 1\leq rs\leq Q\}\,.
\]

By \eqref{1.60} it satisfies 

\begin{equation}
\label{1.61}
0\leq B\leq AQ X^{1/2}Y^{1/2}\,.
\end{equation}

We are going to prove that it also satisfies the following ``differential" inequality. Denote $v:=(X,Y,x,y,r,s)$, $v_-=(X_-,Y_-,x_-,y_-,r_-,s_-)$, $v_+=(X_+,Y_+,x_+,y_+,r_+,s_+)$, let $v,v_+,v_-$ lie in $R_Q$,  and let $v=\frac12(v_{-}+v_+)$. Then

\begin{equation}
\label{1.62}
B(v)-\frac12\left(B(v_+) + B(v_-)\right) \geq |x_{+}  - x_-||y_{+}  - y_-|
\,.
\end{equation}

In fact, let $f,g,w$ almost maximize $B(v)$ (on the interval $J$), let $f_+,g_+,w_+$ do this for $B(v_+)$, $f_-,g_-,w_-$ do this for $B(v_-)$. The freedom of scale for $B$ allows us to put $f_+,g_+,w_+$  on $J_+$ and $f_-,g_-,w_-$ on $J_-$. Then we have  ``gargoyle" functions 
\[
F=
\begin{cases} f_+ & \text{on} \,J_+\\
f_- & \text{on} \, J_-
\endcases \qquad  G=
\cases g_+ & \text{on} \,J_+\\
g_- & \text{on}\, J_-
\endcases \qquad W =
\cases w_+ & \text{on} \,J_+\\
w_- & \text{on} \, J_-\,.
\end{cases} \qquad 
\]
Obviously, $\langle F\rangle_J=\frac12(x_+ +x_-) =x,\,\langle G\rangle_J = y,\,\langle W\rangle_J = r,\, \langle W^{-1}\rangle_J = s,\,\langle F^2W\rangle_J = X,\,\langle G^2W^{-1}\rangle_J=Y$. These numbers together form the vector $v$. In other words $F,G,W$ compete with $f,g,w$ in the definition \eqref{1.60} of Bellman function $B(v)$. By this definition,
\[
B(v) \geq \frac1{4|J|}\Sigma_{I\in {\mathcal D},\,I\subset J} |\langle F\rangle_{I_-}-\langle F\rangle_{I_+}|\, |\langle
G\rangle_{I_-}-\langle G\rangle_{I_+}||I|\,.
\]

But the almost optimality of $f_+,g_+,w_+$ on $J_+$ and $f_-,g_-,w_-$ on $J_-$ gives us (recall that $F=f_\pm$ on 
$J_\pm$, $G=g_\pm$ on $J_\pm$):
\[
B(v_+)\leq \varepsilon + \frac1{4|J_+|}\Sigma_{I\in {\mathcal D},\,I\subset J_+} |\langle F\rangle_{I_-}-\langle F\rangle_{I_+}|\,
|\langle G\rangle_{I_-}-\langle G\rangle_{I_+}||I|\,,
\]
and

\[
B(v_-)\leq \varepsilon + \frac1{4|J_-|}\Sigma_{I\in {\mathcal D},\,I\subset J_-} |\langle F\rangle_{I_-}-\langle F\rangle_{I_+}|\,
|\langle G\rangle_{I_-}-\langle G\rangle_{I_+}||I|\,.
\]

Combining these, we get

\[
\aligned
B(v)&-\frac12(B(v_+)+B(v_-)) \geq -2\varepsilon + \frac14|\langle F\rangle_{J_-}-\langle F\rangle_{J_+}|\, |\langle G\rangle_{J_-}-\langle G\rangle_{J_+}\\
= -2\varepsilon +& \frac14|\langle f_-\rangle_{J_-}-\langle f_+\rangle_{J_+}|\, |\langle g_-\rangle_{J_-}-\langle g_+\rangle_{J_+} = -2\varepsilon +\frac14|x_--x_+||y_--y_+|\,.
\endaligned
\]

We are done with \eqref{1.62} because $\varepsilon$ is an arbitrary positive number. Therefore, our $B$ is a very concave function. We are going to modify $B$  to have its Hessian satisfy the conclusion of Theorem \ref{1.3}. To do that we fix a compact $K$ in the interior of $R_Q$, and we choose $\varepsilon$ such that $100 \varepsilon < dist(K, \pd R_Q)$. Consider the convolution of $B$ with $\frac1{\varepsilon^6}\f(\frac{v}{\varepsilon}), \, v\in\R^6$, where $\f$ is a bell shape infinitely differentiable function with support in the unit ball of $\R^6$. It is now very easy to see that this convolution (we call it $B_{K,Q}$) satisfies the following inequalities

\begin{equation}
\label{1.63}
0\leq B_{K,Q} \leq 6Q(X+Y)\,,
\end{equation}
and
for any vector $\xi=(\xi_1,\xi_2,\xi_3,\xi_4,\xi_5,\xi_6)\in\R^6$,
\begin{equation}
\label{1.64}
-(d^2B_{K,Q}\xi,\xi)_{\R^6} \geq 2|\xi_2||\xi_3|\,.
\end{equation}

\medskip

\noindent The factor $2$ appears because $B(v)-\frac12(B(v_+)+B(v_-))$ in \eqref{1.62} corresponds to $-\frac12 d^2B$, and $|x-x_+|=\frac12 |x_--x_+|$
(the same being valid with $y$'s replacing $x$'s and $-$ replacing $+$).

\medskip

\noindent Theorem \ref{1.3} is completely proved modulo the proof of Theorem \ref{tE}.

\vspace{.2in}

\begin{proof}[Proof of Theorem \ref{tE}]

To prove Theorem \ref{tE} we need the  following decomposition:
\begin{lemma}
\label{decomp}
\begin{equation}
\label{hI}
h_I = \al_I h_I^w + \beta_I \frac{\chi_I}{\sqrt{I}}\,,
\end{equation}
where

1)  $|\alpha_I| \le \sqrt{\langle w\rangle_I}$,

2)$ |\beta_I| \le \frac{|\Delta_I w|}{ \langle w\rangle_{I}}$, where $\Delta_I w:= \langle w\rangle_{I_+}-\langle w\rangle_{I_+}$,

3) $\{h_I^w\}_{I} $ is an orthonormal basis in $L^2(w)$,

4) $h_I^w$ assumes on $I$ two constant values, one on $I_+$ and another on $I_-$.
\end{lemma}
\begin{proof}
To find $\al, \beta$ we first apply $\|\cdot\|_{L^2(w)}^2$ to both parts of \eqref{hI}: $\langle w\rangle _I=\|h_I\|_{L^2(w)}^2 = \al^2 +\beta^2 \langle w\rangle_I$, and secondly we multiply 
\eqref{hI} by $\chi_I/\sqrt{|I|}$ and integrate with respect to $w\,dx$: $\frac12(\langle w\rangle_{I_+}-\langle w\rangle_{I_+}) = \beta_I\langle w\rangle_I$. Clearly Lemma is proved.

\end{proof}

Now let 
$$
\sha F:=\sum_I c_I(f, h_I)\,h_I\,, \,\,\text{where constants}\,\, c_I\,\,\text{are such that}\,\, |c_I|\le 1\,.
$$

Let $\sigma:= w^{-1}$ for the rest of the proof. Fix $\phi\in L^2(w), \psi\in L^2(\sigma)$. We need to prove
\begin{equation}
\label{main}
|(\sha\,\,\phi w,\psi\sigma)|\le C\, \|\phi\|_{w}\|\psi\|_{\sigma}\,.
\end{equation}

\vspace{.2in}

We  estimate $(\sha\,\,\phi w, \psi \sigma)$ as

$$
|\sum_{I} c_I(\phi w, h_I)(\psi \sigma,h_I)|\le 
$$
$$
\sum_{I} |c_I (\phi w, h^w_I)\sqrt{\langle w \rangle_I}(\psi \sigma,h^{\sigma}_I)|\sqrt{\langle \sigma \rangle_I}|\,+
$$
$$
\sum_I |c_I \langle \phi w\rangle_I\frac{\Delta_I w}{\langle w \rangle_I}(\psi \sigma,h^{\sigma}_I)\sqrt{\langle \sigma \rangle_I}\sqrt{I}|\,+
$$
$$
\sum_I |c_I \langle \psi \sigma\rangle_J\frac{\Delta_I\sigma}{\langle \sigma \rangle_I}(\phi w,h^{w}_I)\sqrt{\langle w \rangle_I}\sqrt{I}|\,+
$$
$$
\sum_I|c_I \langle \phi w\rangle_I\langle \psi \sigma\rangle_J \frac{\Delta_I w}{\langle w \rangle_I} \frac{\Delta_I\sigma}{\langle \sigma \rangle_I}\sqrt{I}\sqrt{I}|=: I + II +III +IV\,.
$$

So we have
$$
I \le  \sum_I (\phi w, h^w_I)\sqrt{\langle w \rangle_I}  \cdot (\psi \sigma, h^\sigma_I)\sqrt{\langle \sigma \rangle_I}
,\,
II \le \sum_I(\phi w, h^w_I)\sqrt{\langle w \rangle_I} \cdot \langle \psi \sigma\rangle_I  \frac{|\Delta_I \sigma|}{\langle \sigma \rangle_I}\sqrt{|I|},\,\,
\,
$$
$$
III\le \sum_I\langle \phi w\rangle_I \frac{|\Delta_I w|}{\langle w \rangle_I}\sqrt{|I|} \cdot (\psi \sigma, h^\sigma_I)\sqrt{\langle \sigma \rangle_I}
,\,\,\,
IV\le  \sum_I \langle \phi w\rangle_I \frac{|\Delta_I w|}{\langle w \rangle_I}\sqrt{|I|} \cdot  \langle \psi \sigma\rangle_I  \frac{|\Delta_I \sigma|}{\langle \sigma \rangle_I}\sqrt{|I|}\,.
$$

The estimate of $I$ is trivial because $h^w_I$, $h^\sigma_I$ are orthonormal systems in $L^2(w), L^2(\sigma)$ correspondingly:
\begin{equation}
\label{Iest}
I \le \sup_I\sqrt{\langle w\rangle_I\langle \sigma\rangle_I}\sqrt{\sum_I (\phi w, h^w_I)^2}\sqrt{\sum_I (\psi \sigma, h^\sigma_I)^2} \le [w]_{A_2}^{1/2} \|\phi\|_w\|\psi\|_{\sigma}\,.
\end{equation}

\bigskip

To estimate the rest let us fix $\al\in (0, 1/2)$ and introduce
\begin{equation}
\label{caral}
\mu_I:= \langle w\rangle_I^\al \langle \sigma\rangle_I^\al\bigg(\frac{|\Delta_I w|^2}{\langle w\rangle_I^2} + \frac{|\Delta_I \sigma|^2}{\langle \sigma\rangle_I^2} \bigg)|I|\,.
\end{equation}

We are going to give {\bf a Bellman  function} proof of the following lemma.

\bigskip

\begin{lemma}
\label{caralL}
The sequence $\{\mu_I\}_{I\in D}$ is a Carleson sequence with Carleson constant at most $C\, [w]_{A_2}^{\al}$.
\end{lemma}

\bigskip

We take Lemma \ref{caralL} for granted till the end of the proof of Theorem \ref{tE}. First introduce a notation, let $\mu$ be a positive measure on $\R$, then
$$
M^d_{\mu}f (x) :=\sup_{I\in D, x\in I}\frac{1}{\mu (I)}\int_I |f|\,d\mu\,.
$$
This is called dyadic weighted maximal function. We will use it with $\mu =wdx$ or $\sigma dx$.

To estimate $IV$,  $II$, and symmetric to it $III$ we notice that
$$
 \frac{|\Delta_I \sigma|}{\langle \sigma \rangle_I}\sqrt{|I|} \le \langle w\rangle^{-\al/2}\langle \sigma\rangle^{-\al/2} \sqrt{\mu_I}\,,
$$
so, choosing $p\in (1,2)$
$$
\langle \psi \sigma\rangle_I  \frac{|\Delta_I \sigma|}{\langle \sigma \rangle_I}\sqrt{|I|} \le  \langle w\rangle^{-\al/2}\langle \sigma\rangle^{-\al/2}(\langle |\psi|^p \sigma\rangle_I)^{1/p} \langle \sigma\rangle^{1-1/p} \sqrt{\mu_I}\le
$$
$$
 \langle w\rangle^{-\al/2}\langle \sigma\rangle_I^{1-\al/2}\inf_{x\in I}(M^d_\sigma |\psi|^p (x))^{1/p}\cdot \sqrt{\mu_I}\,,
$$
where $M^d_\sigma$ is the dyadic weighted maximal function. Therefore,
$$
IV \le \sum_I \langle w\rangle^{1-\al}\langle \sigma\rangle_I^{1-\al}\inf_{ I}(M^d_\sigma |\psi|^p)^{1/p} \cdot \inf_I (M^d_w |\phi|^p)^{1/p} \cdot \mu_I\,.
$$
$$
II \le \sum_I (\phi w, h_I^w) \langle w\rangle_I^{1-\al/2}\langle \sigma\rangle_I^{1-\al/2}\frac{\inf_{ I} (M^d_\sigma |\psi|^p)^{1/p}}{\langle w\rangle_I^{1/2}}\cdot \sqrt{\mu_I}\,.
$$
The estimate of $III$ will be totally symmetric, so we omit it.
We continue:

$$
IV \le [w]_{A_2}^{1-\al}\sum_I \inf_{ I}(M^d_\sigma |\psi|^p)^{1/p} \cdot \inf_I (M^d_w |\phi|^p)^{1/p} \cdot \mu_I\,.
$$
$$
II \le [w]_{A_2}^{1-\al/2}\sqrt{\sum_I (\phi w, h_I^w)^2} \sqrt{\sum_I\frac{\inf_{ I} (M^d_\sigma |\psi|^p)^{2/p}}{\langle w\rangle_I}\cdot \mu_I}\,.
$$

Choose $F=(M^d_\sigma |\psi|^p)^{1/p} \cdot (M^d_w |\phi|^p)^{1/p}$ and $G=(M^d_\sigma |\psi|^p)^{2/p}$ and apply the following simple lemma ({\bf Exercise}!)

\begin{lemma}
\label{carl1}
Let $\{\al_L\}_{L\in D}$ define Carleson measure with intensity $B$.  Let $F$ be a positive function on the line. Then
\begin{equation}
\label{1}
\sum_L (\inf_L F)\, \al_L \le 2 B\int_{\R} F\,dx\,.
\end{equation}
\begin{equation}
\label{2}
\sum_L\frac{ \inf_L G}{\langle w\rangle_L} \al_L \le C\,B\int_{\R}\frac{G}{w} dx\,.
\end{equation}
\end{lemma}

Then using Lemma \ref{caralL} we get
$$
IV \le [w]_{A_2}^{1-\al}[w]_{A_2}^{\al} \int_{\R}(M^d_\sigma |\psi|^p)^{1/p} \cdot  (M^d_w |\phi|^p)^{1/p}\, dx =
$$
$$
[w]_{A_2} \int_{\R}(M^d_\sigma |\psi|^p)^{1/p} \cdot (M^d_w |\phi|^p)^{1/p}\, w^{1/2} \sigma^{1/2}dx \le
$$
$$
[w]_{A_2}\sqrt{\int (M^d_w |\phi|^p)^{2/p}\, wdx}\sqrt{\int (M^d_\sigma |\phi|^p)^{2/p}\, \sigma dx} \le C[w]_{A_2}\|\phi\|_w\|\psi\|_{\sigma}\,.
$$
As to $II$, we have again using Lemma \ref{carl1} (the second part) and Lemma \ref{caralL}:
$$
II \le  C\,[w]_{A_2}^{1-\al/2}[w]_{A_2}^{\al/2} \sqrt{\sum_I(\phi w, h_I^w)^2}\sqrt{\int_{\R}\frac{ (M^d_\sigma |\psi|^p)^{2/p}(x)}{w(x)}\,dx}\le
$$
$$
C\,[w]_{A_2} \|\phi\|_w \sqrt{\int_{\R} (M^d_\sigma |\psi|^p)^{2/p}(x)\,\sigma(x)dx} \le C\,[w]_{A_2} \|\phi\|_w\|\psi\|_{\sigma}\,.
$$

Theorem \ref{tE} is completely proved, function $B_Q$ is constructed.

\end{proof}

\vspace{.1in}

We need only to see the validity of Lemma \ref{caralL}. This is done by yet another {\bf Bellman function}.

\begin{proof}[Bellman proof of Lemma \ref{caralL}] We prove even a more general statement, namely we prove the version in $\R^d$ and even in each metric space with {\bf geometric doubling condition} and doubling measure $\mu$. So let us have a metric space with geometric doubling condition, meaning that every ball of radius $r$ can fit only at most $K$ disjoint balls of radius $r/2$, $K$ being independent of the ball and its radius. Such metric spaces carry a doubling measure $\mu$ by a theorem of Konyagin--Volberg \cite{KV}, and let $D$ denote  the family of ``dyadic cubes" on this metric space (constructions are numerous, the first belongs to M. Christ \cite{Ch}), and let $s_i(I)$ are dyadic children of $I\in D$. Finally, let $I\in D$ and let
$$
\mu_I:= ( \langle w\rangle_{\mu, I} \langle \sigma\rangle_{\mu, I})^{\al}\bigg(\frac{(\langle w\rangle_{\mu,s_i( I)}-\langle w\rangle_{\mu, I})^2}{\langle w\rangle_{\mu, I}^2} +\frac{(\langle \sigma\rangle_{\mu,s_i( I)}-\langle \sigma\rangle_{\mu, I})^2}{\langle \sigma\rangle_{\mu, I}^2} \bigg)\mu(I)\,.
$$
Lemma \ref{caralL} becomes the following statement, which we are proving below:
\begin{equation}
\label{caralLmu}
\forall I\in D\,\,\sum_{J\in D, J\subset I} \mu_J \le C_{\al} [w]_{\mu,A_2}^{\al} \mu(I)\,.
\end{equation}

Let $Q> 1, 0<\alpha<\frac12$. In domain  $\Omega_Q:=\{(x,y): x>0, y>0, 1<xy\le Q\}$ function $b_Q(x,y):=x^{\al}y^{\al}$ satisfies the following estimate of its Hessian matrix  (of its second differential form, actually)
$$
-d^2 b_Q(x,y)\ge \al(1-2\al)x^{\al}y^{\al}\bigg(\frac{(dx)^2}{x^2} +\frac{(dy)^2}{y^2}\bigg)\,.
$$
The form $-d^2 b_Q(x,y)\ge 0$ everywhere in $x>0, y>0$. Also obviously $0\le b_Q(x,y) \le Q^{\al}$ in $\Omega_Q$.
\begin{proof}
Direct calculation.
\end{proof}

\medskip

Fix now a  cube $I$ and let $s_i(I), i=1,...,M$, be all its sons. Let $a=(\langle w\rangle_{\mu, I}, \langle \sigma\rangle_{\mu, I})$, $b_i=(\langle w\rangle_{\mu,s_i( I)}, \langle \sigma\rangle_{\mu, s_i(I)})$, $i=1,\dots, M$, be points--obviously--in $\Omega_Q$, where $Q$ temporarily means $[w]_{A_2}$. Consider $c_i(t)=a(1-t)+ b_it, 0\le t \le 1$ and $q_i(t):= b_Q(c_i(t))$.  We want to use Taylor's formula
\begin{equation}
\label{Lagr}
q_i(0)-q_i(1) = -q'_i(0) - \int_0^1dx\int_0^x q_i''(t)\,dt\,.
\end{equation}
Notice two things: Sublemma  shows that $-q_i''(t) \ge 0$ always. Moreover, it shows that if $t\in [0,1/2]$,  then we have that the following qualitative estimate holds:
\begin{equation}
\label{wI}
-q_i''(t) \ge c\,( \langle w\rangle_{\mu, I} \langle \sigma\rangle_{\mu, I})^{\al}\bigg(\frac{(\langle w\rangle_{\mu,s_i( I)}-\langle w\rangle_{\mu, I})^2}{\langle w\rangle_{\mu, I}^2} +\frac{(\langle \sigma\rangle_{\mu,s_i( I)}-\langle \sigma\rangle_{\mu, I})^2}{\langle \sigma\rangle_{\mu, I}^2} \bigg)
\end{equation}
This requires a small explanation. If we are on the segment $[a, b_i]$, then the first coordinate of such a point cannot be larger than $C\, \langle w\rangle_{\mu, I}$, where $C$ depends only on doubling of $\mu$ (not $w$). This is obvious. The same is true for the second coordinate with the obvious change of $w$ to $\sigma$. But there is no such type of estimate from below on this segment:  the first coordinate cannot be smaller than $k\, \langle w\rangle_{\mu, I}$, but $k$ may (and will) depend on the doubling of $w$ (so ultimately on its $[w]_{A_2}$ norm). In fact, at the ``right" endpoint of $[a, b_i]$ the first coordinate is $\langle w\rangle_{\mu, s_i(I)}\le \int_I\,w\,d\mu/ \mu(s_i(I)) \le C\,  \int_I\,w\,d\mu/ \mu(I)=C\, \langle w\rangle_{\mu, I}$, with $C$ only depending on the doubling of $\mu$. But the estimate from below will involve the doubling of $w$, which we must avoid. But if $t\in [0,1/2]$, and we are on the ``left half" of interval $[a, b_i]$ then obviously the first coordinate is $\ge \frac12 \langle w\rangle_{\mu, I}$ and the second coordinate is $\ge \frac12 \langle \sigma\rangle_{\mu, I}$.

We do not need to integrate $-q_i''(t)$ for all $t\in [0,1]$ in \eqref{Lagr}. We can only use integration over $[0,1/2]$  noticing that $-q_i''(t)\ge 0$ otherwise. Then the chain rule 
$$
q_i''(t)=(b_Q(c_i(t))''=(d^2b_Q(c_i(t)) (b_i-a), b_i-a)\,,
$$ 
(where $(\cdot, \cdot)$ means the usual scalar product in $\R^2$) immediately gives us \eqref{wI} with constant $c$ depending on the doubling of $\mu$ but {\it independent} of the doubling of $w$.

\medskip

Next step is to add all \eqref{Lagr}, with convex coefficients $\frac{\mu(s_i(I))}{\mu(I)}$, and to notice that $\sum_{i=1}^M\frac{\mu(s_i(I))}{\mu(I)} q_i'(0) =\nabla b_{Q}(a) \sum_{i=1}^M\cdot (a-b_i)\frac{\mu(s_i(I))}{\mu(I)}=0$, because by definition
$$
a= \sum_{i=1}^M \frac{\mu(s_i(I))}{\mu(I)}\,b_i\,.
$$
Notice that the addition of all \eqref{Lagr}, with convex coefficients $\frac{\mu(s_i(I))}{\mu(I)}$ gives us now (we take into account \eqref{wI} and positivity of $-q_i''(t)$)
$$
b_Q(a)- \sum_{i=1}^M \frac{\mu(s_i(I))}{\mu(I)}\,b_Q(b_i) \ge 
$$
$$
c\,c_1\,( \langle w\rangle_{\mu, I} \langle \sigma\rangle_{\mu, I})^{\al}\sum_{i=1}^M\bigg(\frac{(\langle w\rangle_{\mu,s_i( I)}-\langle w\rangle_{\mu, I})^2}{\langle w\rangle_{\mu, I}^2} +\frac{(\langle \sigma\rangle_{\mu,s_i( I)}-\langle \sigma\rangle_{\mu, I})^2}{\langle \sigma\rangle_{\mu, I}^2} \bigg)\,.
$$
We used here the doubling of $\mu$ again, by noticing that $\frac{\mu(s_i(I))}{\mu(I)}\ge c_1$ (recall that $s_i(I)$ and $I$ are almost balls of comparable radii).
We rewrite the previous inequality using our definition of $\Delta_I w, \Delta_I\sigma$ listed above as follows
$$
\mu(I)\,b_Q(a)- \sum_{i=1}^M \mu(s_i(I))\,b_Q(b_i) \ge c\,c_1\,( \langle w\rangle_{\mu, I} \langle \sigma\rangle_{\mu, I})^{\al}\bigg(\frac{(\Delta_I w)^2}{\langle w\rangle_{\mu, I}^2} +\frac{(\Delta_I\sigma)^2}{\langle \sigma\rangle_{\mu, I}^2} \bigg)\mu(I)\,.
$$
Notice that $b_Q(a)=\langle w\rangle_{\mu, I}^{\al}\langle\sigma\rangle_{\mu,I}^{\al}$. Now we iterate the above inequality and get for any of  dyadic $I$'s:
$$
\sum_{J\subset I\,, J\in\,D} ( \langle w\rangle_{\mu, J} \langle \sigma\rangle_{\mu, J})^{\al}\bigg(\frac{(\Delta_J w)^2}{\langle w\rangle_{\mu, J}^2} +\frac{(\Delta_J\sigma)^2}{\langle \sigma\rangle_{\mu, J}^2} \bigg)\mu(J) \le C\, Q^{\al}\mu(I)\,.
$$
This is exactly the Carleson property of the measure $\{\mu_I\}$ indicated in our Lemma \ref{caralL}, with Carleson constant $C\,Q^{\al}$. The proof showed that $C$ depended only on $\al\in (0, 1/2)$ and on the doubling constant of measure $\mu$. Lemma \ref{caralL} is completely proved.

\end{proof}

\section{Estimates for Ahlfors--Beurling operator. Towards the Big Iwaniec problem by Bellman footsteps}
\label{AB}

In the previous section we estimated AB operator $T$ in weighted $L^2(w)$. The estimate was sharp in $[w]_{A_2}$:
\begin{equation}
\label{w2}
|(Tf,g)|\le C\, [w]_{A_2} \|f\|_{L^2(w)}\|g\|_{ L^2(w^{-1})}\,,
\end{equation}
it implied a sharp in $[w]_{A_p}$ estimate in weighted $L^p(w)$:
\begin{equation}
\label{wp}
|(Tf,g)|\le C(p)\, [w]_{A_p}^{\max (1, \frac1{p-1})} \|f\|_{L^p(w)}\|g\|_{ L^{p'}(w^{-1/(p-1)})}\,,\,\, p':=p/(p-1)\,.
\end{equation}

But we did not care about $C, C(p)$ at all. Now we consider just $w=1$, but we care about $C(p)$ very much. Big Iwaniec's problem conjectures
\begin{equation}
\label{BI}
C(p)= \max (p, p/(p-1))-1=: p^*-1\,.
\end{equation}

This is {\bf open} at the moment of writing this phrase. Using various {\bf Bellman functions} we will show the row of improvements
\begin{equation}
\label{BI1}
C(p)\le 2(p^*-1)\,.
\end{equation}
\begin{equation}
\label{BI2}
C(p)\le 1.7(p^*-1)\,.
\end{equation}
\begin{equation}
\label{BI3}
C(p)\le 1.575(p^*-1)\,.
\end{equation}
\begin{equation}
\label{BI4}
C(p)\le 1.4(p^*-1)\,.
\end{equation}

Recall things that we already know: 1) $T= R_1^2-R^2_2 + 2i R_1R_2$, where $R_i$ are Riesz transforms = multipliers with symbol $\xi_i/(|\xi_1|^2 +|\xi_2|^2)^{1/2}$, $i=1,2$;

\begin{equation}
\label{22}
2) \,\,(R_i^2 f, \bar g) = -2 \iint_{\R^3_+} \frac{\pd f}{\pd x_i} \frac{\pd g}{\pd x_i} \,dx dt\,,
\end{equation}
where $f, g$ in the left hand side are from $C_0^{\infty}(\R^2)$, and $f, g$ in the right hand side are heat extensions of functions in the left.

Hence Conjecture \ref{BI} is nothing else as the following innocent looking {\bf conjecture}

\begin{equation}
\label{BIh}
2\bigg|\iint_{\R_+^3} \bigg(\frac{\pd f}{\pd x_1}+ i \frac{\pd f}{\pd x_2}\bigg)\cdot \bigg(\frac{\pd g}{\pd x_1}+ i \frac{\pd g}{\pd x_2}\bigg)\,dx dt\bigg| \le (p-1) \|f\|_p\|g\|_{p'}\,,\,\, p>2\,.
\end{equation}

Let complex-valued functions $f=u+iv, g= \phi+i\psi$. Consider $f:=(u,v)$ as a map $\R^2\rightarrow \R^2$, do the same with $G:=(\phi, \psi)$. We have Jacobian matrices $DF, DG$ then. These are $2\times 2 $ matrices.

Imagine that we want to have a stronger estimate than \eqref{BIh} (which is probably {\bf too much!}):

\begin{equation}
\label{BIhm}
2\iint_{\R_+^3} \bigg|\frac{\pd f}{\pd x_1}+ i \frac{\pd f}{\pd x_2}\bigg|\cdot \bigg|\frac{\pd g}{\pd x_1}+ i \frac{\pd g}{\pd x_2}\bigg|\,dx dt \le (p-1) \|f\|_p\|g\|_{p'}\,,\,\, p>2\,.
\end{equation}

This is exactly

\begin{equation}
\label{BIhm1}
2\iint_{\R_+^3} (|DF|_2^2-2\det DF)^{1/2}(|DG|_2^2-2\det DG)^{1/2}\,dx dt \le (p-1) \|f\|_p\|g\|_{p'}\,,\,\, p>2\,,
\end{equation}
where $|\cdot|_2$ is the Hilbert-Schmidt norm of the matrix.

Nobody can prove \eqref{BIhm1} and equivalent to it \eqref{BIhm}. They may be wrong!

\bigskip

However, we will start with proving slightly lighter estimates:

\begin{equation}
\label{pr}
2\iint_{\R_+^3} \bigg|\frac{\pd f}{\pd x_1} \bigg|\bigg|\frac{\pd g}{\pd x_1}\bigg|+\bigg|\frac{\pd f}{\pd x_1}\bigg| \bigg|\frac{\pd g}{\pd x_1}\bigg|\,dx dt \le (p-1) \|f\|_p\|g\|_{p'}\,,\,\, p>2\,.
\end{equation}

Moreover, we will prove a stronger than \eqref{pr} (but weaker than \eqref{BIhm}) estimate 

\begin{equation}
\label{sq}
2\iint_{\R_+^3} \bigg(\bigg|\frac{\pd f}{\pd x_1} \bigg|^2 +\bigg|\frac{\pd f}{\pd x_2}\bigg|^2\bigg)^{1/2} \bigg(\bigg|\frac{\pd g}{\pd x_1} \bigg|^2 +\bigg|\frac{\pd g}{\pd x_2}\bigg|^2\bigg)^{1/2}\,dx dt \le (p-1) \|f\|_p\|g\|_{p'}\,,\,\, p>2\,.
\end{equation}

This will give us \eqref{BI1}, \eqref{BI2} correspondingly. To get to \eqref{BI3} and further improvements as \eqref{BI4} we will need a bit more (stochastic integrals).

Notice that we already know (by \eqref{22}) that \eqref{pr} immediately proves the following

\begin{theorem}
\label{r1r2}
$\|R_1^2-R_2^2\|_p\le p-1, \, p\ge 2\,.$
\end{theorem}

Because $2R_1R_2 = U\circ (R_1^2-R_2^2)\circ U^{-1}$, where $U$ is an isometry in all $L^p$ spaces (in fact, $U$ is the rotation of the argument of function by $45^{\circ}$), we get \eqref{BI1} from doubling
the claim of Theorem \ref{r1r2}.

\bigskip

\begin{proof}[proof of \eqref{pr}]
The first step is  by examination of what we already had in Section \ref{proofexist} after the statement of Theorem \ref{tE}. We do now {\bf exactly} the same: 

{\bf Suppose we have the following inequality for functions on interval $[0,1]$ provided with dyadic lattice $\mathcal{D}$}:
\begin{equation}
\label{burk1}
\Sigma_{I\in {\mathcal D}} |(f,h_I)|\, |(g,h_I)| \leq (p-1)\,\|f\|_{L^p}\|g\|_{L^{p'}}\,, p\ge 2\,.
\end{equation}

This inequality is scaleless, so we write it as 

\begin{equation}
\label{bu2}
J\in {\mathcal D},\,\, \frac1{4|J|}\Sigma_{I\in {\mathcal D},\,I\subset J} |\langle f\rangle_{I_-}-\langle f\rangle_{I_+}|\, |\langle
g\rangle_{I_-}-\langle g\rangle_{I_+}||I| \leq (p-1)\,\langle |f|^p\rangle_{J}^{1/p}\langle |g|^{p'}\rangle_{J}^{1/p'}\,.
\end{equation} 

Here $I_-,I_+$ are the left and the right halves of $I$, and $\langle\cdot\rangle_l$ means averaging over $l$ as usual. 
Given a fixed $J\in {\mathcal D}$, $p\ge 2$, we wish to introduce the Bellman function of \eqref{1.60}:

\[
\aligned
B_p(X,Y,x,y) &= \sup\{\frac1{4|J|}\Sigma_{I\in {\mathcal D},\,I\subset J} |\langle f\rangle_{I_-}-\langle f\rangle_{I_+}|\, 
|\langle g\rangle_{I_-}-\langle g\rangle_{I_+}||I| :\\
\langle f\rangle_J&=x,\,\langle g\rangle_J = y,\,
\langle |f|^p\rangle_J=X,\, \langle |g|^{p'}\rangle_J=Y \}\,.
\endaligned
\]

Obviously, the function $B$ does not depend on $J$, but it does depend on $p$. Its domain of definition is the following:
\[
R_p := \{  (X,Y,{\bf x},{\bf y}), \, |{\bf x}|^p \leq X,\,  |{\bf y}|^{p'}\leq Y\}\,.
\]

By \eqref{burk1} it satisfies 

\begin{equation}
\label{bu3}
0\leq B\leq (p-1)\,X^{1/p}Y^{1/p'}\,.
\end{equation}

We are going to prove that it also satisfies the following ``differential" inequality. Denote $v:=(X,Y,{\bf x},{\bf y})$, $v_-=(X_-,Y_-,{\bf x}_-,{\bf y}_-)$, $v_+=(X_+,Y_+,{\bf x}_+,{\bf y}_+)$, let $v,v_+,v_-$ lie in $R_p$,  and let $v=\frac12(v_{-}+v_+)$. Then

\begin{equation}
\label{bu4}
B(v)-\frac12\left(B(v_+) + B(v_-)\right) \geq\frac14 |{\bf x}_{+}  - {\bf x}_-||{\bf y}_{+}  - {\bf y}_-|
\,.
\end{equation}
The proof is {\bf verbatim} the same as in Section \ref{proofexist}.   And this inequality in infinitesimal sense becomes

\begin{equation}
\label{bu5}
d^2 B_p \ge 2|d{\bf x}||d{\bf y}|
\,.
\end{equation}

Having the function $B_p$ satisfying 

1)  $0\le B_p \le (p-1) X^{1/p}Y^{1/p'}$;

2) $-d^2 B_p \ge 2|d{\bf x}||d{\bf y}|$.

Assuming that $B_p$ is sufficiently smooth (which incidentally it is, one can write the formula for $B_p$), we can repeat {\bf verbatim}  we can repeat the proof of Theorem \ref{1.2}: we start with analyzing ($x=(x_1,x_2)\in \R^2$)
\begin{equation}
\label{svsn1}
\left(\frac{\pd}{\pd t} -\Delta\right) b(x,t)\,
\end{equation}
exactly as in the proof of Theorem \ref{1.2}: the only difference that $b$ now is not $B_Q\circ v$ but our $B_p\circ v$ and $v$ also slightly different, it is now
$$
v(x,t) := (|f|^p(x,t), |g|^{p'}(x,t), f(x, t), g(x, t))\,,
$$
where these are heat extensions of functions on $\R^2$ with corresponding symbol. We estimate the expression in \eqref{svsn1} in a pointwise way from below using 2), and in the average on a slab, using 1) we got exactly \eqref{pr}, Theorem \ref{r1r2}, and, therefore, \eqref{BI1}.

\bigskip

\noindent{\bf Remark.} Notice that variables ${\bf x},{\bf y}$ are complex, they are ``bench guards" (``mestoblyustiteli") for complex-valued functions $f=u+iv, g=\phi+i\psi$. So actually $B_p$ is a function of $6$ real variables, and, hence, \eqref{bu5} should be understood as
\begin{equation}
\label{bu6}
-d^2 B_p(X;Y; u,v; \phi, \psi) = (H_{B_p} h, h) \ge 2\sqrt{du^2+dv^2}\sqrt{d\phi^2+d\psi^2}
\,,
\end{equation}
where $u, v, \phi, \psi$ are just real variables (they are ``bench guards" for functions with the same symbols and their heat extensions), and $h=(dX, dY, du, dv, d \phi, d\psi)$ is a notation (strange may be) for an arbitrary vector in $\R^6$.
\end{proof}

To obtain \eqref{BI2} we notice first that in Theorem \ref{r1r2} we can use $R_1\cos\theta-R_2\sin\theta$ in place of $R_1$, and $R_1\sin\theta+R_2\cos\theta$ in place of $R_2$. In fact this is just application of rotation on $\theta$ in arguments. Then we notice that $(R_1\cos\theta-R_2\sin\theta)^2 - (R_1\sin\theta+R_2\cos\theta)^2= (R_1^2-R_2^2)\cos 2\theta -2R_1R_2\sin 2\theta$. Therefore, we got
\begin{theorem}
\label{phi}
For any $\phi\in (0, 2\pi]$, $\|(R_1^2-R_2^2)\cos \phi -2R_1R_2\sin \phi\|_p \le p-1$ if $p\ge 2$.
\end{theorem}

We notice that a certain estimate of $T=(R_1^2-R_2^2) +2iR_1R_2$ can be obtained if we answer the following question. Suppose $A, B$ are two operators in $L^p(\mu)$, and 
for any angle $\|A\cos \phi -B\sin \phi\|_p \le 1$, then what is the estimate of $\|A-iB\|_p$?

This is easy on real functions, let $f\in L^p_{real}(\mu)$, and let $A, B$ map real functions to real functions ($A=R_1^2 - R_2^2, B=2R_1R_2$ are such). In fact,
$$
\int |f|^p\,d\mu\ge \int |(Af)(x)\cos\phi +(Bf)(x)\sin\phi|^p \, d\mu= 
$$
$$
\int (|Af|^2 + |Bf|^2|^{p/2} |\cos (a(x)-\phi)|^p \, d\mu(x) \,.
$$
Integrate this over $\frac1{2\pi}\int_0^{2\pi}\dots$, by Fubini' theorem we will get
\begin{equation}
\label{av}
\int |f|^p\,d\mu\ge \int (|Af|^2 + |Bf|^2)^{p/2}\,d\mu \cdot \frac1{2\pi}\int_0^{2\pi}|\cos \phi |^p\,d\phi\,.
\end{equation}
Put
$$
\tau(p) := \left( \frac1{2\pi}\int_0^{2\pi}|\cos \phi |^p\,d\phi\right)^{1/p}\,,
$$
then on {\bf real functions}
\begin{equation}
\label{av1}
\|A+iB\|_p \le \sup_{\phi}\|A\cos\phi +B\sin\phi\|_p/\tau(p)\,.
\end{equation}

Unfortunately this was in real category. We do not know how obtain \eqref{av1}--or something like that--for general operators $A,B$ on complex function. 
May be this is also {\bf an exercise?}

However, we will obtain now \eqref{BI2}. First we need

\begin{proof}[The proof of \eqref{sq} ]

We use the following elementary lemma from Linear Algebra:

\begin{lemma}[Linear Algebra lemma]
Let $A,B,C$ be nonnegative matrices of size $d\times d$. Let
\begin{equation}
\label{ABC}
(Ah, h)\ge 2 (Bh, h)^{1/2}(Ch, h)^{1/2}\,,\,\,\forall h\in \C^d\,.
\end{equation}
Then there exists $\tau\in (0, \infty)$ independent of $h$ such that
$$
(Ah, h)\ge \tau (Bh, h) +\frac1{\tau}(Ch, h)\,,\,\,\forall h\in \C^d\,.
$$
\end{lemma}
\begin{proof} {\bf Exercise.}
\end{proof}

We apply this lemma separately   to  $h_1, h_2$, where ($f=u+iv, g=\phi+i\psi$, )
 $$h_1=(\pd_{x_1} |f |^p(x,t), \pd_{x_1} |g|^{p'}(x,t),\pd_{x_1} u(x,t),\pd_{x_1}v(x,t),\pd_{x_1} \phi(x,t),\pd_{x_1} \psi(x,t))\,,$$   $$h_2=(\pd_{x_2} |f |^p(x,t), \pd_{x_2} |g|^{p'}(x,t),\pd_{x_2} u(x,t),\pd_{x_2}v(x,t),\pd_{x_2} \phi(x,t),\pd_{x_2} \psi(x,t))\,,$$ and $A=H_{B_p}(|f |^p(x,t),  |g|^{p'}(x,t), u(x,t),v(x,t),\phi(x,t), \psi(x,t))$, and $B$ consisting of all zeros except $3,3$ and $4,4$ entries, where we have $1$, and $C$ consisting of all zeros except $5,5$ and $6,6$ entries, where we have $1$. 

Then we immediately get \eqref{sq}.

\end{proof}

\begin{proof}[The proof of \eqref{BI2} ]
We use the previous notations. We want a better estimate of $Tf=(A+iB)(u+iv)= Au-Bv + i (Av+Bu)$. Using the trick above \eqref{av} we can average the following equality over $(0,2\pi)$
$$
 \int |(Au-Bv)(x)\cos\phi +(Av +Bu)(x)\sin\phi|^p d\mu= 
 $$
 $$
 \int (|Au-Bv|^2 + |Av+Bu|^2|^{\frac{p}2} |\cos (a(x)-\phi)|^p  d\mu(x) \,.
$$
Then we get
$$
\tau (p)\cdot (\int |Tf|^p)^{1/p} \le \sup_{\phi}(\int |(Au-Bv)(x)\cos\phi +(Av +Bu)(x)\sin\phi|^p)^{1/p} =
$$
$$
\sup_{\phi}\sup_{\text{real}\,\psi\,,\,\|\psi\|_{p'}\le1} \int [(Au-Bv)(x)\cos\phi +(Av +Bu)(x)\sin\phi]\psi(x)\,dx=:E\,
$$
However the last expression can be rewritten using \eqref{22} and integration by parts as follows:

$$
E= 2\Re\iint_{\R^3_+} (\pd_{x_1}+i\pd_{x_2}) f(x,t) (\pd_{x_1}+i\pd_{x_2})e^{-i\phi}\psi(x,t)dxdt\le
$$
$$
2\sqrt{2}\iint \left( |\pd_{x_1}f|^2 +|\pd_{x_2} f|^2\right)^{1/2} \left((\pd_{x_1}\psi)^2 +(\pd_{x_2} \psi)^2\right)^{1/2}\le \sqrt{2} (p-1)\|f\|_p\,,\, p>2\,.
$$
We used \eqref{sq}. Here $\sqrt{2}$ appeared trivially from 
$$
|(\pd_{x_1}+i\pd_{x_2}) f(x,t)|\le \sqrt{2}\left( |\pd_{x_1}f|^2 +|\pd_{x_2} f|^2\right)^{1/2} \,.
$$

Finally we get
\begin{equation}
\label{asymp1}
\|T\|_p \le \frac{\sqrt{2}(p-1)}{\left( \frac1{2\pi}\int_0^{2\pi}|\cos \phi |^p\,d\phi\right)^{1/p}} \,, p>2\,.
\end{equation}

Asymptotically this is $1.41...(p-1)$. Choosing large $p$, interpolating between $L^2$, where the norm of $T$ is $1$ and the estimate \eqref{asymp1} for this large $p$, then optimizing by the choice of $p$ one can get  \eqref{BI2} ({\bf exercise!}).

\end{proof}

Notice that \eqref{sq} immediately proves the following 

\begin{theorem}
\label{real}
1) $\|T:L^p_{real}\rightarrow L^p\| \le \sqrt{2} (p-1), \, p\ge 2$;

2) $|(Tf, g)|\le (p-1)\|f\|_p\|g\|_{p'},\, p\ge 2\,,\,\text{if}\,\, f, g\,\,\text{are real valued}\,.$
\end{theorem}

\begin{proof}
Just look at \eqref{BIhm1}, compare it with \eqref{sq} and the fact that $(|DF|_2^2-2\det DF)^{1/2} \le \sqrt{2} |DF|$, $F=(u,v)$. We also need to notice that in this inequality  for real valued $f=u+i0$ we have $\det DF=0$ and the constant $\sqrt{2}$ can be replaced by $1$. Finish the proof: {\bf exercise}.

\end{proof}

\bigskip

So everything above hinges on inequality \eqref{burk1}. This inequality was proved by Burkholder in mid 80's and it is one of the {\bf remarkable inventions}. It is done by use of {\bf Bellman function technique}.

\subsection{The proof of inequality \eqref{burk1}. Burkholder's Bellman function.}
\label{prbu1}

We follow \cite{Bu1}, \cite{Bu5}, \cite{Bu7}--but loosely. See also the exposition in the review paper \cite{Ba}.

Let $f$ be real valued on  $[0,1]=:I_0$. Let $\{h_I\}_{I\in \mathcal D}$ be the usual Haar functions on $I_0$ normalized in $L^2$. Consider an operator 
$$
T_{\eps} f= \sum_{I\in \mathcal D} \eps_I (f, h_I) h_I\,,\,\, \eps:=\{\eps_I\}_I\,,\,\, \eps_I=\pm 1\,.
$$
This family is called martingale transforms.

Burkholder proved the following remarkable

\begin{theorem}
\label{Buth}
$\sup_{\eps}\|T_{\eps}\|_p=p^*-1: =\max (p, p/(p-1)) -1\,.$
\end{theorem}

He gave several proofs, all difficult, to be found in \cite{Bu1}--\cite{Bu7}. Another proof by Vasyunin--Volberg see arxiv: 1006.2633, \cite{VaVo}.

In all these proofs the following object is indispensable. It is Burkholder's {\bf Bellman function}.

Let $\Om:=\{(x,y,z): |x|^p\le z\}$ and let
$$
B(x,y, z):= \sup\{ \|g\|_p^p: \langle f\rangle_{I_0} =x, \,\langle g\rangle_{I_0} =y, \langle |f|^p\rangle_{I_0} =z, \forall I\in \mathcal D \,\,|(g, h_I)|=|(f,h_I)|\}\,.
$$
{\bf Symmetries:}

\begin{equation}
\label{sym}
B(tx,ty, t^p z) = t^p\,B(x,y,z)\,,\, B(-x, y)= B(x,y)\,,\, B(x, -y) = B(x,y)\,.
\end{equation}

Burkholder found {the formula} for $B$:

Consider for positive $x,y$
$$
F_p(x,y)=\begin{cases} y^p -(p^*-1)^p x^p\,,\,\,\text{if}\,\,  y\le (p^*-1)x\,;\\
p\left( 1-\frac1{p^*}\right)^{p-1} (y+x)^{p-1} (y-(p^*-1)x)\,,\,\,\text{if}\,\, y\ge (p^*-1)x\,.\end{cases}
$$
Consider the solution of an implicit equation:
$$
F_p (|x|, |y|) = F_p(z^{1/p}, B^{1/p}(x,y,z))\,.
$$
If $p\ge 2$ Burkholder's function is the solution of this equation. If $p\in (1,2]$, then one considers $F_p (|y|, |x|) = F_p( B^{1/p}(x,y,z), z^{1/p})$.

Obviously one gets a
\begin{theorem}
\label{001}
$B(0,0,1) =(p^*-1)^p\,,$
\end{theorem}
\noindent which gives Theorem \ref{Buth}, from which we get that \eqref{burk1} is proved right away. In fact,
\begin{proof}[The proof of \eqref{burk1}]
We write $\sup_{\eps}|(T_{\eps}f,g)| \le (p^*-1)\|f\|_p\|g\|_{p'}$, which follows from Theorem \ref{Buth}. But this supremum obviously is equal to 
$$
\Sigma_{I\in {\mathcal D}} |(f,h_I)|\, |(g,h_I)|\,.
$$
Therefore \eqref{burk1} is proved.
\end{proof}

\noindent{\bf Remarks.} 1) As soon as \eqref{burk1} is proved we have our {\bf Bellman} function $B_p$. 

2) It gives all our inequalities like \eqref{sq} and its consequences like \eqref{BI2}.

3) It is {\bf not} Burkholder's function. 

4) The existence of our Bellman function $B_p$ follows from the existence of Burkholder's Bellman function. {\bf These are demographic creatures, they create one another--we saw this in previous sections too.}

\bigskip

We are left to prove Theorem \ref{Buth}. Instead of finding exact formula for $B(x,y,z)$ listed above we will use a certain shortcut (invented already by Burkholder himself). Suppose Burkholder's $B$ is finite.

\begin{proof} [The shortcut proof of Theorem \ref{Buth}]
Along with symmetries \eqref{sym} it has very good concavity properties:
\begin{equation}
\label{unison}
B(x,y,z)-\frac12( B(x+\al, y+\al, z+\beta) + B(x-\al, y-\al, z-\beta)) \ge 0\,,
\end{equation}
if all points lie in $\Om$. Also
\begin{equation}
\label{antiunison}
B(x,y,z)-\frac12( B(x+\al, y-\al, z+\beta) + B(x-\al, y+\al, z-\beta)) \ge 0\,,
\end{equation}
if all points lie in $\Om$.

Inequalities \eqref{unison}, \eqref{antiunison} are left as {\bf exercise}.

Notice that this means that
$$
M(a, b, c):=  B(a+b, a-b, c)
$$ is concave in $(a,c)$, and in $(b, c)$.

{\bf Definition}. Such $M$ is called bi-concave.

{\bf Definition.} Function $\f$ on $\R^2$ is called zigzag concave if
$$
\f(x,y)-\frac12( \f(x+\al, y+\al) + \f(x-\al, y-\al) \ge 0\,,
$$
$$
\f(x,y)-\frac12( \f(x+\al, y-\al) + \f(x-\al, y+\al) \ge 0\,,
$$
or, which is the same as,
$$
\f(x,y) -\frac12 (\f(x^+, y^+) + \f(x^-, y^-)) \ge 0\,,\,\,\text{if}\,\,
$$
$$
 |x^+-x^-|=|y^+-y^-|\,, x=\frac12(x^++x^-)\,,\,\, y=\frac12(y^++y^-)\,.
$$

\begin{theorem}
\label{phi}
Put $\f(x,y):= \sup_{(x,y,z)\in \Om} [ B(x,y,z) -(p^*-1)^p z]$. It is zigzag concave. It is the least zigzag concave majorant of $h(x,y):= |y|^p-(p^*-1)^p |x|^p$. There is {\bf no} zigzag concave majorant $\psi$ such that
$\psi(tx, ty) =t^p \psi(x,y)$ of function $h_c:=|y|^p-c\,|x|^p$ if $c<(p^*-1)^p$.
\end{theorem}

\begin{proof}
Put $c_p= (p^*-1)^p$. Fix $(x^-, y^-)$ and $(x^+, y^+)$. Find $z^-$ which almost gives supremum in $\f(x^-, y^-)=\sup [B(x^-, y^-, z)- c_p z]$. Do the same for $\f(x^+, y^+) $ to find $z^+$.
Then
$$
B(x^-, y^-, z^-)- c_p z^- \le \f(x^-, y^-) \le B(x^-, y^-, z^-)- c_p z^- +\eps\,,
$$
$$
B(x^+, y^+, z^+)- c_p z^+ \le \f(x^+, y^+) \le B(x^+, y^+, z^+)- c_p z^+ +\eps\,.
$$
Let $x=\frac12(x^++x^-)\,,\,\, y=\frac12(y^++y^-)$ and put $z=\frac12(z^++z^-)$. Then
$$
\f(x,y) =\sup\dots \ge B(x,y,z) -c_p z = B(x,y,z)-c_p\frac12(z^++z^-) \ge
$$
$$
\frac12 (B(x^-, y^-, z^-)- c_p z^-) + \frac12 (B(x^+, y^+, z^+)- c_p z^+) \ge \frac12(\f(x^-, y^-) + \f(x^+, y^+) )-2\eps\,.
$$
So $\f$ is zigzag concave. Also
$$
\f(x,y) =\sup\dots\ge \lim_{z\rightarrow |x|^p+} [B(x,y,z)-c_p z] \ge |y|^p -c_p |x|^p = h(x,y)\,.
$$
So $\f$ is a zigzag concave majorant of $h$. Why the least?  Let $\psi$ be any zigzag concave function such that
$$
h \le \psi\,.
$$
Put $\Psi:= \psi(x,y) + c_p z$. Then it is easy to see that $\Psi$ satisfies \eqref{unison}, \eqref{antiunison}. Also on $\pd\Om=\{z=|x|^p\}$ we have
$$
\Psi(x, y, z) \ge h(x,y) + c_p z= h(x,y) +c_p |x|^p =|y|^p\,.
$$
Then  combination of the last inequality and the fact that $\Psi$ satisfies \eqref{unison}, \eqref{antiunison} gives (attention {\bf exercise!})
$$
\Psi(x,y,z)\ge B(x,y,z)\,.
$$
This a non-trivial exercise.  But then trivially for every $(x,y)$
$$
\psi(x,y) =\sup_{z: (x,y,z)\in \Om}[\Psi(x,y,z) -c_p z] \ge \sup_{z: (x,y,z)\in \Om}[B(x,y,z) -c_p z] =\f(x,y)\,.
$$

We need now to prove that $h_c, c<c_p$ does not have zigzag concave homogeneous majorant.

This and more is done in 
\begin{lemma}
\label{zz}
Function $h_c, c<c_p$ does not have zigzag concave homogeneous majorant. If $c=c_p$, then the function $h_{c_p}=:h$ has such majorant given by
$$
\Phi_0(x,y) :=\begin{cases} |y|^p - (p^*-1)^p =\,h(x,y) \,,\,\,\text{if}\,\, h\le 0\,;\\
p\left(1-\frac1{p^*}\right)^{p-1} (|y|+|x|)^{p-1} (|y|- (p^*-1)|x|)\,,\,\,\text{if}\,\, h>0\,.\end{cases}
$$
Another zigzag concave majorant of $h=h_{c_p}$  (but not the least) is given by
$$
\Phi(x,y) := p\left(1-\frac1{p^*}\right)^{p-1} (|y|+|x|)^{p-1} (|y|- (p^*-1)|x|)\,.
$$
\end{lemma}

\noindent{\bf Remark.} The fact that $\Phi_0(x,y)\le 0$ if $|x|\ge |y|$ will be crucial for the proof of Theorem \ref{Buth}.

\begin{proof}[The proof of Lemma \ref{zz}]
We work in the firs quadrant.  Homogeneous $\f$ can be written as 
$$
\f(x,y) = (x+y)\f (\frac{x}{x+y}, \frac{y}{x+y})\,,\, s:= \frac{y-x}{y+x}\,,\,\, \text{then}\,\, \frac{1+s}{2}= \frac{y}{x+y}\,,\, \frac{1-s}{2}= \frac{x}{x+y}\,.
$$
$$
s'_x = -\frac{1+s}{x+y}\,,\, s'_y = \frac{1-s}{x+y}\,.
$$ 
We put $g(s):= \f ( \frac{1-s}{2}, \frac{1+s}{2})$. Next we list some results of computations:
$$
\f_x = p(x+y)^{p-1} g(s) + (x+y)^{p-1} g'(s) (1+s)\,,\, \f_y= p(x+y)^{p-1} g(s) + (x+y)^{p-1} g'(s) (1-s)\,.
$$
$$
\f_{xx}= p(p-1) (x+y)^{p-2} g(s) - 2(p-1)(x+y)^{p-2} g'(s) (1+s) + (x+y)^{p-2} g''(s)(1+s)^2\,,
$$
$$
\f_{yy}= p(p-1) (x+y)^{p-2} g(s) + 2(p-1)(x+y)^{p-2} g'(s) (1-s) + (x+y)^{p-2} g''(s)(1-s)^2\,,
$$
$$
\f_{xy}= p(p-1) (x+y)^{p-2} g(s) - 2(p-1)(x+y)^{p-2} g'(s) s - (x+y)^{p-2} g''(s)(1-s^2)\,.
$$
So on $x+y=1$
$$
\f_{xy} = -[(1-s^2) g''(s) +2(p-1) s g'(s) -p(p-1) g(s)\,,
$$
$$
\f_{xx} +\f_{yy} = 2(1+s^2) g''(s) - 4(p-1) s g'(s) +2p(p-1) g(s)\,.
$$
So  combining the two:
$$
(\pd_{x-y})^2 \f = \f_{xx}-2\f_{xy} + \f_{yy} =  4 g''(s)\,.
$$
$$
(\pd_{x+y})^2 \f = \f_{xx}+ 2\f_{xy} + \f_{yy} = 4 g''(s) + 4\f_{xy} =
4( s^2 g''(s) +(p-1) (-2sg'(s) +pg(s)))\,.
$$
Zigzag concave means the last two lines have sign $\le 0$. To find $\f$ satisfying these two $\le 0$ differential inequalities, let us try first to find it in such a way that the first inequality is equality! Hence we seek for the linear $g$! Then put 
\begin{equation}
\label{form}
g(s) = a \left( \frac{1+s}{2} -\rho\frac{1-s}{2}\right)\,.
\end{equation}

Then the second inequality $s^2 g''(s) +(p-1) (-2sg'(s) +pg(s))\le 0$ becomes 
\begin{equation}
\label{2conv}
2sg'(s) -p g(s) \ge 0\,,\,\,\text{on}\,\, [-1,1]\,.
\end{equation}
It is satisfied (as $g$ is linear) if and only if  it is satisfied in $-1$ and $1$.
We get
$$
a(1+\rho-p)\ge 0\,,\, -a(1+\rho +p\rho) \ge 0\,.
$$
As $g$ is greater than $\left (\frac{1+s}{2}\right)^p -c\left(\frac{1-s}{2}\right)^p$, it is positive at $s=1$, so $a>0$. then we get from previous inequalities that
$$
\rho \ge \max \left(p-1, \frac1{p-1}\right)= p^*-1\,.
$$
Let us try linear $g$ with $\rho= p^*-1$.
So $g$ has zero at $s_p$ such that
$$
p^*-1 = \frac{1+s_p}{1-s_p}\,.
$$
But if 
$$
h(x,y)= (x+y)^p H(s)\,,\, H(s):= \left (\frac{1+s}{2}\right)^p -c_p\left(\frac{1-s}{2}\right)^p\,,
$$
then $H$ has zero at the same point $s_p$. Now let us find $a$ from the condition
$$
H'(s_p) = g'(s_p) \Rightarrow a= p\left( 1 -\frac1{p^*}\right)^{p-1}\,.
$$
Where $H$ is concave on $[-1,1]$? Inflection point $i_p$ is such that
$$
\left (\frac{1+s}{2}\right)^{p-2} -c_p\left(\frac{1-s}{2}\right)^{p-2}=0\,.
$$
So it is clear that it is always $> s_p$. As $H$ is concave on $[-1, i_p]$ it is concave on $[-1, s_p]$ (and a little bit on the right of $s_p$ too).

It is also easy to see that on $[-1,s_p]$
\begin{equation}
\label{also}
s^2 H''(s) +(p-1) (-2s H'(s) +p H(s)) \le 0\,.
\end{equation}

{\bf This is an  exercise}.

\bigskip

Put now
$$
\tilde g(s) =\begin{cases} \text{our linear}\,\, g(s)\,\, s\in [s_p,1]\,;\\
H(s)\,,\,\, s\in [-1,s_p]\,.\end{cases}
$$
Then $\Phi_0(x,y)= (x+y)^p \tilde g(\frac{x}{x+y}, \frac{y}{x+y})$ is exactly the same $\Phi_0$ as in Lemma \ref{zz}'s statement. We just checked that it is a zigzag concave majorant of $h(x,y)$. We also checked that $\Phi(x,y)= (x+y)^p g(\frac{x}{x+y}, \frac{y}{x+y})$, where $g$ is our linear function built above, is zigzag concave majorant of $h(x,y)$ as well. It is exactly function $\Phi$ as in Lemma \ref{zz}'s statement. 

\bigskip

Now let $c<c_p$. 
Linear function cannot be higher than corresponding $H_c$ on $[-1,1]$ and satisfy \eqref{2conv}. In fact, if $\al s +\beta$ is higher, then $\al +\beta>0$. Also \eqref{2conv} gives
$$
2s\al -p(\al +\beta) \ge 0\Rightarrow \al (2-p)s -p\beta\ge 0, \al(p-2) -p\beta\ge0\,.
$$
Then $\beta <0, \al >0$. So linear function is positive in $1$ and negative  at zero. So it must vanish on $[-1,1]$, hence it has the form \eqref{form}. Hence, we we can see that their minorant $H_c$ can have only $c\ge c_p$. In fact, we remember that $\rho\ge p^*-1$ in \eqref{form}. Then the zero of our linear function must be $\ge  c_p$. But if our linear function is a majorant of $H_c$ with $c<c_p$ it is also a majorant of $H_{c_p}$. Therefore, its zero must be $<c_p$. This is a contradiction, and a linear solution of two differential inequalities will not have minorant with $c<c_p$. Concave solution will not have such minorants  either. {\bf Exercise}.

Lemma \ref{zz} is finished.

\end{proof}

Theorem \ref{phi} is completely proved.

\end{proof}

\begin{proof}[Finishing the proof of Theorem \ref{Buth}. The real case.]

Now that we have function $\Phi$ ($\Phi_0$ will work too) that is

1) zigzag concave on the plane,

2) is such that $\Phi(x,y) \ge h(x,y):= |y|^p - (p^*-1)^p |x|^p$,

\noindent we can do the following. Fix $f, g$ step functions on $I:=[0,1]$. Consider points $P=(x,y)=(\la f\ra_I , \la g\ra_I)$, $P^+=(x^+, y^+)=(\la f\ra_{I_+} , \la g\ra_{I_+})$, $P^-=(x^-,y^-)=(\la f\ra_{I_-} , \la g\ra_{I_-})$.
Notice that of course  $P=\frac12(P^++P^-)$. Also $|x^+-x^-|=|y^+-y^-|$ because this differences are $\frac{2}{\sqrt{|I|}}|(f, h_I)$ and $\frac{2}{\sqrt{|I|}}|(g, h_I)$ correspondingly, and we assumed in Theorem \ref{Buth} that for every dyadic interval $|(f, h_I)|= |(g,h_I)|$. Let also $|x|\ge |y|$ (for example both are zeros)

Then we can use properties of $\Phi$: 
$$
0\ge \Phi(x,y) \ge \Phi(x^+,y^+) |I_+|   +  \Phi(x^-,y^-) |I_-| \,.
$$
As intervals $I^+, I^-$ are as good as $I$ we can repeat this for them. Just iterating this procedure and denoting by $I^{\sigma}$ dyadic intervals of size $2^{-n}$ with $\sigma$ being any string of $\pm$ of length $n$ we get
$$
\Sigma_{\sigma} \Phi(x^\sigma,y^\sigma) |I^\sigma| \le 0\,.
$$
Combine this with property 2) above. Then
$$
\Sigma_{\sigma} |y^\sigma|^p |I^\sigma| \le (p^*-1)^p \Sigma_{\sigma} |x^\sigma|^p |I^\sigma| \,.
$$
But by our construction $y^\sigma=\la g\ra_{I^\sigma}, x^\sigma=\la f \ra_{I^\sigma}$. So we get
$$
\Sigma_{\sigma} |\la g\ra_{I^\sigma}|^p |I^\sigma| \le (p^*-1)^p \Sigma_{\sigma} |\la f \ra_{I^\sigma}|^p |I^\sigma| \,.
$$
Going to the limit when $n\rightarrow\infty$ we get $\la |g|^p \ra_I\le (p^*-1)^p \la |f|^p\ra_I$, which gives the claim of Theorem \ref{Buth} in the case of real-valued $f,g$.

\end{proof}

\begin{proof}[Finishing the proof of Theorem \ref{Buth}. The complex-valued and Hilbert-valued cases.]

A certain ``miracle" happens: $\Phi, \Phi_0$ have extra properties of symmetry, not apparent at this moment.

{\bf Extra symmetry.}
 Consider $\f(x,y) = \Phi(\sqrt{x_1^2+x_2^2}, \sqrt{y_1^2+y_2^2})$. We use standard notations, now $x, y$ are vectors, $\|\cdot\|$ is the norm of a vector, $dx:= (dx_1, dx_2), dy :=(dy_1, dy_2)$ are also arbitrary vectors. 
 
 We want to see that 
 
 1) $-d^2\f:=-(H_\f\diff,\diff) \ge 0$, if $\|dx\|=\|dy\|$. This is ``zigzag concavity'' direct analog.
 
 2) $\f(x,y) \ge h(\|x\|, \|y\|)$.
 
 The second is obvious, but the first happens by a ``miracle".
 Let us prove it and see, where the ``miracle" happens.  
 
 Calculations (really abusing the language we understand that $\Phi_x, \Phi_y$ are partial derivatives of $\Phi$ with respect to the first and the second variables):
 $$
 \f_{x_1}=\Phi_x\cdot \frac{x_1}{\sqrt{x_1^2+x_2^2}}\,,\,  \f_{x_2}=\Phi_x\cdot \frac{x_2}{\sqrt{x_1^2+x_2^2}}\,.
 $$
 $$
 \f_{x_1x_1}=\Phi_{xx}\cdot \frac{x_1^2}{x_1^2+x_2^2}+\Phi_x \frac{x_2^2}{(x_1^2+x_2^2)^{3/2}}\,,\,  \f_{x_2x_2}=\Phi_{xx}\cdot \frac{x_2^2}{x_1^2+x_2^2}+\Phi_x \frac{x_1^2}{(x_1^2+x_2^2)^{3/2}}\,.
 $$
 $$
  \f_{x_1x_2}=\Phi_{xx}\cdot \frac{x_1x_2}{x_1^2+x_2^2}-\Phi_x \frac{x_1x_2}{(x_1^2+x_2^2)^{3/2}}\,.
$$
Symmetrically for $y$ derivatives.
Also
$$
\f_{x_iy_j} =\Phi_{xy} \cdot \frac{x_iy_j}{\sqrt{x_1^2+x_2^2}\sqrt{y_1^2+y_2^2}}\,,\, i,j=1,2\,.
$$
Therefore,
$$
-(H_\f\diff, \diff) = \frac{\Phi_x}{\|x\|}\left(\frac{x_2dx_1-x_1dx_2}{\|x\|}\right)^2 +\frac{\Phi_y}{\|y\|}\left(\frac{y_2dy_1-y_1dy_2}{\|y\|}\right)^2 +
$$
$$
\Phi_{xx}\left(\frac{x_2dx_1+x_1dx_2}{\|x\|}\right)^2 +2\f_{xy}\left(\frac{x_2dx_1+x_1dx_2}{\|x\|}\right)\left(\frac{y_2dy_1+y_1dy_2}{\|y\|}\right) +\f_{yy}\left(\frac{y_2dy_1+y_1dy_2}{\|y\|}\right)^2
$$
$$
= \frac{\Phi_x}{\|x\|}\|\hat{dx}\|^2 + \frac{\Phi_y}{\|y\|}\|\hat{dy}\|^2 + \Phi_{xx}(dx, \frac{x}{\|x\|})^2+ 2 \Phi_{xy}(dx, \frac{x}{\|x\|})(dy, \frac{y}{\|y\|}) + \Phi_{yy}(dy, \frac{y}{\|y\|})^2\,,
$$
where $\hat{dx}, \hat{dy}$ are projections of vectors $dx, dy$ on direction orthogonal to $x,y$ correspondingly.

Recall that up to a positive constant (which we drop now abusing the language)
$$
\Phi(x,y) = (y-(p-1)x)(x+y)^{p-1}\,,\, \text{if} \,\, p\ge 2\,,
$$
and
$$
(p-1) \Phi(x,y) = -(x-(p-1)y)(x+y)^{p-1}\,,\, \text{if} \,\, p\le 2\,.
$$
Let us consider $p\ge 2$, the other case being similar. Looking at the formulae above we get by direct calculation with formula for $\Phi$ that for any numbers $h', k'$

$$
\Phi_{xx} h'^2 + 2 \Phi_{xy} h'k' +\Phi_{yy} k'^2 =-p(p-1) (x+y)^{p-2}(h'^2-k'^2) -p(p-1)(p-2) x(x+y)^{p-3}(h'+k')^2\,.
$$
(By the way we immediately see that this form is $\le 0$ if $|k'|=|h'|$, which is infinitesimal version of zigzag concavity.)

Now let us combine our formulae, putting $h'=(dx, \frac{x}{\|x\|}), k' =(dy, \frac{y}{\|y\|})$. Then
$$
(H_\f\diff,\diff) =  \frac{\Phi_x}{\|x\|}\|\hat{dx}\|^2 + \frac{\Phi_y}{\|y\|}\|\hat{dy}\|^2  - p(p-1) (\|x\|+\|y\|)^{p-2}(h'^2-k'^2) 
$$
$$
-p(p-1)(p-2) \|x\|(\|x\|+\|y\|)^{p-3}(h'+k')^2\,.
$$
Let us look at the first line of the last formula. Calculate 
$$
\frac{\Phi_x}{\|x\|}\|\hat{dx}\|^2 + \frac{\Phi_y}{\|y\|}\|\hat{dy}\|^2 =\left(\frac{\Phi_x}{\|x\|} +\frac{\Phi_y}{\|y\|}\right)\|\hat{dy}\|^2 + \frac{\Phi_x}{\|x\|}( k'^2 -h'^2) + Term\,,
$$
where $Term:= \frac{\Phi_x}{\|x\|} (\|h\|^2-\|k\|^2)$. This is just because $\|\hat{dx}\|^2+h'^2=\|h\|^2$, and the same is true for $k$. In particular, 
\begin{equation}
\label{term}
Term =0\,\,\text{if}\,\, \|h\|=\|k\|\,,\,\text{and}\,\, Term \le 0\,,\text{if}\,\,\|h\|\ge\|k\|\\,.
\end{equation}
In fact,
\begin{equation}
\label{fx}
\frac{\Phi_x}{\|x\|} = -p(p-1) (\|x\|+\|y\|)^{p-2}<0\,.
\end{equation}
Combine three last formulae. Then we have
$$
(H_\f\diff,\diff) =\left(\frac{\Phi_x}{\|x\|} +\frac{\Phi_y}{\|y\|}\right)\|\hat{dy}\|^2 
$$
$$
 +Term -p(p-1)(p-2) \|x\|(\|x\|+\|y\|)^{p-3}(h'+k')^2\,.
$$
The second line is obviously negative (see \eqref{term}). To have the first line negative it is necessary and sufficient to have
\begin{equation}
\label{miracle} 
\left(\frac{\Phi_x}{\|x\|} +\frac{\Phi_y}{\|y\|}\right)\le 0\,.
\end{equation}
Calculate: 
$$
\frac{\Phi_y}{\|y\|}=p (\|y\|-(p-2)\|x\|) (\|x\|+\|y\|)^{p-2}\,.
$$
Combine this with \eqref{fx} to get
$$
\left(\frac{\Phi_x}{\|x\|} +\frac{\Phi_y}{\|y\|}\right)= -(\|x\|+\|y\|)^{p-2}\left( p(p-1) -p +p(p-2)\frac{\|x\|}{\|y\|}\right)=
$$
\begin{equation}
\label{fxfy}
-p(p-2)\frac{(\|x\|+\|y\|)^{p-1}}{\|y\|} \le 0\,,\text{if}\,\, p\ge 2\,.
\end{equation}

The case $p<2$ goes along the same lines with corresponding change in the formula for $\Phi$. The proof of Theorem \ref{Buth} is finished in the complex-valued case.
One can notice that the same proof works in any Hilbert space, not just $2$-dimensional as above, {\bf exercise}!

\end{proof}

Theorem is finally completely proved.

\end{proof}

We want to remember a formula that has been just obtained ($p\ge 2$):

$$
(H_\f\diff,\diff) =-p(p-2)\frac{(\|x\|+\|y\|)^{p-1}}{\|y\|}\|\hat{dy}\|^2-p(p-1) (\|x\|+\|y\|)^{p-2} (\|dx\|^2-\|dy\|^2)
$$
\begin{equation}
\label{d2fb}
 -p(p-1)(p-2) \|x\|(\|x\|+\|y\|)^{p-3}\left((dx, \frac{x}{\|x\|})+(dy, \frac{y}{\|y\|})\right)^2\,, 
\end{equation}
where $\|\hat{dy}\|^2= \|dy\|^2-(dy, \frac{y}{\|y\|})^2$.
Also
\begin{equation}
\label{maj1}
p\left(1-\frac1{p^*}\right)^{p-1} \f \ge \|y\|^p-(p^*-1)^p\|x\|^p\,.
\end{equation}
On the other hand, if $1<p<2$, we know that 
$$
\f(x,y) =(\|y\|^p-(p^*-1) \|x\|)(\|x\|+\|y\|)^{p-1}
$$
 will satisfy
$$
(H_\f\diff,\diff) =-p(2-p)\frac{(\|x\|+\|y\|)^{p-1}}{\|x\|}\|\hat{dx}\|^2  -p(p-1) (\|x\|+\|y\|)^{p-2} (\|dy\|^2-\|dx\|^2)
$$
\begin{equation}
\label{d2fm}
 -p(p-1)(2-p) \|x\|(\|x\|+\|y\|)^{p-3}\left((dx, \frac{x}{\|x\|})+(dy, \frac{y}{\|y\|})\right)^2\,, 
\end{equation}
where $\|\hat{dx}\|^2= \|dx\|^2-(dx, \frac{x}{\|x\|})^2$.
The same majorization \eqref{maj1} happens for $1<p<2$ as well.

\vspace{.3in}

Inequalities \eqref{BI1}, \eqref{BI2} are completely done.
However, to move further, in particular to \eqref{BI3}, \eqref{BI4}, we need a new tool={\bf stochastic integrals}.

\section{Stochastic Integrals. It\^o's formula}
\label{ito}

Let $w(s):=w_s$ denote Brownian motion started at $0$, that is $w_0=0$, and
for all $t_1<t_2<t_3$, random variables $w_{t_2}- w_{t_1}$, $w_{t_3}- w_{t_2}$ are Gaussian independent with zero average and variances $\sqrt{t_2-t_1}$, $\sqrt{t_3-t_2}$ correspondingly.

We want to understand what does it mean
$$
\int_a^b \xi(t) dw_t\,.
$$
It is {\bf not} the Riemann sum defintion as the following example shows.

\noindent{\bf Example.}  Consider two simplest Riemann sums built on a partition of the interval $[a,b]$:
$$
\Sigma_1:= \sum_{i=1}^m  w(t_{i-1})(w(t_i)-w(t_{i-1}))\,,
$$
$$
\Sigma_2:= \sum_{i=1}^m  w(t_{i})(w(t_i)-w(t_{i-1}))\,.
$$
If refinement is small, we should have had (if stochastic integral were a Riemann sum thing) that these two random variables $\Sigma_1$ and $\Sigma_2$ are close. Let us see, whether this is the case.

Notice that uniformly (when the partition changes) they are in $L^2(\Omega,  \mathcal{F}, \mathbb{P})$, where $(\Omega,  \mathcal{F}, \mathbb{P})$ is the probability space on which Brownian motion is given.
$$
\bE(\Sigma_1)^2=\sum_{i<j}\bE (w(t_{i-1})(w(t_i)-w(t_{i-1}))w(t_{j-1})\cdot 
(w(t_i)-w(t_{i-1}))+
$$
$$
\sum_i \bE ((w(t_{i-1})^2(w(t_i)-w(t_{i-1})^2) = \sum_{i<j}\bE (w(t_{i-1})(w(t_i)-w(t_{i-1}))w(t_{j-1})\cdot 
\bE(w(t_i)-w(t_{i-1})+
$$
$$
\sum_i t_i(t_i-t_{i-1}) \le b(b-a)\,.
$$
Also
$$
\bE(\Sigma_2)^2=2\bE \Sigma_1^2 + 2 \bE(\sum_i((w(t_i)-w(t_{i-1}))^2)^2\le
$$
$$
2b(b-a) + 2 \bE (\sum \xi_i^2)^2\,,
$$
where $\xi_i:= w(t_i)-w(t_{i-1})$ are Gaussian independent with average zero and $\sigma_i^2=|t_i-t_{i-1}|$. Then
$$
\bE (\sum \xi_i^2)^2 =2\sum_{i<j}\bE\xi_i^2\bE\xi_j^2 +\bE \xi_i^4 =
$$
$$
2\sum_{i<j} |t_i-t_{i-1}||t_j-t_{j-1}| + 3\sum_i |t_i-t_{i-1}|^2 \le 5(b-a)^2\,.
$$

The correct definition of integral should have been such that if $\Sigma_1, \Sigma_2$ are uniformly in $L^2$ and are both the Riemann sums, they should have 	been close in some sense. Suppose they are close (as random variables) in probability (one of the weakest sense possible). Then we use a simple {\bf exercise} that if $\|f_n\|_{L^2(\bP)} \le C$, $f_n \Rightarrow 0$, then
$\|f_{n_k}\|_{L^1(\bP)}\rightarrow 0$.

In our case, nothing like that happened: 
$$
\bE \Sigma_1=0\,,
$$
$$
\bE \Sigma_2 =\bE\Sigma_1 + \bE (\sum_i (w(t_i)-w(t_{i-1}))^2= \sum_i (t_i-t_{i-1})= b-a \neq 0\,.
$$
We understand now that stochastic integral $\int_a^b \xi(t)\, dw(t)$ is a much more subtle thing than Riemann sum integral. Stochastic integrals were understood by Kioshi It\^o.

\bigskip

\subsection{A bit on It\^o's definition}. 

Let $B\cF$ be a sigma algebra of sets $A\subset \R\times \Omega$ such that
for every $t\in [a,b]$ we have $A\cap ((-\infty, t]\times \Omega)$ is in $B_t\times \cF_t$, where $B_t$ is Borel sigma algebra on $(-\infty,t]$, $\cF_t$ is a sigma algebra generated by $\{w_s\}_{s\le t}$. Let $M_2[a,b]$ is the set of functions measurable with respect to $B\cF$ such that

(a) $f(t)$ is measurable with respect to $\cF_t$ for each $t$,

(b) with probability $1$, $\int_a^b|f(t)|^2\,dt <\infty$.

For all such random  functions (random processes) It\^o defines
\begin{equation}
\label{ito}
\int_a^b f(t)\,dw(t)\,.
\end{equation}

\noindent{\bf Definition.} $f\in M_2[a,b]$ is called a step function if there exists a partition such that
$f(t) =f(t_i)(\omega)\,,\text{for}\,\, t\in [t_i,t_{i+1})$.
We introduce the stochastic integral for them in a natural way
$$
\int_a^b f dt:= \sum_i f(t_i)(\omega)\cdot (w(t_{i+1})-w(t_i))\,.
$$

\begin{lemma}
\label{step}
For every $f\in M_2[a,b]$ there exits a sequence of step functions as above such that
with probability $1$
$$
\lim_{n\rightarrow\infty} \int_a^b |f(t)-f_n(t)|^2\,dt =0\,.
$$
Moreover if in addition
$$
\bE \int_a^b |f(t)|^2\, dt <\infty\,,
$$
then step functions can be chosen to have
$$
\lim_{n\rightarrow\infty}\bE \int_a^b |f(t)-f_n(t)|^2\,dt =0\,.
$$
\end{lemma}

We need the following Lemma.

\begin{lemma}
\label{itocheb}
Let $\f$ be a step function as above. Let $\delta, \epsilon >0$. Then
$$
\bP\{|\int_a^b\f(t)\,dw(t)|> \eps\} \le \frac{\delta}{\epsilon^2} +
\bP\{\int_a^b|\f(t)|^2\,dt>\delta\}\,.
$$
\end{lemma}

This lemma immediately gives the following reasoning. If--as above--$f\in M_2[a,b]$ and $f_n$ are step functions from Lemma \ref{step}, then
$$
\bP-\lim_{n\rightarrow\infty}\int_a^b |f(t)-f_n(t)|^2\,dt=0\,.
$$
Then
$$
\bP-\lim_{n,m\rightarrow\infty}\int_a^b |f_m(t)-f_n(t)|^2\,dt=0\,.
$$
By definition
$$
\forall \eps>0\,,\,\bP\{\int_a^b |f_m(t)-f_n(t)|^2\,dt>\eps\} \rightarrow 0\,,m,n\rightarrow\infty\,.
$$
Now we use Lemma \ref{itocheb} to have
$$
\limsup_{m,n\rightarrow\infty} \bP\{|\int_a^b f_n(t)\,dw(t)-\int_a^b f_m(t)\,dw(t)|>\eps\}\le \frac{\delta}{\eps^2}
$$
for any $\delta>0$.  So the sequence of random variables $\xi_n:=\int_a^b f_n(t)\,dw(t)$ is Cauchy convergent in measure (in probability). So in probability it converges to a certain random variable $\xi$.
 This $\xi$ is {\bf by definition} $\int_a^b f dw(t)$. {\bf It\^o's stochastic integral is constructed.}
 
 \bigskip

This integral has many nice properties:

If $
\bE \int_a^b |f(t)|^2\, dt <\infty
$, then $\bE \int_a^b f(t)\, dw(t)=0$ and
$$
\bE (\int_a^b f(t) dw(t))^2 =\bE\int_a^b |f(t)|^2\,dt\,.
$$
If in addition $
\bE \int_a^b |g(t)|^2\, dt <\infty
$ then
\begin{equation}
\label{product}
\bE (\int_a^b f(t) dw(t)\cdot \int_a^b g(t) dw(t))= \bE \int_a^b f(t)\cdot g(t)\,dt\,.
\end{equation}
({\bf Integral of the product is the product of integrals.})

\subsection{Stochastic differential.}
Let $b(t)\in M_2[a,b]$, and $a(t)$ be measurable with respect to $\cF_t$ for every $t$, and
$$
\int_a^b |a(t)|\,dt <\infty\,.
$$
Suppose $\zeta(t)$ is a random process such that for all $t_1, t_2$ such that
$a\le t_1\le t_2\le b$
$$
\zeta (t_2)-\zeta(t_1) =\int_{t_1}^{t_2} a(t)\,dt+ \int_{t_1}^{t_2} b(t)\, dw(t)\,.
$$
Then we write the above line  as {\bf stochastic differential}:
$$
d\zeta(t)= a(t)dt + b(t)dw(t)\,.
$$
\noindent{\bf Remark.} If $a=0$ this integral is a martingale (obviously) on the filtration $\{\cF_t\}_{t>0}$ of sigma algebras  generated by Brownian motions.

\subsection{It\^o' formula.}

 Let $\zeta$ have the stochastic differential in the sense above and let $u(t,x)$ be a (several times) smooth function. Consider new process
 $$
 \eta(t):= u(t, \zeta(t))\,.
 $$
 
 \begin{theorem}
 \label{itoth}
 Then $\eta$ also has stochastic differential and
 $$
 d\eta(t) = [ u'_t(t, \zeta(t) + u'_x(t,\zeta(t)) a(t) + \frac12 u''_{xx}(t,\zeta(t))\cdot b^2(t)]\,dt + u'_x (t,\zeta(t))\cdot b(t)\cdot dw(t)\,.
 $$
 \end{theorem}
 
 \begin{proof}
 The proof is quite subtle. See \cite{GS}, \cite{We}.
 
 \end{proof}

 Matrix It\^o's formula also exists and will be used. Let $a$ be $m\times 1$ column of processes, $\sigma$ is a $m\times k$ matrix of processes (with entries in $M_2[a,b]$). Let $W(t)$ be a column of $k$ independent Brownian motion.
 Let $\zeta(t)$ be a $m\times 1$ process with stochastic differential
 $$
 d\zeta(t) =a(t)\,dt+ \sigma\,dW(t)\,.
 $$
 Let $u(t, x)$ be a smooth function, where $x\in \R^m$.
 Let $\eta(t)= u(t, \zeta(t))$. Then $\eta$ also has stochastic differential, and matrix It\^o's formula gives
 \begin{equation}
 \label{itom}
 d\eta(t) = [\pd u/\pd t +  \nabla_x u (t, \zeta)\cdot a(t) +\frac12 \text{trace} (\sigma\,H_u(t,\zeta)\sigma^*)]\,dt +
 \nabla_x u\cdot\sigma  dW(t)\,. 
 \end{equation}
 Here $\cdot$ is the scalar product in $\R^m$.

 \subsection{Space-time Brownian motion.}
 
 Let us discuss Theorem \ref{itoth}. If $a=0$, then the process $\zeta$ is a martingale (see Remark before the theorem). However, it is quite unrealistic to expect that if we consider the composition of a non-linear function $u$ and a martingale, then we would get another martingale. And in fact, if $a=0$ the formula in Theorem \ref{itoth} becomes (if $a=0$)
 \begin{equation}
 \label{ar0}
 d\eta(t) = [ u'_t(t, \zeta(t)  + \frac12 u''_{xx}(t,\zeta(t))\cdot b^2(t)]\,dt + u'_x (t,\zeta(t))\cdot b(t)\cdot dw(t)\,,
 \end{equation}
 and the ``non-martingale" part (called {\bf drift}) in square brackets is very much present. But there is one very important exception.
 
 Suppose $f\in C_0^{\infty}$ and $u^f(t,x)$ is the heat extension of $f$, in other words, the solution of  the heat equation:
 \begin{equation}
 \label{heat}
 \left(\frac{\pd}{\pd t} -\frac12\frac{\pd^2}{\pd x^2}\right)u^f =0\,,\, u^f(0,x)=f(x)\,.
 \end{equation}
 
 Fix large positive $T$ and consider function of $(t, x)$ given by $u=u^f(T-t, x)$. We want to compose it with stochastic process as in Theorem \ref{itoth}, with $a=0, b=1$. Then we get the process
 $$
 \eta:=u^f(T-t, w_t)\,.
 $$
 It will be a martingale on $[0,T]$. In fact, we can use \eqref{ar0} to get
 $$
 d\eta(t) = [-\frac{\pd u^f}{\pd t} (T-t, w_t) + \frac12 \frac{\pd^2 u^f}{\pd x^2}(T-t, w_t)] \,dt+ \frac{\pd u^f}{\pd x}(T-t, w_t) dw_t\,,
 $$
 and by \eqref{heat} the drift term in the brackets disappears.
 
 If we work with heat extension for functions on $\R^k$ the same will be true. Now Brownian motion $W_t$ is $k$-dimensional (just $k$ independent Brownian motions) and $u^f$ is the solution of heat equation

\begin{equation}
 \label{heatk}
 \left(\frac{\pd}{\pd t} -\frac12\Delta\right)u^f =0\,,\, u^f(0,x)=f(x)\,.
 \end{equation}

Then we get the following stochastic differential
\begin{equation}
\label{heatmart}
d \,u^f(T-t, W_t) = \nabla_x u^f(T-t, W_t)\cdot dW_t\,,
\end{equation}
where $\cdot$ is the scalar product in $\R^k$.

We are interested now in the case of complex valued function $f$ on $\R^2$, so $k=2$. Thinking that gradient is always column vector and $W_t$ is $2$-dimensional row vector it is convenient to rewrite
\eqref{heatmart} as
\begin{equation}
\label{heatmart1}
f(W_T)-u^f(T, 0) = \int_0^T dW_t\cdot \begin{bmatrix}\pd_x\\\pd_y\end{bmatrix} u^f(T-t, W_t)\,.
\end{equation}

\noindent{\bf Definition.} The expressions $\int_0^T dW_t\cdot \begin{bmatrix}\pd_x\\\pd_y\end{bmatrix} u^f(T-t, W_t)$ will be called {\bf heat martingales}.

But we will need a bigger class, where heat martingales are supplemented by their {\bf martingale transforms}. The simplest martingale transforms are given by expressions
$$
\int_0^T dW_t\cdot A\begin{bmatrix}\pd_x\\\pd_y\end{bmatrix} u^f(T-t, W_t)\,,
$$
where $A$ is a fixed matrix not depending neither on $\omega$ (elementary event) nor on time $t$.

Consider a special matrix
\begin{equation}
\label{barA}
A:= \begin{bmatrix} 1, & i\\i,& -1\end{bmatrix}
\end{equation}
Then we get
$$
\int_0^T dW_t\cdot \begin{bmatrix}\pd_x+i\pd_y\\i(\pd_x+i\pd_y)\end{bmatrix} u^f(T-t, W_t)\,,
$$
which is
\begin{equation}
\label{dbar}
2\int_0^T dW_t\cdot \begin{bmatrix}\bar\pd\\i\bar\pd\end{bmatrix} u^f(T-t, W_t)\,.
\end{equation}
This is quite suggestive. In fact, denoting temporarily the Ahlfors--Beurling transform
$R_1^2-R_2^2+2iR_1R_2$ by symbol $AB$, we recall that $AB\bar\pd=\pd$. The following theorem holds.
\begin{theorem}
\label{stocAB}
 $$1)\,\,\lim_{T\rightarrow\infty}\bE (
\int_0^T dW_t\cdot \begin{bmatrix}\pd_x\\\pd_y\end{bmatrix} u^f(T-t, W_t)|W_T=z) =f(z)\,,
$$

 $$
 2)\,\,\lim_{T\rightarrow\infty}\bE (\int_0^T dW_t\cdot \begin{bmatrix}\pd_x+i\pd_y\\i(\pd_x+i\pd_y)\end{bmatrix} u^f(T-t, W_t)|W_T=z) =AB(f)(z)\,.
$$

\end{theorem}

\begin{proof}
Let us consider a test function $g$ and build a heat martingale $X(t), 0\le t\le  T$, by formula 1), but with $f$ replaced by $g$:
$$
X(t):=g(T,0)+\int_0^t dW_s\cdot \begin{bmatrix}\pd_x\\\pd_y\end{bmatrix} u^g(T-s, W_s)\,.
$$
 Let $Y(t), 0\le t\le T$, denote the martingale in formula 2):
 $$
 Y(t):=\int_0^t dW_s\cdot \begin{bmatrix}\pd_x+i\pd_y\\i(\pd_x+i\pd_y)\end{bmatrix} u^f(T-s, W_s)\,.
 $$
Then by ``rule" that the product of stochastic integrals is ``the integral of the product",  we get (below $k(t;x,y):=\frac1{2\pi\,t}e^{-\frac{x^2+y^2}{t}}$)
$$
2\pi\,T\bE ( Y(T)\cdot X(T)) =2\pi\,T\int_0^T\iint_{\R^2} \bar\pd u^f(T-t; x,y)\,\bar\pd u^{ g}(T-t; x,y)k(t;x,y)\,dxdy\,dt
$$
$$
=-2\pi\,T\int_0^T\iint_{\R^2} \bar\pd u^f(t; x,y)\,\bar\pd u^{ g}(t; x,y)k(T-t;x,y)\,dxdy\,dt\,.
$$
Notice that $2\pi\,T\,k(T-t;x,y)\rightarrow 1$ if $T$ goes to infinity. It is not then difficult to see that the last expression becomes very close to
$$
\int_0^T\iint_{\R^2} \bar\pd u^f(t; x,y)\,\bar\pd u^{ g}(t; x,y)\,dxdy\,dt\,,
$$
when $T$ goes to infinity.

Recall formula \eqref{1.13} and formula $AB=R_1^2-R_2^2+ 2iR_1R_2$.( Number $2$ in \eqref{1.13} should be dropped now as we are working with extensions with respect to $\frac{\pd}{\pd t} -\frac12\Delta$ unlike before formula \eqref{1.13}, where we worked with $\frac{\pd}{\pd t} -\Delta$.)  Combined they give us that the last expression would be equal to
$(AB(f), g)$ if the integration would be $\int_0^{\infty}...dt$ and not $\int_0^T...dt$. But as $T$ is large and $f,g$ are nice the ``error" goes to zero when $T$ goes to infinity.
So 
\begin{equation}
\label{ABf}
2\pi T\,\bE (Y(T)\cdot X(T)) = (AB(f), g)+o(1)\,.
\end{equation}
On the other hand, $X(T)= g(W_T)$ by \eqref{heatmart1} with $f$ replaced by $g$. Therefore, for any test function $g$
$$
2\pi T\,\bE (Y(T)\cdot g(W_T)) = 2\pi T\int_{\C} d\mu_T(z) \bE (Y(T)\,|\, W_T=z) g(z)\,,
$$
where $d\mu_T= \frac{1}{2\pi\,T} e^{-\frac{|z|^2}{T}}\,dm_2(z)$ is given by the density distribution of $W_T$. Now using the facts that $g$ is a nice test function and that
$2\pi T \frac{d\mu_T(z)}{d m_2(z)} \rightarrow 1$ pointwise and in a bounded fashion when $T\rightarrow\infty$ we obtain

$$
\int_{\C}  \bE (Y(T)\,|\, W_T=z) g(z)\, dm_2(z) = 2\pi T\,\bE (Y(T)\cdot g(W_T)) + o(1)\,.
$$
Comparing this with \eqref{ABf} we
 get the formula
\begin{equation}
\label{ABform}
AB(f)(z)=\lim_{T\rightarrow \infty} \bE (\int_0^T dW_t\cdot A\nabla_{x,y} u^f(T-t;W_t)| W_T=z)\,.
\end{equation}
Theorem is proved.
\end{proof}

\noindent {\bf Remark}. It is very easy to see now that for martingale $\{Y(t)\}_{0\le t \le T}$  constructed above
$$
 \|AB(f) \|^p_{L^p(\C, dm_2)} \le  \lim_{T\rightarrow \infty}2\pi\,T \,\bE |Y(T)|^p
$$
 for any $p$. It is a sort of averaging operator. Moreover, for martingale $\{X(t)\}_{0\le t \le T}$  we obviously have limiting equality
$$
 \|g\|^p_{L^p(\C, dm_2)} =  \lim_{T\rightarrow \infty}2\pi\,T \,\bE |X(T)|^p\,.
$$
This is trivial from \eqref{heatmart1}: just raise both part to the power $p$ and take the expectation (first conditioning over $W_T=z$, then integrating with respect to $d\mu_T(z)$) and use again the fact that 
$2\pi T \frac{d\mu_T(z)}{d m_2(z)} \rightarrow 1$. 

Now put $g=f$. We see that 
$$
\|AB(f)\|_p \le 2 M_p\|f\|_p
$$
 follows from $\bE |Y|^p\le M_p\bE|X|^p$ for martingales  $Y$, $X$. Notice that $Y$ is just the martingale transform of $X$ with the help of matrix $A$, whose norm is $2$. This explains the constant $2$ in the above display inequality. This is why we study below $X,Y$ and their relationship. 

\noindent{\bf Remark.} The reader can find many interesting  examples, references and explanations in recent review of Banuelos devoted to Burkholder's estimate: \cite{Ba}.

\subsection{Orthogonal (conformal) martingales.}

Introducing two martingales on the filtration of Brownian motion
$$
X(t):=\int_0^t dW_s\cdot \nabla_{x,y} u^f(T-s;W_s)\,, 0\le t\le T\,;
$$
$$
Y(t):=\int_0^t dW_s\cdot A\nabla_{x,y} u^f(T-s;W_s)=:A\star X(t)\,,0\le t\le T\,,
$$
and using the previous remark, we get that it might be a fruitful idea to look for a sharp {\bf martingale transform inequality}
\begin{equation}
\label{sharpli}
\bE\|A\star X (t)\|_p \le M_p\bE\|X\|_p\,.
\end{equation}

Let $f=\phi+i\psi$. Introduce the notations
$$
H_1^s=\begin{bmatrix} u^{\phi}_x(T-s;W_s), & u^{\psi}_x(T-s;W_s)\end{bmatrix}\,,
$$
$$
H_2^s=\begin{bmatrix} v^{\phi}_y(T-s;W_s), & v^{\psi}_y(T-s;W_s)\end{bmatrix}\,.
$$
$$
K_1^s=\begin{bmatrix} u^{\phi}_x(T-s;W_s)-u^{\psi}_y(T-s;W_s),& u^{\phi}_y(T-s;W_s)+u^{\psi}_x(T-s;W_s)\end{bmatrix}\,,
$$
$$
K_2^s=\begin{bmatrix} -u^{\phi}_y(T-s;W_s)-u^{\psi}_x(T-s;W_s),& u^{\phi}_x(T-s;W_s)-u^{\psi}_y(T-s;W_s)\end{bmatrix}\,,
$$
we can write complex martingale $X=X_1+iX_2$, $Y=A\star X=Y_1+iY_2$ in the form (below $dW_s$ is a $2$-row-vector)
\begin{equation}
\label{X}
\begin{bmatrix} X_1(t),& X_2(t)\end{bmatrix} =\int_0^t dW_s \begin{bmatrix}H_1^s\\ H_2^s\end{bmatrix}=\int_0^t (H_1^sdw_s^1+H_2^s dw_s^2)\,.
\end{equation}
\begin{equation}
\label{Y}
\begin{bmatrix} Y_1(t),& Y_2(t)\end{bmatrix} =\int_0^t dW_s \begin{bmatrix}K_1^s\\ K_2^s\end{bmatrix}=\int_0^t (K_1^sdw_s^1+K_2^s dw_s^2)\,.
\end{equation}

\noindent{\bf Properties of $H,K$.} Vector processes $H$ and $K$ are related by
\begin{equation}
\label{DS1}
\|K_1^s\|^2 + \|K_2^s\|^2 \le 4(\|H_1^s\|^2 + \|H_2^s\|^2)\,
\end{equation}
for all elementary events $\omega$ and all times $s$.

Relationship \eqref{DS1} is called {\bf differential subordination} of martingale $Y$ to martingale $2X$.

\begin{theorem} [Burkholder's theorem]
\label{2const}
If martingale $M$ is differentially subordinated to martingale $N$, then
$$
\bE\|M(t)\|_p \le (p^*-1)\bE\|N(t)\|_p\,, p^*:=\max(p, p/p-1)\,.
$$
The constant is sharp. 
\end{theorem}
In particular,
\begin{equation}
\label{DS2}
\bE\|A\star X\|_p \le 2(p^*-1)\bE\|X(t)\|_p\,.
\end{equation}
But the constant is not sharp! We followed the probabilistic proof in \cite{BaM}, which ``randomize" the idea of \cite{PV}. There is an analytic proof following \cite{PV} more directly, see \cite{NV}.

\bigskip

\noindent{\bf More properties of $H,K$.} Vector processes $K$ have extra properties:
\begin{equation}
\label{conf1}
K_1^s\cdot K_2^s =0\,,\, \|K_1^s\|=\|K_2^s\|
\end{equation}
for all elementary events $\omega$ and all times $s$.
Such martingales are called {\bf orthogonal or conformal}.

\begin{theorem} [Banuelos--Janakiraman's theorem]
\label{BJconst}
We make the exposition of \cite{BaJ} using the notations above. If martingale $M$ is differentially subordinated to martingale $N$, and martingale $M$ is conformal and $\|M(0)\|\le \|N(0)\|$ then
$$
\bE\|M(t)\|_p \le \sqrt{\frac{p(p-1)}{2}}\bE\|N(t)\|_p\,, p>2\,.
$$ 
\end{theorem}
In particular,
\begin{equation}
\label{DS2}
\bE\|A\star X\|_p \le \sqrt{2p(p-1)}\bE\|X(t)\|_p\,,\, p>2\,.
\end{equation}
But the constant is not sharp!

However, this inequality gives
$$
\|T\|_p \le \sqrt{2p(p-1)}, \,p>2\,.
$$ 
Interpolation between $p=2$ and large $p$ with this estimate, optimization in this large $p$, will give \eqref{BI3}.

\begin{proof}[The proof of Theorem \ref{BJconst}]
Our main tool will be formula \eqref{d2fb}.
Trivial renormalization shows that to prove Theorem \ref{BJconst} it is enough to prove that if $M, N$ are two martingales on the filtration of $2$-dimensional Brownian motion and $M$ is differentially subordinated to
$\sqrt{\frac{2(p-1)}{p}}\cdot N$, $p>2$, and $M$ is {\bf conformal} then
\begin{equation}
\label{MN}
\bE\|M(t)\|_p \le (p-1)\bE\|N(t)\|_p\,, p>2\,.
\end{equation}

Consider such $M,N$, and their $H_1, H_2, K_1, K_2$. We know that
\begin{equation}
\label{KH}
\|K\|^2\le \frac{2(p-1)}{p}\|H\|^2\,,
\end{equation}
where $\|K\|^2:=\|K_1\|^2 + \|K_2\|^2,\|H\|^2:=\|H_1\|^2 + \|H_2\|^2 $, and
$$
k_{11}\cdot k_{21} + k_{12}\cdot k_{22}=0\,.
$$
This and equality $\|K_1\|=\|K_2\|$ easily implies
\begin{equation}
\label{conf}
k_{11}\cdot k_{12} + k_{21}\cdot k_{22}=0\,.
\end{equation}

Let $V(M,N):= \|M\|^p-(p-1)\|N\|^p$, $p>2$, $\f(M, N) := p(1-1/p)^{p-1} (\|M\|+\|N\|)^{p-1}(\|M\|-(p-1)\|N\|)$.

We would like to prove that $\bE (V(M(t), N(t))\le 0$.
But it has been proved that
$V\le \f$. So it is enough to prove
\begin{equation}
\label{fMN}
\bE (\f(M(t), N(t)))\le 0\,.
\end{equation}

To prove \eqref{fMN} we use:
\begin{equation}
\label{int}
\f(M(t), N(t)) = \f(M(0), N(0)) + \int_0^t d\f (M(s), N(s))\,.
\end{equation}
To compute $\bE d\f$ we use  It\^o's formula, which of course involves Hessian $H_{\f}$. More precisely, $d\f(s)$ will involve
$$
(H_{\f}\begin{bmatrix} H_1^s\\K_1^s\end{bmatrix},\begin{bmatrix} H_1^s\\K_1^s\end{bmatrix}) +
$$
$$
(H_{\f}\begin{bmatrix} H_2^s\\K_2^s\end{bmatrix},\begin{bmatrix} H_2^s\\K_2^s\end{bmatrix})
$$ 
Now we look at formula \eqref{d2fb}, which gives
$$
d\f = p(1-1/p)^{p-1}(A+B+C+D)\,,\, A:=-p(p-1)(\|M(s)\|+ \|N(s)\|)^{p-2} (\|H\|_2^2 -\|K\|^2)\,,
$$
$$
B:= -p(p-2)(\|M(s)\|+ \|N(s)\|)^{p-1}\|M(s)\|^{-1}\left[\left(\frac{M_2 k_{11} -M_1 k_{12}}{\|M\|}\right)^2 + \left(\frac{M_2 k_{21} -M_1 k_{22}}{\|M\|}\right)^2\right]\,,
$$
where we need to recall that $M=M_1+iM_2$, $M_1, M_2$ being its real and imaginary parts. Part $C$ comes from the last part of formula \eqref{d2fb}, and, obviously, 
$$
C\le 0\,,
$$
$$
D= ...dw^1_s + ...dw^2_s\,,
$$
where $...$ involve functions of $k_{ij}^s, h_{ij}^s$ and $\nabla\f (M(s), N(s))$. This shows that $\int_0^t D(s)$ is a martingale starting at $0$ and so
\begin{equation}
\label{D}
\bE \int_0^t  D(s) =0\,.
\end{equation}

We open the brackets in $B$, use \eqref{conf}, and the fact that
$k_{11}^2 + k_{21}^2 = k_{12}^2 + k_{22}^2$, to get
$$
B\le -p(p-2) (\frac12 \|K\|_2^2) (\|M\|+\|N\|)^{p-2}\,.
$$
Now 
$$
A+B =-p(\|M\|+\|N\|)^{p-2} [(p-1) \|H\|_2^2  -\frac{p}2 \|K\|_2^2]\le 0\,,
$$
if \eqref{KH} is valid. Term $C$ is non-positive. Term $D$ disappears after integration $\bE\int_0^t$, As a result we come to (see \eqref{int}):
$$
\bE\int_0^t d\f(M(s), N(s)) =\bE \f(M(0), N(0))=\f(M(0), N(0))\le 0\,,
$$
because if $x:=\|M(0)\|\le   y:=\|N(0)\|$ then $\f(x,y)\le 0$, which is obvious from the formula for $\f$.

\end{proof}

As we already mentioned, this proves \eqref{BI3}. To prove \eqref{BI4} one needs even more careful stochastic analysis, and we leave this for the next
round of lectures somewhere in the future.

\section{Bellman function of Stochastic Optimal Control problems}
\label{SOC}

Let $W_s$ be $d_1$ dimensional Brownian motion. Let $x(t)$ is a $d$-dimensional process given by
\begin{equation}
\label{sdint}
x(t) = x + \int_0^tb(\al(s), x(s))\,ds+ \int_{0}^t \sigma(\al(s), x(s))\,dW_s \,,
\end{equation}
in other words the process starts at $x\in \R^d$ and satisfies a stochastic differential equation
$$
dx(t)= b(\al(t), x(t))\,dt + \sigma(\al(t), x(t))\,dW_t\,,
$$
where $\al$ is a $d_2$-dimensional control process, we can choose it ourselves, but it must be adapted, that is $\al(s)$ has to be measurable with respect to sigma algebra $\cF_s$ generated by $W_{t}, 0\le t\le s$. Also values of the process $\al$ are often restricted: $\al(s, \omega) \in A\subset \R^{d_2}$.

Matrix function $\sigma$ is smooth and $d\times d_1$-dimensional, and $b$ is a smooth column function of size $d$. Everything happens in $\Omega\subset \R^d$ (often =$\R^d$). 

The choice of adapted process $\al(s)$ gives us different motions, all started at the same initial $x\in \R^d$.

This is a ``broom" of motions, hidden elementary even $\omega$ gives ``one stem of a broom".

\bigskip

Suppose we are given the profit function $f(\al, x)$, meaning that on a trajectory of $x(t)$, for the time interval $[t, t+\Delta t]$, the profit is $f(\al(t), x(t)) +o(\Delta t)$. So on the whole trajectory we earn
$$
\int_0^{\infty} f(\al(t), x(t))\,dt\,.
$$
We are also given the pension--we call it bonus function $F$--how much one is given at the end of the life. We want to choose a control process $\al=\al(s)$ to {\bf maximize the average profit}:

\begin{equation}
\label{avprofit}
v^{\al}(x):= \bE \int_0^{\infty} f(\al(t), x(t))\, dt + \limsup_{t\rightarrow\infty}\bE F(x(t))\,.
\end{equation}
If $b=0$ and $F$ is convex then one ca write $\lim$ instead of $\limsup$.

\medskip

{\bf The optimal average gain}, or
\begin{equation}
\label{v}
v:=\sup_{\al} v^{\al}(x)
\end{equation}
is called the Bellman function of stochastic optimal control problem \eqref{sdint}, \eqref{avprofit}.

\medskip

Usually the analysis consists of

a) writing Bellman PDE on $v$;

b) solving it;

c) using ``verification theorem", which says that under certain conditions on data $\sigma, b, F, f, \Omega, A$ the classical solution of Bellman PDE is exactly $v$ from \eqref{v}.

\bigskip

\subsection{Writing Bellman PDE.}

This consists of a) It\^o's formula, b) Bellman's principle of dynamic programming.

Using It\^o's formula \eqref{itom} we get
$$
dv(x(s)) = \sum_{k=1}^d \frac{\pd v}{\pd x_k} (x(s)) \sum_{j=1}^{d_1} \sigma_{kj} (\al(s), x(s)) \, dw_s^j+
$$
$$
\sum_{k=1}^d \frac{\pd v}{\pd x_k} (x(s)) b_k(\al(s), x(s)) \,ds+
$$
$$
\frac12 \sum_{i,j=1}^d \frac{\pd^2 v}{\pd x_i\pd x_j} (x(s)) a^{ij} (\al(s), x(s))\,,
$$
where 
$$
a^{ij} (\al, x):= \sum_{k=1}^{d_1}\sigma_{ik}(\al, x)\sigma_{kj}(\al, x)
$$
is $i,j$ matrix element of $d\times d$ matrix $\sigma\sigma^*$.

Introduce two linear differential operators with non-constant coefficients:
$$
\cL_1(\al, x):= \sum_{k=1}^d b_k(\al, x) \frac{\pd}{\pd x_k}\,,
$$
$$
\cL_2(\al, x):= \sum_{i,j=1}^d a_{ij}(\al, x) \frac{\pd^2}{\pd x_i\pd x_j}\,,
$$
and
$$
\cL(\al,x):= \cL_1(\al, x) + \cL_2(\al, x)\,.
$$
Let us hit our formula for $dv(x(t))$ above by the expectation $\bE$, then the first line becomes $0$, and we get
$$
\bE \left[ \frac{d}{dt} v(x(t))\right] = \bE [\cL(\a(t), x(t))v ](x(t))\,.
$$
Or
\begin{equation}
\label{itov}
\bE v(x(t)) = v(x) +\bE\int_0^t  [\cL_1(\a(s), x(s))+\cL_2(\a(s), x(s))]v (x(s))\,ds\,.
\end{equation}

\bigskip

Now we need the second ingredient to write the Bellman equation: {\bf the Bellman principle= dynamic programming principle}. It is in this next equality:
$$
v(x)=\sup_{\al}\bE [\int_0^{\infty} f(\al(t), x(t))\, dt +\limsup_{t\rightarrow \infty}\dots] 
$$
\begin{equation}
\label{BP}
=\sup_{\al}\bE [\int_0^{t} f(\al(t), x(t))\, dt + v(x(t))]\,,\,\, \forall t>0\,.
\end{equation}

A minute though shows that this reflects the stationarity of Brownian motion and the fact that to be perfect one has to be perfect every second.

\medskip

Now plug $\bE v(x(t))$ from \eqref{itov} into \eqref{BP}. We get

$$
0= \sup_{\al} \bE [ \int_0^{\infty} f(\al(t), x(t))+ \cL(\al(t), x(t))v (x(t))]\,dt\,,\,\, \forall t>0\,.
$$
Divide by $t$ and tend $t$ to zero. We ``obtain" {\bf Bellman equation}:

\begin{equation}
\label{SBE}
\sup_{\al\in A} [ (\cL(\al, x) v)(x) + f(\al, x)]=0\,.
\end{equation}

Positivity (usually present) of $f$ and convexity (usually present) of $F$ imply (if there is no drift, that is if $b(\al, x)=0$)  {\bf obstacle condition}:

\begin{equation}
\label{OC}
v(x)\ge F(x)\,,\,\, \forall x\in \Omega\,.
\end{equation}
Often it becomes {\bf boundary condition}:
\begin{equation}
\label{OC}
v(x)= F(x)\,,\,\, \forall x\in \pd\Omega\,.
\end{equation}
The definition of $v$ in domain (not in the whole $\R^d$) should be slightly changed. The integration of profit function now is not from zero to infinity, but from zero to the stopping time of the first hit of $\pd \Omega$ by the trajectory $x(t)$.

See details of obtaining \eqref{SBE} in the beautiful book of N. Krylov \cite{Kr}.

In  applications one is also interested in supersolutions of the Bellman equation \eqref{SBE}~:
\begin{equation}
\label{4}
\left\lbrace\begin{array}{ll}& \mathop{\sup}\limits_{\alpha \in A} [ \mathcal{L}(\alpha, x) V(x)+f({\alpha},x) ]\leq 0 , x\in \Omega \ ,\\& V(x) \geq F(x)\ ,\ x\in \Omega\ .\end{array}\right.
\end{equation}
{\bf Lemma~:} Let $V$ solves (4) and let $v$ be the Bellman function,then $V \geq v$ in $\Omega$.

\medskip

\noindent{\bf Proof~:} Equation \eqref{SBE} states that $-\mathcal{L}(\alpha, x)V(x)\geq f(\alpha, x)$. Using \eqref{itov}  for $V$ and then \eqref{4}, one gets
$$
V(x) =\mathbb{E} V(x(t)) - \mathbb{E} \int^{t}_{0} (\mathcal{L}(\alpha(s), x(s)) V) (x(s)) ds
$$
$$\geq \mathbb{E} F(x(t)) +\mathbb{E} \int^{t}_{0} f(\alpha(s), x(s)) ds.
$$
Writing $\mathop{\overline{\lim}}\limits_{t\to \infty}$ of both parts,we get $V(x) \geq v^{\alpha}(x)$. It rests to take the supremum over the control process $\alpha$. 

\subsection{Special matrices $\sigma$ bring us to Harmonic Analysis.}
Let us consider a very simple matrix $\sigma$ not depending on $x$:
\begin{equation}
\label{HAd1}
d_1=1\,,\,\, \sigma(\al, x) = \begin{bmatrix} \al_1\\ \vdots \\ \al_d\end{bmatrix}=: \al\,.
\end{equation}

If on the top of that $b=0$ then operator $\cL$ just involves Hessian matrix $H_v$ of function $v$:
$$
(\cL(\al)v) (x) =\frac12 \sum_{i,j=1}^d\frac{\pd^2}{\pd x_i\pd x_j} v(x)=\frac12( H_v(x)\al, \al) \,.
$$
We claim that this is the generic case of Harmonic Analysis problems in $\R^1$. Equation \eqref{SBE} becomes
\begin{equation}
\label{HASBE}
\left\lbrace\begin{array}{ll}& \mathop{\sup}\limits_{\alpha \in A} [ \frac12\langle H_v(x)\al, \al\rangle +f({\alpha},x) ]= 0 , x\in \Omega \ ,\\& v(x) \geq F(x)\ ,\ x\in \Omega\ .\end{array}\right.
\end{equation}

If $b\neq 0$ then we just add the first order differential operator (called {\bf drift}):
\begin{equation}
\label{HASBEdr}
\left\lbrace\begin{array}{ll}& \mathop{\sup}\limits_{\alpha \in A} [ \frac12\langle H_v(x)\al, \al\rangle+\sum_{k=1}^d b_k(\al, x) \frac{\pd}{\pd x_k}v(x) +f({\alpha},x) ]= 0 , x\in \Omega \ ,\\& v(x) \geq F(x)\ ,\ x\in \Omega\ .\end{array}\right.
\end{equation}

Harmonic analysis on $\R^2$ ``becomes" the analysis of the following Bellman equation (and this is exactly what we did in the sections above devoted to the analysis of the Ahlfors--Beurling operator):
\begin{equation}
\label{HAd1}
d_1=2\,,\,\, \sigma(\al, x) = \begin{bmatrix} \al_{11}&\al_{12}\\ \vdots  &\vdots\\ \al_{d1}&\al_{d2}\end{bmatrix}=: \al\,.
\end{equation}

\bigskip

{\bf Conformal restrictions}:
Matrix $\al$ can have restrictions $\al\in A$ of the type that the first row is orthogonal to the second row and that the norms of the rows are equal. The reader can notice that these are {\bf Cauchy--Riemann} conditions, and the corresponding solution of \eqref{sdint} will be a conformal martingale (again if $b=0$). Bellman equation becomes

\begin{equation}
\label{HASBEdrD2}
\left\lbrace\begin{array}{ll}& \mathop{\sup}\limits_{\alpha \in A} [ \frac12\text{trace}(\al^* H_v(x)\al)+\sum_{k=1}^d b_k(\al, x) \frac{\pd}{\pd x_k}v(x) +f({\alpha},x) ]= 0 , x\in \Omega \ ,\\& v(x) \geq F(x)\ ,\ x\in \Omega\ .\end{array}\right.
\end{equation}

\bigskip

\noindent{\bf Remarks.} 1) This is (exactly as \eqref{HASBEdr}) very non-linear (actually an example of so-called {\bf fully non-linear}) equation of the second order.

2) This equation is much more difficult to analyze than \eqref{HASBEdr}. On the other hand, we can easily notice that conformal restrictions  on $\al$ makes clear that Hessian of $v$ should be replaced by Laplacian of $v$ (or some kind of semi-Laplacian-semi-Hessian).

\section{Examples showing almost perfect analogy between Stochastic Optimal Control and Harmonic Analysis}
\label{analogy}

\subsection{$A_{\infty}$ weights and associated Carleson measures. Buckley's inequality.}
We call a nonnegative function on $\mathbb{R}$ an $A_{\infty}$weight (dyadic $A_{\infty}$ weight actually) if

\begin{equation}
\label{5}\langle w\rangle_{J} \leq C_{1} e^{\langle \log w\rangle_{J}} \ ,\,\forall J \in \mathcal{D}\ .
\end{equation}
Here $\mathcal{D}$ is a dyadic lattice on $\mathbb{R}$, $\langle\  \cdot\ \rangle_{J}$ is the averaging over $J$.We are going to illustrate our use of Bellman function technique by a collection of examples, the first of which is the result of Buckley that can be found (along with ``continuous analogs'') in the paper of Fefferman-Kenig-Pipher \cite{FKP}.

\medskip

\begin{theorem}
\label{Buckley} Let $w\in A_{\infty}$. Then
\begin{equation}
\label{6}\forall I \in \mathcal{D} \ , \frac{I}{|I|}\mathop{\sum}\limits_{\ell \subseteq I\,,\,  \ell \in \mathcal{D}} \left( \frac{\langle w\rangle_{\ell_{+}}- \langle w\rangle_{\ell_{-}}}{\langle w \rangle_{\ell}}\right)^{2} |\ell| \leq C_{2}\ ,   
\end{equation}
Where $C_{2}$ depends only on $C_{1}$ in \eqref{5}. Here $\ell_{\pm}$ are right and left sons of $\ell \in\mathcal{D}$.
\end{theorem}

\bigskip

\centerline{\bf Who moves~?}

\medskip

$$x_{1} , x_{2} =\langle w\rangle_{J} \ ,\ \langle \log w \rangle_{J}
$$
$$\alpha_{1} = \langle w\rangle_{\text{son of}J} - \langle w\rangle_{J}\Rightarrow | \alpha_{1}| =\frac{1}{2} | \langle w\rangle_{J_{-}} -\langle w \rangle_{J_{+}}|
\,.$$

Function of profit can be read off \eqref{6} if one notices that$\frac{1}{|I|} \mathop{\sum}\limits_{\ell \subseteq I\,,\, \ell \in  \mathcal{D}} \cdots$ is the average over the lines of life. Each line of life initiates at I and then proceeds to $I_{\varepsilon_{1}} (\varepsilon_{1} = + 1$ or$\varepsilon_{1} = - 1)$, then to $I_{\varepsilon_{1}\varepsilon_{2}}(\varepsilon_{2} = + 1$ or $\varepsilon_{2}= -1)$, etc.

\medskip

Thus $\frac{1}{|I|} \mathop{\sum}\limits_{\ell \subseteq I \,,\, \ell \in  \mathcal{D}} \cdots$ plays the role of $\mathbb{E}\int^{\infty}_{0}\cdots$. This allows us to choose the correct profit function
$$f({\alpha},x) = \frac{4\alpha^{2}_{1}}{x_{1}^{2}}\,.
$$

Bonus function $F\equiv 0$ here.
Bellman equation reads now
\begin{equation}
\label{7}\sup_{\alpha = (\alpha_{1},\alpha_{2})} \left[ \langle H_v\alpha,\alpha \rangle + \frac{8\alpha_{1}^{2}}{x_{1}^{2}}\right] = 0
\end{equation}
to be solved in
\begin{equation}
\label{8}
\Omega = \left\{ (x_{1},x_{2}): 1 \leq x_{1} e^{-x_{2}} \leq c_{1}\right\}
\end{equation}
with the obstacle condition
\begin{equation}\label{9}v(x) \geq 0 \quad \forall x\in \Omega\ .
\end{equation}

{\bf Compare with} \eqref{HASBE}!

\bigskip

\subsection{A two-weight inequality}
$$
\forall J\in \mathcal{D}\,\, \langle u\rangle_{J} \langle v\rangle_{J}\leq 1 \Rightarrow \forall I \in \mathcal{D}
$$
$$
\frac{1}{|I|} \sum_{\ell \subseteq I \,,\, 
\ell\in\mathcal{D}} | \langle u\rangle_{\ell +} - \langle u \rangle_{\ell} | | \langle v\rangle_{\ell +} - \langle v\rangle_{\ell -} |\ | |\ell |
$$

$$\leq C \langle u\rangle_{I}^{1/2} \langle v \rangle_{I}^{1/2}.$$

\bigskip

\centerline{\bf Who moves~?}

$$x_{1} ,x_{2} = \langle u\rangle_{J} \ ,\ \langle v\rangle_{J} .$$As in the previous problem $f^{\alpha}(x)$ is easy to find~:

$$
f(\al,x)) = 4 |\alpha_{1} |\ |\alpha_{2} | \,.
$$
Bonus function $F\equiv 0$ here again.
Bellman equation
$$
\sup_{\alpha = (\alpha_{1},\alpha_{2})\in \R^2} \left[ \langle  H_v \alpha  , \alpha \rangle + 8 |\alpha_{1} | \ | \alpha_{2}|\right]=0 _ , v\geq 0\ 
 \text{in}\  \Omega = \{ x = (x_{1}, x_{2}): 0 \leq x_{1} ,x_{2} ;x_{1}x_{2} \leq 1\}\,.
 $$

{\bf Compare with} \eqref{HASBE}!

\bigskip

\subsection{John-Nirenberg inequality~: Bellman equation with a drift  but with $f(\al,x)\equiv 0$.}

$$\forall J \in\mathcal{D} \langle |\varphi - \langle \varphi\rangle_{J} |^{2} \rangle_{J} \leq \delta \Rightarrow \forall I \in \mathcal{D}$$

$$\langle e^{\varphi}\rangle_{I} \leq C_{\delta} e^{\langle  \varphi\rangle_{I}}\ .$$

\bigskip

\centerline{\bf Who moves~?}
$$
x_{1} =\langle \varphi \rangle_{J} \ ,\ x_{2} = \langle |\varphi  -\langle\varphi \rangle_{J}|^{2} \rangle_{J} =
 $$
$$
= \frac{1}{|J|} \sum_{I\subseteq J\,,\, I\in \mathcal{D}} \left\{  \frac{\langle \varphi\rangle_{I_{+}} - \langle \varphi    \rangle_{I-}}{2} \right\}^{2} |I|\,.
$$

Notice that $(t=n)$:
$$
x_{2}^{t} -\mathbb{E} (x_{2}^{t+1} |x^{t})=x_{2} -\frac{x_{2}^{+}+x_{2}^{-}}{2} = \left( \frac{x_{1}^{+}    -x_{1}^{-}}{2}\right)^{2} = (\alpha_{1}^{t})^{2}\,.
 $$
 But
 $$
x_{1}^{t} -\mathbb{E} (x_{1}^{t+1} |x^{t})=x_{1} -\frac{x_{1}^{+}+x_{1}^{-}}{2} = 0\,.
 $$
 On the other hand,
 $$x^{t+1} =x^{t} + \int_{t}^{t+1} \sigma d w^{s} + \int_{t}^{t+1} b ds \,.
 $$
 Thus drift $b$ stands for $\mathbb{E} (x^{t+1}|x^{t})-x^{t}$ (in the case of discrete time). Therefore, $b(\alpha ,x) = \begin{pmatrix}0 \\  -\alpha_{1}^{2}\end{pmatrix}$ in our case.
 Notice that $f(\al,x) \equiv 0$ as there is no $\frac{1}{|I|}\mathop{\sum}\limits_{\ell \subseteq I}$... in the functional. Bellman equation in this case has a form
 $$
 \sup_{\alpha = (\alpha_{1} , \alpha_{2})} \left[ \frac{1}{2} \langle  d^{2} v \alpha , \alpha \rangle - \frac{\partial v}{\partial x_{2}}  \alpha_{1}^{2} \right] = 0\,.
 $$

{\bf Compare with} \eqref{HASBEdr}!

\bigskip

 In other words~:
 \begin{equation}
 \label{13}
 \begin{pmatrix}\frac{\partial^{2}v}{\partial x_{1}^{2}} - 2 \frac{\partial  v}{\partial x_{2}} & \frac{\partial^{2} v}{\partial x_{1} \partial  x_{2}}\\ &\\ \frac{\partial^{2}v}{\partial x_{1} \partial x_{2}} &  \frac{\partial^{2}v}{\partial x_{2}^{2}}\end{pmatrix} \leq 0\,,\, \det\begin{pmatrix}\frac{\partial^{2}v}{\partial x_{1}^{2}} - 2 \frac{\partial  v}{\partial x_{2}} & \frac{\partial^{2} v}{\partial x_{1} \partial  x_{2}}\\ &\\ \frac{\partial^{2}v}{\partial x_{1} \partial x_{2}} &  \frac{\partial^{2}v}{\partial x_{2}^{2}}\end{pmatrix} =0\,.
 \end{equation}
 in $\Omega_{\delta} = \{ x = (x_{1},x_{2}), x_{1}\in \mathbb{R}, 0\leq x_{2} \leq \delta\}$.
 The obstacle condition is
 \begin{equation}
 \label{14}
 v(x) \geq F(x) \equiv e^{x_{1}} \ \text{in}\  \Omega_{\delta}
 \end{equation}

 Denote $B^{d}_{\delta}$ the dyadic Bellman function of a corresponding problem. There are many solutions of the above equation in $\Omega_{\delta}$ which satisfy   the obstacle condition $\ge e^{x_1}$ in  $\Omega_{\delta}$ and even satisfying the boundary condition $= e^{x_1}$ on $x_2=0$.  These are
 $$
 \varphi_{\varepsilon ,q} (x_{1},x_{2} )= q\frac{(1-\sqrt{\varepsilon -    x_{2}})}{1-\sqrt{\varepsilon}} e^{x_{1} + \sqrt{\varepsilon    -x_{2}}- \sqrt{\varepsilon}} , \delta \leq \varepsilon < 1 , q\geq1\,.
 $$
 One can compute $v_{\delta}$-the smallest solution of the above equation satisfying the obstacle condition. 
 $$
  v_{\delta} = \frac{1-\sqrt{\delta    -x_{2}}}{1-\sqrt{\delta}} e^{x_{1}+\sqrt{\delta -x_{2}}  -\sqrt{\delta}}\,.
  $$
 This is not $B^d_{\delta}$! In fact, $B^d_{\delta}> v_{\delta}$. However, $v_{\delta}$  is {\bf the Bellman function for non-dyadic John--Nirenberg inequality}!!!
 The rest is in Vasyunin's lectures and in \cite{SVa}.
 
 \bigskip

\subsection{Burkholder-Bellman function}$$\forall I\in \mathcal{D}\,\, | \langle g\rangle_{I+} - \langle g\rangle_{I-}| \leq | \langle f\rangle_{I+} - \langle f\rangle_{I-} |$$
$$\Rightarrow \forall I\in\mathcal{D}\  \text{such that}\  |\langle g\rangle_{I} | \leq |\langle f\rangle_{I} |$$one has
$$\langle |g|^{p} \rangle_{I}\leq (p-1)^{p} \langle |f|^{p} \rangle_{I}\,\ p\geq 2\ .$$
The constant $(p-1)^{p}$ is sharp. This is a famous theorem of Burkholder which he proved by constructing the corresponding Bellman function. He found it by solving a corresponding Bellman PDE - a     complicated one. We would like to show a simple ``heuristic'' method of solution.

\bigskip

\centerline{\bf Who moves~?} 

$$x_{1} = \langle g \rangle_{J} \ ,\ x_{2} = \langle f\rangle_{J} \,\, x_{3} = \langle |f|^{p} \rangle_{J} \ .$$
Our rules say that $f(\al,x)=0$, $\mathbb{E} F(x_{1}^{t},x_{2}^{t}, \,x_{3}^{t})\thickapprox \mathbb{E}|g|^{p}$. Denoting by $\mathcal{F}_{t}$ the $\sigma$-algebra generated
by dyadic subintervals of $I$ of length $2^{-n} |I| , t = 2^{n}$, we can write $\mathbb{E} |g|^{p} \thickapprox \mathbb{E} |(\mathbb{E}x_{1} |\mathcal{F}_{t})|^{p} =\mathbb{E}|x_{1}^{t}|^{p}$ which gives us the correct bonus function $F(x_{1} ,x_{2} ,x_{3}) =|x_{1}|^{p}$.
Notice that $A =\{ \alpha = (\alpha_{1} ,\alpha_{2},\alpha_{3}) :|\alpha_{1} | \leq |\alpha_{2}| \}$ now.

This is because $|\alpha_{1}| =\frac{1}{2} | \langle g\rangle_{J+} -\langle g\rangle_{J-}| , |\alpha_{2}| =\frac{1}{2} | \langle f\rangle_{J+} - \langle f\rangle_{J-}|$, and we are given that the first quantity is always majorized by the second one. 

So we have the Bellman equation
$$
\mathop{\sup}\limits_{|\alpha_{1}|\leq |\alpha_{2}| \,,\, \alpha_{3}}\langle H_v\alpha ,\alpha \rangle = 0
$$
in $\Omega =\{ x: (x_{1},x_{2}, x_{3}) : |x_{2}|^{p} \leq x_{3}\}$(convex), with obstacle condition
$$
v(x_{1}, x_{2}, x_{3}) \geq |x_{1}|^{p}\,.
$$

{\bf Compare with} \eqref{HASBE}!

\bigskip

This example is interesting because we have a non-trivial set of restrictions $A$ for ``control"  $\al$.

Solutions were given by Burkholder \cite{Bu1} (see also \cite{Bu2}--\cite{Bu7}) and also (a different approach using Monge--Amp\`ere equation) can be found in \cite{VaVo}.
See also a very interesting review \cite{Ba}.

An interesting Bellman function built by the use of Monge--Amp\`ere equation can be also found in \cite{VaVoMA}, \cite{VaVoNotes}.

\section{The technique of laminates, Bellman function, and estimates of singular integrals from below}
\label{lam}

\noindent{\bf Definition.} Laminate  on $M^s_{2\times 2}$ is a positive finite measure on symmetric real matrices $M^s_{2\times 2}$ such that
\begin{equation}
\label{la1}
f(A) \ge \int f(A+M) \, d\nu(M)
\end{equation}
for all rank $1$ concave functions $f$.

\begin{theorem}
\label{JB}
Any laminate  on $M^s_{2\times 2}$ can be approximated weakly by the push forward of Lebesgue measure on the plane by the Hessian of smooth compactly supported functions, in other words,
 for {\bf any} good $F$
$$
\int F(M)\, d\nu(M)\approx \int_{\R^2} F (Du)\,dxdy\,,
$$
where $Du(x,y):=\begin{bmatrix} u_{xx} & u_{xy}\\u_{yx} & u_{yy}\end{bmatrix}$.
\end{theorem}

\noindent{\bf Observation.} Laminates supported by diagonal matrices $\begin{bmatrix} X& 0\\0 & Y\end{bmatrix}$are just exactly exactly the measures on $\R^2$ such that ($z=(X, Y)$)
\begin{equation}
\label{la2}
f(a)\ge \int_{\C} f(a+z)\, d \nu(z)
\end{equation}
for all {\bf bi-concave} (meaning separately concave in $X$ and $Y$) function $f$.

\bigskip

\noindent{\bf Definition.} $(\int X\,d\nu, \int Y\,d\nu)$ is called baricenter of a laminate $\nu$ supported on diagonal matrices.

\bigskip

Fix $p>2$ and $p_{\eta}=p+\eta, \eta>0$. Put 
$$
s_0:= 1- \frac2{p_{\eta}}\,,\, K:= \frac{p_\eta}{p_\eta-2}\,,\, p-\eta-1=\frac{K+1}{K-1}\,.
$$

We are going to construct very interesting laminates supported on
$$
Y= KX\,,\, Y+\frac1K X\,.
$$

\vspace{.2in}

Fix $p\ge 2$, fix small $\eta>0$, put
$$
p_\eta:=p+\eta\,,
$$
\begin{equation}
\label{pks0}
s_0:= 1-\frac{2}{p_\eta}, \, K:= \frac1{s_0}= \frac{p_\eta}{p_\eta-2},\, p_\eta=\frac{2K}{K-1},\, p_\eta-1=\frac{K+1}{K-1}
\end{equation}

We are going to present an interesting laminate with baricenter $(1,1)$ supported by lines
$$
L_K:\,\, Y=KX\,,\,\,\, L_{1/K}: \,\, Y=\frac1K X\,.
$$

Let $f$ be a bi-concave function and
\begin{equation}
\label{p}
f(z) = O(|z|^p),\, z\rightarrow\infty\,.
\end{equation}

Then concavity in horizontal variable gives

\begin{equation}
\label{c1}
f(t,t+h) \ge \frac{t-\frac1K (t+h)}{t+h-\frac1K (t+h)}f(t+h, t+h) +\frac{h}{t+h-\frac1K (t+h)} f(\frac1K( t+h), t+h)\,.
\end{equation}

Rewrite it as

\begin{equation}
\label{ch}
f(t+h,t+h) \le \frac{t+h-\frac1K (t+h)}{t-\frac1K (t+h)}f(t, t+h) -\frac{h}{t-\frac1K (t+h)} f(\frac1K( t+h), t+h)\,.
\end{equation}

The concavity in vertical variable gives

\begin{equation}
\label{cv}
f(t,t) \ge \frac{t-\frac1K t}{t-\frac1K t +h}f(t, t+h) +\frac{h}{t-\frac1K t +h} f(t, \frac1K t)\,.
\end{equation}

From \eqref{ch}, \eqref{cv} we obtain (of course we divide by $h$, and next, we will make $h$ tend to $0$)

$$
\frac{f(t+h, t+h)-f(t,t)}{h} \le \frac1h\bigg[\frac{t+h-\frac1K (t+h)}{t-\frac1K (t+h)}-1 +1 -\frac{t-\frac1K t}{t-\frac1K t +h}\bigg]f(t,t+h) -
$$
$$
\frac1{t-\frac1K(t+h)} f(\frac1K (t+h), t+h) - \frac1{t-\frac1K t +h} f(t,\frac1K t)\,.
$$
Make $h\rightarrow 0$. Then
\begin{equation}
\label{diffeq}
f'(t,t) -\frac{2K}{K-1} \frac{f(t,t)}{t} \le  -\frac{K}{K-1} \frac{f(\frac1K t,t)}{t}   -\frac{K}{K-1} \frac{f(t,\frac1K t)}{t} \,.
\end{equation}

We recall \eqref{pks0} and multiply by $1/t^{p_\eta}$. Notice that after that $LHS= \bigg(\frac{f(t,t)}{t^{p_\eta}}\bigg)'$. We integrate from $1$ to $\infty$ and use \eqref{p} to forget the term at infinity.
Then we obtain for any bi-concave function on the plane
\begin{equation}
\label{one}
-f(1,1) \le -\frac{K}{K-1} \int_1^\infty f(\frac1K t, t) \frac{dt}{t^{p_\eta+1}} -\frac{K}{K-1} \int_1^\infty f( t, \frac1K t)\frac{dt}{t^{p_\eta+1}}
\end{equation}

\bigskip

Introduce $\nu_{K,\eta}$:
$$
\int_{\R^2} \phi \,d\nu_{K,\eta} = \frac{K}{K-1}\int_1^\infty \phi(\frac1K t, t)\frac{dt}{t^{p+\eta+1}}\,.
$$
It is a laminate supported by $L_K: Y=KX$. And
introduce $\nu_{1/K,\eta}$:
$$
\int_{\R^2} \phi \,d\nu_{1/K,\eta} = \frac{K}{K-1}\int_1^\infty \phi( t, \frac1K t)\frac{dt}{t^{p+\eta+1}}\,.
$$
It is a laminate supported by $L_{1/K}: Y=\frac1K X$.
Now \eqref{one}
can be rewritten as 
\begin{equation}
\label{two}
f(1,1) \ge \int f\,(d\nu_{K,\eta}+ d\nu_{1/K,\eta})
\end{equation}
If all our concavity in getting \eqref{two} become linearities then we have equality in\eqref{two}. So
$d\nu_{K,\eta}+ d\nu_{1/K,\eta}$ is a laminate with baricenter $(1,1)$.

Consider a new laminate, now with baricenter $(0,0)$:

$$
\mu_{K,\eta} = \frac14(d\nu_{K,\eta}+ d\nu_{1/K,\eta}) +\frac14\delta_{(-1,1)} +\frac12\delta_{(0,1)}\,.
$$

Test it on 
$$
\phi_1(X,Y)= |X+Y|^p,\,\phi_2(X,Y)= |X-Y|^p\,.
$$
Then
\begin{equation}
\label{ratio}
\frac{\int\phi_1 d\mu_{K,\eta}}{\int\phi_2 d\mu_{K,\eta}}=\frac{\frac14 K((K+1)^p + (K+1)^p/K^p) \eta^{-1} +\frac12 (K-1)}{\frac14 K((K-1)^p + (K-1)^p/K^p) \eta^{-1} +\frac12 (K-1)+\frac14 2^p (K-1)}\,.
\end{equation}

Choosing $\eta>0$ very small we get

\begin{equation}
\label{ratio1}
\frac{\int\phi_1 d\mu_{K,\eta}}{\int\phi_2 d\mu_{K,\eta}} \ge \bigg(\frac{K+1}{K-1}\bigg)^p -C\eta\,.
\end{equation}

Notice that we can consider a bit different laminate than $\mu_{K,\eta}$, Namely let us push forward $\mu_{K,\eta}$ by the map $X\rightarrow X, Y\rightarrow -Y$. The new measure is called 
$\sigma_{K,\eta}$. Then \eqref{ratio1} transforms to

\begin{equation}
\label{ratio2}
\frac{\int\phi_2 d\sigma_{K,\eta}}{\int\phi_1 d\sigma_{K,\eta}} \ge \bigg(\frac{K+1}{K-1}\bigg)^p -C\eta\,.
\end{equation}

Now we use Theorem \ref{JB}. It implies that there exist smooth functions with compact support on the plane such that
\begin{equation}
\label{ratio3}
\frac{\int |u_{xx}-u_{yy}|^p\, d m_2}{\int |u_{xx}+u_{yy}|^p\,d m_2} \ge \bigg(\frac{K+1}{K-1}\bigg)^p -C\eta\,.
\end{equation}

Notice that $K$ depends on $\eta$ (see \eqref{pks0}) but
$$
\frac{K+1}{K-1}\rightarrow p-1, \,\eta\rightarrow 0\,.
$$
Thus from \eqref{ratio3} we get the estimate 
\begin{equation}
\label{belowoper1}
\|R_1^2-R_2^2\|_p \ge p-1\,,
\end{equation}
if $p\ge 2$.

\vspace{.2in}

This argument can be applied to some other interesting singular operators. Constant $p^*-1$ can be described as the {\bf smallest} constant $c=c_p$ such that
the function
$$
h_c(X,Y)=|Y+X|^p- c^p |Y-X|^p
$$
has a bi-concave majorant.

\noindent{\bf Definition.} Let us call $\varphi_p(X,Y)$ the smallest bi-concave majorant of $h_c(X,Y)=|Y+X|^p- c^p |Y-X|^p$ for the smallest (as we know) possible $c=c_p=p^*-1$.

\bigskip

We will recall a formula for $\varphi_p$ in the next section.

\bigskip

Now let us consider  {\it a different}   family (it is a perturbation of $h_c$):
$$
h_{c,\tau}:=|((Y+X)^2+\tau^2 (X-Y)^2)^{1/2}|^p- c^p |Y-X|^p\,.
$$

Here is a result proved in \cite{BJV}.

\begin{theorem}
\label{bjv1}
For sufficiently small universal $\tau_0>0$, any $p\in (1, \infty)$, and any $\tau\in [-\tau_0, \tau_0]$, the smallest $c$ for which there exists a bi-concave majorant of $h_{c, \tau}$ is
$c_p(\tau)=((p^*-1)^2 +\tau^2)^{1/2}$.
\end{theorem}

Using the same considerations with laminates as above (especially Theorem \ref{JB}) we can prove the following estimate from below for ``quantum'' linear combination of secon order Riesz transforms:

\begin{theorem}
\label{bjv2}
For sufficiently small $\tau$ and any small positive $\epsilon$ one can find $g\in L^p(m_2)$ such that
$$
\|(|(R_1^2-R_2^2)g|^2 +\tau^2|(R_1^2+R_2^2)g|^2)^{1/2}\|_p \ge ((p^*-1)^2 +\tau^2)^{1/2}\|g\|_p-\epsilon\,.
$$
\end{theorem}

This gives rise to the following problem:

\noindent{\bf Problem.}
\label{bjv3}
For sufficiently small $\tau$ 
$$
\|\begin{bmatrix} R_1^2-R_2^2 \\ \tau \,I\end{bmatrix}: L^p(m_2)\rightarrow L^p(\R^2,l^2)\|= ((p^*-1)^2 +\tau^2)^{1/2}\,?
$$

The answer is affirmative, see \cite{BJV}. Notice that for $p\in (1,2)$ and large $\tau$ this is no longer true. Somewhere we have a ``phase transition" of the sharp constant. It is not clear what is the critical $\tau(p)$.

\subsection{``Explanation'' of laminates above via Burkholder's function $\varphi_p(X,Y)$ and its properties}
\label{Bu}

We introduce  coordinates $(x,y)$:
$$
Y= y+x,\, X=y-x\,.
$$
Let
$$
\gamma_p= p(1-\frac1{p^*})^{p-1}\,.
$$
In the first and second quadrants of $xy$, Burkholder's function in these coordinates is equal to (here the reader should glance at \eqref{pks0} and make $\eta=0$ in it, $s_0$ and $k$ below are as in \eqref{pks0}, but with $\eta=0$)
\begin{equation}
\label{varphi}
\varphi_p(x,y) :=\begin{cases}\gamma_p (y-(p^*-1)|x|)(|x|+y)^{p-1}\,,\,\, \text{if}\,\, \frac{y-|x|}{y+|x|}\ge s_0:= 1-\frac2{p} =\frac1k\\
y^p-(p^*-1)^p |x|^p\,,\,\, \text{if}\,\,-1\le \frac{y-|x|}{y+|x|}\le s_0\,.\end{cases}
\end{equation}
Now {\bf extend} $\varphi_p(x,y)$ to the whole plane by
$$
\varphi_p(x,y)=\varphi_p(-x,-y)\,.
$$

Burkholder proved \cite{Bu1}
\begin{theorem}
\label{bu1}
Such a function coincides with the smallest majorant of $h_c(x,y)= |y|^p-c^p|x|^p,\, c=p^*-1$ bi-convex in $X,Y$ coordinates.
For $c\in [0, p^*-1)$ there is no such bi-concave majorant of $h_c$.
\end{theorem}

\bigskip

\noindent {\bf Observation 2}. We use here both coordinates $(X,Y)$ and $(x,y)$. In the cone $X\le Y\le KX$ function $\varphi_p$ is linear along $Y=const$ segments. Similarly,
In the cone $\frac1K X\le Y\le X$ function $\varphi_p$ is linear along $X=const$ segments.

This linearity allows to calculate (we are in $(X,Y)$ now)
$$
\varphi_p(t+h, t+h) -\varphi_p(t,t)
$$
virtually without any loss if we use the $T$-shape $4$-tuple of points in $\R^2$: $( (\frac1K (t+h), t+h), (t, t+h), (t+h, t+h), (t,\frac1K t))$ as in Section \ref{lam}.

If we move one of the lines $L_K, L_{1/K}$  then two things may happen: 1) we go outside of these linearity cones, and subsequently we get strict inequality for $\varphi_p(1,1)$,
or 2) if we do not go outside of linearity cones, but then we loose 
$$
\varphi_p= h_{p^*-1}
$$ 
equality because by the definition of of $\varphi_p$ (see \eqref{varphi}) this equality holds only on the boundary of and {\it outside} of the union of linearity cones.

Notice also that on these lines $L_K, L_{1/K}$ (recall that $K=\frac{p+\eta}{p+\eta-2}$ if $p\ge 2$) we have that
$$
\varphi_p\approx h_{p^*-1}\approx 0\,,\, \varphi_p\ge h_{p^*-1}\,.
$$

For $c$ larger than $p^*-1$ we can again choose the lines where $h_c$ coincides with its bi-convex majorant, but then they will be quite negative there and integration of $h_c$ along a laminate supported on such lines cannot be almost positive as it was the case above.

\section{Stochastic Calculus and $1/2$ quasiconvexity}
\label{qc12}

We have a bijection of matrices $M=\begin{bmatrix}a,&b\\c,&d\end{bmatrix}$ onto $(z,w)$: $z=a+d +i(b-c),\, w= a-d +i(b+c)$. Notice that $2\,\det M= |z|^2-|w|^2$.

Recall that Sverak's function is the following ``simple'' object
$$
S(z,w):=\begin{cases} |z|^2-|w|^2, & |z|+|w|\le 1\\ 2|z|-1,\,\,\text{otherwise}\end{cases}
$$

Function

$$
\psi_p(z,w):=  ((p-1)|z|-|w|)(|z|+|w|)^{p-1},\, p\ge 2,
$$
can be easily obtained from $S$ using the idea of Iwaniec, see e.g. \cite{BMS}. The process is a certain averaging. Therefore, the fact  that $S$ is rank-$1$ convex implies that  $\psi_p$ is also 
rank-$1$ convex. 

To solve the Big Iwaniec problem of the previous sections it would be enough that any of these functions is quasiconvex at zero matrix. This is an outstanding and very difficult problem.

On the other hand we can formulate two problems which seem to be easier and may be readily reachable by Stochastic Calculus methods:

\noindent{\bf Problem.}
\label{S12}
Prove that $S(z, \frac12 w)$ is a quasiconvex function at zero matrix.

At least we feel that the following problem is directly reachable by methods of Stochastic Calculus:

\noindent{\bf Problem.}
\label{psi12}
Prove that $\psi_p(z, \frac12 w), p\ge 2,$ is a quasiconvex  function at zero matrix.

See interesting results in recent paper \cite{AIPS}.


\begin{thebibliography}{XXX}
\label{rf}
\bibitem{A94} Kari Astala,{\em Area distortion of quasiconformal mappings}, Acta Math., v. 173 (1994), 37--60.

\bibitem{AIS01} Kari Astala, Tadeusz Iwaniec, Eero Saksman,{\em  Beltrami operators on the plane}, Duke Math. J., v. 107 (2001), pp. 27--56.
 
 
 \bibitem{AIPS} Kari Astala, Tadeusz Iwaniec, Istv‡n Prause,  Eero Saksman, 
{\em Burkholder integrals, Morrey's problem and quasiconformal mappings, } arXiv:1012.0504 .

\bibitem{Ba} R.~Ba\~nuelos,  {\em The foundational inequalities of D. L. Burkholder and some of their ramifications}, Illinois J. of Math., volume in honor of D. L. Burkholder, to appear.

\bibitem{BaJ} R.~Ba\~nuelos, P. Janakiraman, {\em $L^p$-bounds for the Beurling--Ahlfors transform}, Trans. Amer. Math. Soc., {\bf 360},  (2008), no. 7, 3603--3612.

\bibitem{BaL} R. Banuelos, A. Lindeman, {\em A martingale study of the Beurling-Ahlfors transform in $\R^n$}, J. Funct. Analysis, {\bf 145}, (1997), 224-265.

\bibitem{BaM} R. Banuelos, P. J. M\'endez-Hernandez, {\em  Space-time Brownian motion and the Beurling--Ahlfors transform}, Indiana Univ. Math. J., {\bf 52}, (2003), no. 4, 981--990.

\bibitem{BaW} R.~Ba\~nuelos, G.~Wang, {\em Sharp inequalities for martingales with applications to the Beurling-Ahlfors and Riesz transforms}, Duke Math. J., {\bf 80},
No. 3, (1995), 575-600.


\bibitem{BMS}  A. Baernstein, S. J. Montgomery-Smith, {\em Some conjectures about integral means of $\partial f$ and $\overline{\partial} f$}, in ``Complex Analysis and Differential Equations", Proc. of Marcus Wallenberg Symposium in honor of Matts Essen, ed. Ch. Kiselman, Uppsala, Sweden 1999, 92-109.


\bibitem{Boj1} B.~V.~Bojarski, {\em Homeomorphic solutions of Beltrami systems}, Dokl. Akad. Nauk. SSSR, {\bf 102}, (1955), 661-664.

\bibitem{Boj2} B.~V.~Bojarski, {\em Generalized solutions of a system of differential equations of first order and elliptic type with discontinuous coefficients}, Mat. Sbornik, No 43, {\bf 85}, (1957), 451-503. 


\bibitem{Boj3} B.~V.~Bojarski, {\em Quasiconformal mappings and general structure properties of systems of nonlinear equations elliptic in the sense of Lavrentiev}, Symposia Mathematica (1976).


\bibitem{BojIw} B.~V.~Bojarski T.~Iwaniec, {\em Quasiconformal mappings and non-linear elliptic equations in two variables I,II}, Bull. Acad. Pol. Sci., {\bf 22}, No. 5, (1974), 473-484.

\bibitem{BBCH}{\sc   J. Bennett, N. Bez, A. Carbery, D. Hundertmark} {\it Heat flow
and Strichartz.} Personal communication.

\bibitem{BCT}{\sc   J. Bennett, A. Carbery, T. Tao} {\it On
the multilinear restriction and Kakeya conjectures,} arxiv:
math/0509262v1 12 Sep 2005.

\bibitem{BorJV1}{\sc A. Borichev, P.Janakiraman, A. Volberg}, {\em On burkholder function for ortogonal martingales and zeros of Legendre polynomials}, arXiv.

\bibitem{BorJV2} {\sc A. Borichev, P.Janakiraman, A. Volberg}, {\em Subordination by orthogonal martingales in $L^p$ and zeros of Laguerre polynomials}. Preprint, pp. 1--19, 2010.

\bibitem{BJV}{\sc N. Boros, P. Janakiraman, A. Volberg}, {\em Burkholder function for ``quantum" linear combination of Riesz transforms.} Preprint, pp. 1--28,  2010.


\bibitem{Bu1}
 {\sc D.~Burkholder}, {\em Boundary value problems and sharp estimates for the martingale transforms}, Ann. of Prob. {\bf 12} (1984), 647--702.
 
 \bibitem{Bu2}
 {\sc D.~Burkholder}, {\em An extension of classical martingale inequality}, Probability Theory and Harmonic Analysis, ed. by J.-A. Chao and W. A. Woyczynski, Marcel Dekker, 1986.

\bibitem{Bu3}
 {\sc D.~Burkholder}, {\em Sharp inequalities for martingales and stochastic
integrals}, Colloque Paul L\' evy sur les Processus Stochastiques
(Palaiseau, 1987), Ast\' erisque No. 157-158 (1988), 75--94.

\bibitem{Bu4}
 {\sc D.~Burkholder}, {\em Differential subordination of harmonic functions and martingales}, (El Escorial 1987)
 Lecture Notes in Math., {\bf 1384} (1989), 1--23.
 
 \bibitem{Bu5}
 {\sc D.~Burkholder}, {\em Explorations of martingale theory and its applications}, Lecture Notes in Math. {\bf 1464} (1991), 1--66.
 
 
 \bibitem{Bu6}
 {\sc D.~Burkholder}, {\em Strong differential subordination and stochastic integration}, Ann. of Prob. {\bf 22} (1994), 995--1025.

\bibitem{Bu7}
 {\sc D.~Burkholder}, {\em A proof of the Peczynski's conjecture for the Haar
system}, Studia MAth., {\bf 91} (1988), 79--83.

\bibitem{Ch} Michael Christ, A $T(b)$ theorem with remarks on analytic capacity and the Cauchy integral, Colloq. Math. {\bf 60/61} (1990), no. 2, 601-628.

\bibitem{CUMPbook} D. Cruz-Uribe, J. M.  Martell, C. PŽrez, ``Weights, Extrapolation and the Theory of Rubio de Francia",
Operator Theory: Advances and Applications, Vol. 215, Springer, 2011.


\bibitem{Dac} {\sc B. Dacorogna}, {\em Some recent results on polyconvex, quasiconvex and rank one
convex functions}, Adv. Math. Appl. Sci., World Science Publ.,
1994, pp. 169-176.

\bibitem{Dac1} {\sc B. Dacorogna}, Direct Methods in the Calculus of
Variations, Springer, 1989.




\bibitem{DV1} \textsc{O. Dragicevic, A.~Volberg}  {\em Sharp estimates of the
Ahlfors-Beurling operator via averaging of Martingale transform,}
Michigan Math. J. {\bf 51} (2003), 415-435.


\bibitem{DV2}{\sc O.~Dragicevic, A.~Volberg}, {\em Bellman function, Littlewood--Paley
estimates, and asymptotics of the Ahlfors--Beurling operator in $L^p(\mathbb{C})$,
$p\rightarrow \infty$},
Indiana Univ. Math. J. {\bf 54} (2005), no. 4, 971--995.

\bibitem{DV3}{\sc O. Dragicevic, A. Volberg}, {\em
Bellman function and dimensionless estimates of classical and
Ornstein-Uhlenbeck Riesz transforms.} J. of Oper. Theory, {\bf 56}
(2006) No. 1, pp. 167-198.



\bibitem{DPV}{\sc O.~Dragicevic, S. Petermichl,
A.~Volberg},{\em A rotation method which gives linear
$L^p$-estimates for powers of the Ahlfors-Beurling operator.}
Journal des Math\'ematiques Pures et Applique\'es, {\bf 86}, No. 6
(2006), 492-509.

\bibitem{DTV}{\sc O. Dragicevic, S. Treil,  A. Volberg},
{\em A lemma about $3$ quadratic forms}, arXiv:0710.3249. To appear in Intern. Math. Research Notices.



\bibitem{GMSS} {\sc S. Geiss, S. Montgomery-Smith, E.
Saksman},{ On singular integral and martingale transforms,} arxiv:
math. CA/0701516v1 18 Jane 2007.




\bibitem{StB} S.~Buckley {\em Estimates for operator norms on weighted spaces and reverse Jensen inequalities}, Trans. Amer. Math. Soc., {\bf 340}, (1993), 253-273.

\bibitem{D} Javier Duoandikoetxea,  Extrapolation of weights revisited:new proofs and sharp bounds. J. Funct. Anal., v. 260 (2011), 1886--1901.

\bibitem{DGPP} Oliver Dragicevic, Loukas Grafakos, Cristina Pereyra, Stefanie Petermichl,  Extrapolation and  sharp norm estimates for classical operators on weighted Lebesgue spaces, Publ. Mat., v. 49 ((2005), 73--91.

\bibitem{FKP} R.~Fefferman, C.~Kenig, J.~Pipher {\em The theory of weights and the Dirichlet problem for elliptic equations}, Annals of Math., {\bf 134}, (1991), 65-124.


\bibitem{GCRF} J.~Garcia-Cuerva, J.~Rubio de Francia {\em Weighted Norm Inequalities And Related Topics},
North-Holland Mathematics Studies, 116, North-Holland, Amsterdam-New York-Oxford, 1985.

\bibitem{G1} F.~W.~Gehring, {\em Open problems}, Proceedings of  Rumanian-Finnish Seminar on Teichmuller Spaces and Quasiconformal Mappings, 1969, page 306.

\bibitem{G2} F.~W.~Gehring, {\em The $L^p$--integrability of the partial derivatives of a quasiconformal mapping}, Acta Math., {\bf 130} (1973), 265-277.

\bibitem{G3} F.~W.~Gehring, {\em Topics in quasiconformal mappings}, Proceedings of the ICM 1986,
Berkeley, 62-80.

\bibitem{GR} F.~W.~Gehring, E.~Reich {\em Area distortion under quasiconformal mappings}, Ann. Acad. Sci. Fenn. Ser AI, {\bf 388}, (1966), 1-15.

\bibitem{GS} I. I. Gihkman, A. V.  Skhorohod {\em The theory of Stochastic Processes, I, II}, Springer-Verlag, 1974.


\bibitem{Iw1} T.~Iwaniec {\em Extremal inequalities in Sobolev spaces and quasiconformal mappings}, Z. Anal. Anwendungen, {\bf 1}, (1982), 1-16.

\bibitem{Kr} N. Krylov, Optimal Control of Stochastic Processes. Springer-Verlag, 1980.

\bibitem{KV} S. Konyagin, A. Volberg,  On measures with the doubling condition. Izv. Akad. Nauk SSSR Ser. Mat. 51 (1987), no. 3, 666--675; translation in Math. USSR-Izv. 30 (1988), no. 3, 629Ð638.

\bibitem{NV} F. Nazarov, A. Volberg, {\em Heat extension of the Ahlfors-Beurling operator and estimates of its norms}, Algebra i Analiz, {\bf 15}, (2003), no. 4, 142--158, translated in St. Petersburg Math. J.

\bibitem{NTV99} Nazarov, F.; Treil, S.; Volberg, A. The Bellman functions and two-weight inequalities for Haar multipliers. J. Amer. Math. Soc. 12 (1999), no. 4, 909Ð928.



 \bibitem{PSW}  S. Petermichl, L. Slavin, B. D. Wick, New estimates for the Beurling-Ahlfors operator on differential forms
    . arXiv:0901.0345
    
\bibitem{SVo} L. Slavin, A. Volberg, Bellman Function and the $H^1-BMO$ Duality
    arXiv:0809.0322,
    Journal-ref: Contemporary Math, 428, AMS, 2007

 \bibitem{SVa} L. Slavin, V. Vasyunin, Sharp results in the integral-form John--Nirenberg inequality.
    arXiv:0709.4332 

\bibitem{SSV} A. Stokolos, L. Slavin, V. Vasyunin, Bellman function for maximal operator.

\bibitem{St} E. Stein, Singular Integrals and Differentiability Properties of Functions.

\bibitem{WiP} S. Petermichl, J. Wittwer, A sharp estimate for the weighted Hilbert transform via Bellman functions. Michigan Math. J. 50 (2002), no. 1, 71Ð87.

\bibitem{WiP1} Petermichl, S.; Wittwer, J. Heating of the Beurling operator: sufficient conditions for the two-weight case. Studia Math. 186 (2008), no. 3, 203Ð217.

\bibitem{PV} S. Petermichl, A. Volberg, Heating of the Ahlfors--Beurling operator: weakly quasiregular maps on the plane are quasiregular. Duke Math. J., v. 112 (2002), 281--305.

\bibitem{VaVo} V. Vasyunin, A. Volberg,  
    Burkholder's function via Monge--Ampère equation, arXiv:1006.2633.
    
    \bibitem{VaVoMA} V. Vasyunin, A. Volberg,  Monge-Ampre equation and Bellman optimization of Carleson embedding theorems. Linear and complex analysis, 195Ð238, Amer. Math. Soc. Transl. Ser. 2, 226, Amer. Math. Soc., Providence, RI, 2009. 
    
    \bibitem{VaVoNotes}  V. Vasyunin, A. Volberg, Notes on Bellman functions in Harmonic Analysis, at blog sashavolberg.wordpress.com
    
    \bibitem{We} A. D. Wentzel {\em Introduction to Stochastic Integrals}, Moscow, Nauka, 1972.


\bibitem{Wi} J. Wittwer, A sharp estimate on the norm of the martingale transform. Math. Res. Lett. 7 (2000), no. 1, 1Ð12. 

\end{thebibliography}
\end{document}